\documentclass[onefignum,onetabnum,reqno]{siamart190516}

\usepackage{amsfonts}
\usepackage{graphicx,epstopdf}
\usepackage{mathrsfs}
\usepackage{bm}
\usepackage{color}
\usepackage{hyperref} 
\usepackage{xcolor}

\usepackage{caption, subcaption}%
\usepackage{enumitem}   

\usepackage{geometry}
\ifpdf
  \DeclareGraphicsExtensions{.eps,.pdf,.png,.jpg}
\else
  \DeclareGraphicsExtensions{.eps}
\fi

\newsiamremark{remark}{Remark}
\newsiamremark{example}{Example}
\newsiamremark{hypothesis}{Hypothesis}
\crefname{hypothesis}{Hypothesis}{Hypotheses}
\newsiamthm{claim}{Claim}
\newsiamthm{assum}{Assumption}

\headers{Geometric Quasilinearization Framework}{Kailiang Wu and Chi-Wang Shu}

\title{Geometric Quasilinearization Framework for Analysis and Design of 
	Bound-Preserving Schemes}

\author{Kailiang Wu\thanks{Department of Mathematics, Southern University of Science and Technology, Shenzhen, Guangdong 518055, China ({\tt wukl@sustech.edu.cn}). K.~Wu is supported in part by NSFC grant 12171227.}
\and Chi-Wang Shu\thanks{Division of Applied Mathematics, Brown University, Providence, RI 02912, USA
  ({\tt chi-wang\_shu@brown.edu}). C.-W.~Shu is supported in part by NSF grant DMS-2010107 and AFOSR grant FA9550-20-1-0055.}
}

\usepackage{amsopn}

\DeclareMathAlphabet\mathbfcal{OMS}{cmsy}{b}{n}

\usepackage{stmaryrd} 
\def\jump#1{\llbracket #1 \rrbracket }

\begin{document}

\maketitle

\begin{abstract}
Solutions to many partial differential equations satisfy certain bounds or constraints. For example, the density and pressure are positive for equations of fluid dynamics, and in the relativistic case the fluid velocity is upper bounded by the speed of light, etc. As widely realized, it is crucial to develop bound-preserving numerical methods that preserve such intrinsic constraints. Exploring provably bound-preserving schemes has attracted much attention and is actively studied in recent years. This is however still a challenging task for many systems especially those involving nonlinear constraints. 

Based on some key insights from geometry, we systematically propose an innovative and general framework, referred to as geometric quasilinearization (GQL), which paves a new effective way for studying bound-preserving problems with nonlinear constraints. The essential idea of GQL is to {\em equivalently} transfer all nonlinear constraints into {\em linear} ones, through properly introducing some free auxiliary variables. We establish the fundamental principle and general theory of GQL via the geometric properties of convex regions, and propose three simple effective methods for constructing GQL. We apply the GQL approach to a variety of partial differential equations, and demonstrate its effectiveness and remarkable advantages for studying bound-preserving schemes, by diverse challenging examples and applications which cannot be easily handled by direct or traditional approaches.

\end{abstract}



\section{Introduction}
Solutions to many partial differential equations (PDEs) satisfy certain algebraic constraints, which are usually derived from some (physical) bound principles, for example,  
the positivity of density and pressure. 
Consider such time-dependent PDE systems in a general form 
\begin{equation}\label{eq:gPDE}
		\partial_t {\bf u} +  \mathbfcal{L} ({\bf u}) = {\bf 0}, ~~~
		\qquad {\bf u}( {\bm x}, 0) = {\bf u}_0 ( {\bm x} ), 
\end{equation} 
where $\mathbfcal{L}$ denotes 
the differential operator associated with the spatial coordinates ${\bm x}$, 
and suppose the system \cref{eq:gPDE} is defined in a bounded domain with 
 suitable boundary conditions.  
An important class of such systems, which we are particularly interested in, are the hyperbolic conservation laws: 
\begin{equation}\label{eq:hPDE}
		\partial_t {\bf u} +  \nabla \cdot {\bf f} ( {\bf u} ) = {\bf 0}, 
		\qquad  {\bf u}( {\bm x}, 0) = {\bf u}_0 ( {\bm x} ), 
\end{equation}
and other related hyperbolic or convection dominated equations. 

Assume that the algebraic constraints (bound principles)  
can be expressed by either the positivity or the non-negativity of 
several (linear or nonlinear) functions of ${\bf u}$ as  
\begin{equation}\label{eq:GenPC}
g_i ({\bf u}) >0~~\forall i \in \mathbb I, \qquad g_i({\bf u})\ge 0~~\forall i \in \widehat {\mathbb I}, 	
\end{equation}
where $\mathbb I \cup \widehat {\mathbb I} =\{ 1,\dots,I \} $ with the positive integer $I$ denoting the total number of the constraints. 
In other words, the evolved variables ${\bf u}=(u_1,\dots,u_N)^\top$ belong to  
the admissible state set: 
\begin{equation}\label{eq:ASS-G}
	G = \left \{ {\bf u}  \in \mathbb R^N:~g_i ({\bf u}) >0~~\forall i \in \mathbb I, \quad g_i({\bf u})\ge 0~~\forall i \in \widehat {\mathbb I} \right \}. 
\end{equation}
Throughout this paper, we assume 
$G$ is convex, which is valid for many physical systems (several typical examples will be given in \cref{sec:Examples}). 
It is worth noting that the functions $\{g_i ({\bf u}), 1\le i \le I\}$ are {\em not} necessarily concave (and {\em not} required to be concave in this paper). 
Moreover, we assume that $G$ is an {\it invariant region}  
for the exact solution of the system \cref{eq:gPDE}, namely, 
\begin{itemize}
	\item  If ${\bf u} ( {\bm x}, 0 ) \in G $ for all $\bm x$, then 
	${\bf u} ( {\bm x}, t ) \in G $ for all ${\bm x}$ and $t>0$. 
\end{itemize}

A basic goal 
behind the design of numerical methods solving \cref{eq:gPDE}  
is that they can inherit 
as much as possible  
the intrinsic properties of the system \cref{eq:gPDE}.  
The constraints \cref{eq:GenPC} 
and the associated invariant region $G$  
 carry important properties of the exact solution. 
It is natural and meaningful to explore bound-preserving schemes   
that keep the numerical solutions within the region $G$:  
\begin{itemize}
	\item If ${\bf u}_h ( \cdot , t_0 ) \in G $, then 
	${\bf u}_h ( \cdot , t_n ) \in G $ for all $n \in \mathbb N$, 
\end{itemize}
where ${\bf u}_h ( \cdot, t_n ) $ denotes the numerical solutions at $n$th time level. 
In fact, preserving such constraints is not only necessary for physical significance, but also very crucial for theoretical analysis and  numerical stability. 
If any of the intrinsic physical constraints \cref{eq:GenPC} are violated numerically, the PDE system \cref{eq:gPDE} and its discrete  equations may become ill-posed 
outside the physical regimes. 
For example, when negative density and/or pressure 
are produced in numerically solving the compressible Euler equations, 
the key hyperbolicity of the system would be lost. 
As a result,  
failure to preserve such physically relevant constraints may cause serious numerical problems, for example, nonlinear instability, nonphysical solutions or phenomena, blowups of the code, etc. 
Therefore, it is significant and 
highly desirable to develop bound-preserving schemes.



In the past decades, the exploration of 
bound-preserving 
high-order numerical methods has attracted extensive attention and is actively studied, especially for hyperbolic and convection dominated equations  (e.g.~\cite{perthame1996positivity,zhang2010,zhang2010b,zhang2011,zhang2012maximum,zhang2013maximum,Xu2014,Hu2013,Wu2017a,WuShu2020NumMath}), 
and recently for some other types of time-dependent PDEs (e.g.~\cite{ShenXu2020,cheng2020global,DuJuLiQiao2021,JU2021110405,li2021stabilized}). 
For example, a general framework was established in \cite{zhang2010,zhang2010b} 
for constructing bound-preserving high-order finite volume and discontinuous Galerkin schemes for scalar conservation laws and compressible Euler equations. 
A key step in this framework is to look for high-order schemes that have  
a provable ``weak'' bound-preserving property keeping the 
cell averages of the numerical solutions in the region $G$. 
Once such a property is proven, a simple scaling 
limiter can be used to enforce the constraints  for the numerical solutions at any specified points \cite{zhang2010,zhang2010b,zhang2012maximum}. 
The idea of this methodology has been applied to many other hyperbolic or convection dominated systems; see, for example, \cite{xing2010positivity,zhang2011,cheng2012positivity,rossmanith2011positivity,cheng2013positivity,cheng2014positivity,QinShu2016,yuan2016high,ZHANG2017301,WuTang2017ApJS,JIANG2018,du2019high,wu2021uniformly}. 
Another bound-preserving framework \cite{Xu2014,Hu2013,liang2014parametrized} 
is built on  
flux-correction limiters, which modify any high-order numerical fluxes to enforce 
the constraints by combining a provably bound-preserving (lower-order) numerical flux as the building block. 
This approach has also been applied to various physical systems (cf.~\cite{christlieb2015high,christlieb2015positivity,xiong2015high,WuTang2015,xiong2016parametrized,Wu2017}). 
Recently, continuous finite element approximations 
with convex limiting were developed in \cite{guermond2016invariant,guermond2017invariant,guermond2018second} 
to preserve 
invariant regions for hyperbolic equations. 
Thorough reviews on bound-preserving efforts 
can be found in 
the survey articles 
\cite{xu2017bound,Shu2018}. 


Yet, due to the lack of a general theory, how to rigorously 
analyze or prove whether a numerical scheme is genuinely 
bound-preserving remains a challenging task. 
Despite the success of the limiter-based frameworks (cf.~\cite{zhang2010,zhang2010b,Xu2014,Hu2013}) in constructing high-order bound-preserving schemes, 
the validity of those limiters is actually based on 
some (weak or lower-order) bound-preserving properties of 
the cell-average schemes and/or of the numerical fluxes as the key building blocks. 
Proving such properties is therefore necessary, but often very 
difficult \cite{Shu2018,Wu2017a,WuShu2020NumMath}. 
To illustrate the challenges, we suppose that a numerical scheme for \cref{eq:gPDE} may be written as 
\begin{equation}\label{eq:Ascheme}
{\bf u}_{j}^{n+1} = {\mathbfcal {E}}_h 
( {\bf u}_{j-k}^{n}, {\bf u}_{j-k+1}^{n}, \dots, {\bf u}_{j}^{n}, \dots 
 {\bf u}_{j+s-1}^{n}, {\bf u}_{j+s}^{n} ),
\end{equation}
where ${\mathbfcal {E}}_h$ is the discretization operator, 
the superscripts on $\bf u$ denote the time levels, and the subscripts on $\bf u$ 
indicate the indexes of the spatial grid or nodal points.  
The bound-preserving problem for the scheme \cref{eq:Ascheme} can boil down to 
answer
\begin{equation*}
	\mbox{ whether } ~~ {\bf u}_{j}^{n} \in G~~ \forall j 
	~~\mbox{ implies  } ~~ {\bf u}_{j}^{n+1} \in G~~ \forall j ~ ?  
\end{equation*}
In essence, it is 
to explore whether or not the range of the high-dimensional function ${\mathbfcal {E}}_h$ is always contained in $G$:   
${\mathbfcal {E}}_h ( G^{s+k+1} ) \subseteq G.$ 
For some scalar PDEs with linear constraints, for instance, the scalar conservation laws with the constraints linearly defined by maximum principle, 
a general approach for 
bound-preserving analysis and design is to exploit certain 
monotonicity in schemes; see, e.g.,  \cite{zhang2010,du2017nonlocal,li2019monotonicity}. 
Yet, for PDE systems especially with nonlinear constraints, 
there is no unified tool like monotonicity, so that 
direct and complicated algebraic verification usually 
has to be performed for each constraint   
case-by-case for different schemes and different PDEs; see, e.g., \cite{zhang2010b,olbrant2012realizability,WuTang2015,QinShu2016,ZHANG2017301,meena2017positivity,WuShu2018}. 
Therefore, the design and analysis of 
 bound-preserving schemes involving nonlinear constraints are 
 highly nontrivial, 
 even for first-order schemes; cf.~\cite{tao1999gas,batten1997choice,tang2000positivity,perthame1992second,khobalatte1994maximum,ling2019physical,meena2020positivity,WuTangM3AS,Wu2017a}.

Nonlinear constraints widely exist in many physical PDE systems; see several representative examples in \cref{sec:Examples}. 
For instance, the physical constraints for solutions of the special relativistic magnetohydrodynamic (MHD) equations \eqref{eq:RMHD} 
 include:     
the positivity of density $D$ and thermal pressure $p$, and  
the upper bound of fluid velocity field $\bm v$ by the speed of light $c$, namely,  
\begin{equation}\label{eq:RMHD-constraints}
	D >0,\qquad p({\bf u})>0,\qquad c-\| {\bm v} ({\bf u}) \| > 0, 
\end{equation} 
where the evolved variables ${\bf u} = (D,{\bm m},{\bf B},E)^\top$ with 
the momentum vector ${\bm m} \in \mathbb R^3$,  the magnetic field ${\bf B} \in \mathbb R^3$, and the total energy $E$; see \cref{ex:RMHD} and \cite{WuShu2020NumMath} for more details. {\em The second and third constraints in \cref{eq:RMHD-constraints}
are highly nonlinear with respect to ${\bf u}$, because $p({\bf u})$ and ${\bm v}({\bf u})$ 
cannot be explicitly formulated in terms of 
${\bf u}$.} These  
implicit functions $p({\bf u})$ and ${\bm v}({\bf u})$ are often  
expressed via another implicit function $ \hat \phi ({\bf u})$ as   
\begin{equation}\label{getprimformU}
	p( {\bf u}) = \frac{{\Gamma  - 1}}{{\Gamma \varUpsilon_{ \bf u}^2( \hat \phi )}}\Big( {{ \hat \phi } - D \varUpsilon_{ \bf u} ( \hat \phi )} \Big), \quad 
	{\bm v}({\bf u}) = \left( { {\bm m} +  ( {\bm m} \cdot {\bf B} ) {\bf B}} /\hat \phi \right)/( \hat \phi+ {| {\bf B} |^2}),  
\end{equation}
where $ \hat \phi = \hat \phi ({\bf u})$ is implicitly defined by the positive root of the nonlinear function 
$
F(\phi; {\bf u}) :=  \phi  - E 
	+ {\left\| {\bf B} \right\|^2}   
	- \frac{1}{2}\left( { \frac{{{{( {\bm m} \cdot {\bf B} )}^2}}}{{{ \phi ^2}}} + \frac{{{{\left\|  {\bf B} \right\|}^2}}}{{{{\varUpsilon^2_{ \bf u}( \phi )}}}}   } \right)   
	 + \frac{{\Gamma  - 1}}{\Gamma }\left( {  \frac{D}{ 
			\varUpsilon_{ \bf u}( \phi )}  - \frac{ \phi }{{{{\varUpsilon^2_{ \bf u}( \phi )}}}}   } \right)
$,  the constant $\Gamma$ is the ratio of specific heats, and  
$
\varUpsilon_{ \bf u}( \phi ) := \left( \frac{ { \phi^2}{{( \phi + {\| {\bf B} \|^2})}^2} - \left[ { \phi^2}{ \|{\bm m} \|^2} + (2 \phi +{ \| {\bf B} \|^2})  {{( {\bm m} \cdot {\bf B} )}^2} \right] } {  { \phi^{2}} {{( \phi + {\| {\bf B} \|^2})}^{2}} }
\right)^{ - \frac12 }.
$
If we substitute a scheme \cref{eq:Ascheme} into the implicit functions $p({\bf u})$ and ${\bm v}({\bf u})$, then evaluating these implicit functions and analytically verifying the nonlinear constraints in \cref{eq:RMHD-constraints} for the scheme \cref{eq:Ascheme} are indeed very complicated and difficult (if not impossible).

In this paper we discover that, through properly introducing some extra auxiliary variables {\em independent} of the system variables ${\bf u}$, 
 nonlinear constraints can be {\em equivalently represented} 
by using only {\em linear} constraints, if 
the region $G$ is convex.    
For example, 
the simple nonlinear constraint  
\begin{equation}\label{eq:Ex1A}
	g({\bf u})=u_2 - {u_1^2} > 0
\end{equation}
is exactly equivalent to\footnote{The equivalence of \cref{eq:Ex1A} and \cref{eq:Ex1B} can be easily proven by $\min_{\theta_*\in \mathbb R} \varphi( {\bf u}; \theta_* ) = g({\bf u})$.} 
\begin{equation}\label{eq:Ex1B}
	\varphi( {\bf u}; \theta_* ) := u_2 - 2 u_1 \theta_* + \theta_*^2> 0\qquad \forall \theta_* \in \mathbb R,
\end{equation}
where the extra parameter $\theta_*$ is {\em independent} of $\bf u$ and called {\em free auxiliary variable} in this paper. Clearly, the new constraint \cref{eq:Ex1B} becomes  linear\footnote{This paper	broadly uses the word ``linear'', which means ``affine'' for functions or constraints with respect to $\bf u$.} 
with respect to $\bf u$.  
As we will show, such equivalent linear representation can be found  
for general nonlinear constraints, even if the constraints cannot be explicitly formulated. 
For instance, as it will be shown in \cref{thm:RMHD-GQL}, 
the constraints in \cref{eq:RMHD-constraints} 
can be 
equivalently represented as 
\begin{equation}\label{eq:RMHD22}
	D>0, \qquad {\bf u} \cdot
	{{ {\bf n}_*}} + {p^*_m} >0 \quad \forall {\bf B}_* \in \mathbb R^3~~ \forall {\bm v}_* \in \mathbb B_1({\bf 0}), 
\end{equation}
where 
$\{{\bf B}_*,{\bm v}_*\}$ are the free auxiliary variables; the vector ${{ {\bf n}_*}}$ 
and scalar ${p^*_m}$ are functions of $\{{\bf B}_*,{\bm v}_*\}$, defined by \cref{eq:RMHD:vecns2}--\cref{eq:RMHD:vecns}; $\mathbb B_1({\bf 0}) := \{ {\bm x} \in \mathbb R^3: \|{\bm x}\|<1  \}$. 
Note that the equivalent constraints in \cref{eq:RMHD22} are all linear with respect to $\bf u$. 
Benefited from such linearity,  
this novel equivalent form \cref{eq:RMHD22} has significant  
 advantages over the original form \cref{eq:RMHD-constraints} in designing and 
 analytically analyzing the bound-preserving schemes \cite{WuShu2020NumMath}. 
Several important questions naturally arise: 
Are there any intrinsic mechanisms behind such an equivalent linear representation? What is the condition for its existence? 
In general, how to find or construct it? 

The aim of this article is to 
establish a 
universal framework, 
termed as geometric quasilinearization (GQL), for constructing 
equivalent linear representations for general nonlinear constraints. 
It will be based on some key insights from geometry to understand a convex region $G$. 
The GQL framework would 
shed new light on challenging bound-preserving problems involving nonlinear constraints.   
The novelty and significance of the proposed GQL framework include: 
\begin{itemize}
	\item A distinctive innovation of GQL lies in   
		a novel geometric point of view on the nonlinear algebraic constraints and the convex invariant region $G$.
	\item 
		Through introducing some extra free auxiliary variables, 
		this framework provides a simple yet unified approach to 
		derive the equivalent linear representation (termed as GQL representation)
		for a general convex region $G$. 
	\item GQL offers  
	a highly effective approach for bound-preserving analysis and design for problems with nonlinear constraints.
	\item The GQL representations have simple formulations and are very easy to construct. We will propose three effective methods for constructing GQL.
\end{itemize}
The idea of GQL is motivated from a series of our recent works on 
seeking bound-preserving 
schemes for 
the (single-component) compressible MHD systems \cite{Wu2017a,WuTangM3AS,WuShu2018,WuShu2019,WuShu2020NumMath}. 
For the invariant region of the ideal MHD equations, its equivalent linear representation was first established by technical algebraic manipulations \cite{Wu2017a}. Such a  representation played crucial roles in obtaining the first rigorous positivity-preserving analysis of numerical schemes for the ideal MHD system \cite{Wu2017a}, and also in designing  
the provably positivity-preserving multidimensional MHD schemes  \cite{WuShu2018,WuShu2019,WuShu2020NumMath}.  
The success of the GQL idea in these special cases  
 strongly encourages us to explore its essential mechanisms and universal framework for general systems. 

Our efforts in this article include:  
\begin{itemize}
	\item We
	interpret, from a geometric viewpoint,   
	the fundamental principle behind the GQL representations for general nonlinear algebraic constraints. 
	\item We establish the universal GQL framework and its  mathematical theory. 
	\item We propose three simple effective methods for constructing GQL representations using extra free auxiliary variables in exchange for linearity. As examples, 
	the GQL representations are derived for the invariant regions of various physical systems. 
	\item We illustrate the GQL methodology and related techniques for nonlinear bound-preserving analysis and design, demonstrating its 
	effectiveness and remarkable advantages, by diverse challenging applications which cannot be easily handled by direct or traditional approaches.  
\end{itemize}
We emphasize that GQL 
has no restriction on the specific forms of the equations \cref{eq:gPDE}. This makes the framework 
applicable to general time-dependent PDE systems 
that possess convex invariant regions with nonlinear constraints. 


The paper is organized as follows. \Cref{sec:Examples} presents several examples of physical PDE systems 
with convex invariant regions and nonlinear constraints. 
\Cref{sec:GQL} explores the fundamental principle and general theory for the GQL framework. 
We propose in \cref{sec:Construction} three simple effective methods for constructing GQL representations, along with extensive examples. 
\Cref{sec:GQL-CP} illustrates 
the GQL approach for bound-preserving analysis. 
In \cref{sec:GQL-MMHD} we apply the GQL approach to design bound-preserving schemes for the multicomponent MHD system, 
and further demonstrate its powerful capabilities in addressing challenging bound-preserving 
problems that could not be coped with by direct or traditional approaches.  
Several experimental
results are given in \cref{sec:experiments} to verify the performance of the bound-preserving schemes developed via GQL.   
The 
conclusions follow in
\cref{sec:conclusions}. 
Throughout this paper, we will use ${\rm cl} (G)$, ${\rm int} (G) $, and 
$\partial G$ to denote the closure, the interior, and the boundary of 
a region $G$, respectively. 
We employ $\| {\bf a} \|$ to denote the 2-norm of vector $\bf a$. 
We use $ {\bf a} \cdot {\bf b} $ to denote 
the inner product of two vectors ${\bf a}$ and $\bf b$, and $ {\bf a} \otimes {\bf b} $ to denote
the outer product, i.e., in index notation, $({\bf a} \otimes {\bf b})_{ij}=a_i b_j$.

\section{Examples of PDE systems with nonlinear constraints}\label{sec:Examples}

In this section, we present several examples 
of physical PDE systems involving nonlinear algebraic constraints.   
For convenience, the ideal equation of state $p=(\Gamma -1) \rho e$ is used to close the systems in \cref{ex:Euler,ex:NS,ex:RHD,ex:IMHD,ex:RMHD}, with $p$ denoting the thermal pressure, $\rho$ the (rest-mass) density, $e$ the specific internal energy, and the constant $\Gamma>1$ denoting the ratio of specific heats. 
For the relativistic models in \cref{ex:M1,ex:RHD,ex:RMHD}, normalized units are employed such that the speed of light $c=1$.

\begin{example}[Euler System]\label{ex:Euler}
Consider the 1D compressible Euler equations \cite{zhang2010b}
\begin{equation}\label{eq:1DEuler}
	\partial_t {\bf u} + \partial_x  {\bf f} ( {\bf u} ) = {\bf 0}, \qquad 
{\bf u}= 	
\begin{pmatrix}
	\rho 
	\\
	m
	\\
	E	
\end{pmatrix}, 
\qquad 
{\bf f} ({\bf u})= 	
\begin{pmatrix}
	m 
	\\
	mv + p
	\\
	 (E+p)v	
\end{pmatrix},
\end{equation}
where $\rho$, $m$, $v=m/\rho$, and $p$ denote the fluid density, momentum, velocity, and pressure, respectively. 
The quantity $E=\rho e+\frac12 \rho v^2$ is the total energy, with $e$ being the specific internal energy. 
For this system, the density $\rho$ and the internal energy $\rho e$ are positive, namely, $\bf u$ should stay in the region 
\begin{equation}\label{eq:EulerNS-G1}
	G = \left\{ {\bf u} = (\rho,m,E)^\top \in \mathbb R^3:~ \rho>0,~g({\bf u}) :=
E-\frac{m^2}{2\rho} >0  \right\},
\end{equation}
which is a convex invariant region of the system \cref{eq:1DEuler}. 
If we further consider Tadmor's minimum entropy principle \cite{TADMOR1986211},   $S({\bf u}) \ge S_{min}:=\min_{\bm x} S ( {\bf u}_0({\bm x}) )$, for the specific entropy $S=p \rho^{-\Gamma}$, then 
we obtain another convex invariant region
\begin{equation}\label{eq:EulerNS-G2}
	\widetilde G = \left\{ {\bf u} = (\rho,m,E)^\top \in \mathbb R^3:~\rho>0,~\widetilde g({\bf u})\ge 0  \right\} 
\end{equation}
with $$\widetilde g({\bf u}):=S({\bf u})-S_{min}=\frac{\Gamma -1}{\rho^\Gamma} \left( E-\frac{m^2}{2\rho}
\right) - S_{min}.$$ 
The readers are referred to \cite{zhang2010b,zhang2012minimum} for proofs of the convexity of $G$ and $\widetilde G$. 
Convex invariant regions for the 2D and 3D Euler systems are analogous and omitted here. 

\end{example}

\begin{example}[Navier--Stokes System]\label{ex:NS} 
	Consider the 1D dimensionless compressible Navier--Stokes equations (see, for example, \cite{ZHANG2017301}): 
\begin{equation}\label{eq:1DNS}
	\partial_t {\bf u} + \partial_x  {\bf f} ( {\bf u} ) = \frac{\eta}{ {\tt Re} } \partial_{xx}
	{\bf r}({\bf u})
	, \qquad 
	{\bf r} ({\bf u})= 	
	\begin{pmatrix}
		0 
		\\
		v
		\\
		\frac{ v^2 }2 + \frac{\Gamma}{ {\tt Pr}~\eta } e	
	\end{pmatrix},
\end{equation}
where $\{\eta, {\tt Re}, {\tt Pr} \}$ are positive constants, and the definitions of ${\bf u}$ and ${\bf f}({\bf u})$ are the same as 
 \cref{ex:Euler}. Both sets in \cref{eq:EulerNS-G1} and \cref{eq:EulerNS-G2} are also invariant regions for system \cref{eq:1DNS}. 
\end{example}

\begin{example}[M1 Model of Radiative Transfer]\label{ex:M1} 
For the solutions of the gray M1 moment system of radiative transfer (see, for example, \cite{olbrant2012realizability,berthon2007hllc}), a convex invariant region 
is 
\begin{equation}\label{eq:M1-G}
	G = \left\{ 
	{\bf u} = ( E_r, {\mathbfcal F}_r  )^\top \in \mathbb R^4:~ g ({\bf u}):= E_r - \| {\mathbfcal F}_r \| \ge 0 
	\right\}, 
\end{equation} 
where $E_r$ is the radiation energy, and ${\mathbfcal F}_r$ is the radiation energy flux.  

\end{example}

\begin{example}[Relativistic Hydrodynamic System]\label{ex:RHD}
Consider the 1D governing equations of the special relativistic hydrodynamics (RHD) \cite{WuTang2015,QinShu2016}: 
\begin{equation}\label{eq:1DRHD}
	\partial_t {\bf u} + \partial_x  {\bf f} ( {\bf u} ) = {\bf 0}, \qquad 
	{\bf u}= 	
	\begin{pmatrix}
		D 
		\\
		m
		\\
		E	
	\end{pmatrix}, 
	\qquad 
	{\bf f} ({\bf u})= 	
	\begin{pmatrix}
		D v 
		\\
		m v + p
		\\
		m	
	\end{pmatrix}
\end{equation}
with the density $D=\rho \gamma$, 
the momentum
$m=\rho h \gamma^2 v$, 
the energy $E=\rho h \gamma^2 -p$. Here,  
$\rho$,  $v$, $p$, and $\gamma = (1-v^2)^{-\frac12}$ denote the rest-mass density, velocity, pressure, and Lorentz factor, respectively. The quantity $h=1+e + p/\rho$ represents the specific enthalpy, with $e$ being the specific internal energy.  
For this system, the density and the pressure are positive, 
and the magnitude of $v$ must be smaller than the speed of light ($c=1$). These physical constraints define the invariant region 
\begin{equation}
	G = \left \{  {\bf u} \in \mathbb R^3:  D>0,~p({\bf U})>0,~1 - |v({\bf U})|>0   \right \}.
\end{equation}
It was proven in \cite{WuTang2015} that the region $G$ is convex and can be equivalently represented as 
\begin{equation}\label{eq:RHD-G1}
	G = \left \{  {\bf u} \in \mathbb R^3:~D>0,~g({\bf u}) := E-\sqrt{D^2 + m^2} > 0 \right \}.
\end{equation}  
As shown in \cite{wu2021minimum}, the minimum entropy principle  $S({\bf u}) \ge S_{min}$ also holds for the RHD system \cref{eq:1DRHD}, yielding another invariant region 
\begin{equation}\label{eq:RHD-G2}
	\widetilde G = \left\{ {\bf u} \in \mathbb R^3:~D>0,~g({\bf u}) >0,~\widetilde g({\bf u})  \ge 0  \right\},
\end{equation}
where 
$\widetilde g({\bf u}) := { p({\bf u}) }{ ( \rho({\bf u}) )^{-\Gamma} } - S_{min}$ 
is a highly nonlinear implicit function. 
In the RHD case, 
 the functions $p({\bf u})$ and $\rho({\bf u})$ cannot be explicitly expressed in terms of $\bf u$. Specifically, $p({\bf u})$ is implicitly defined by the positive root  
 of  
 the nonlinear function 
$
F(p;{\bf u}) :=	\frac{ m^2 }{E+p} 
	+ D \big( 1 - \frac{ m^2 }{ (E+p)^2 }  \big)^{\frac12} + \frac{p}{\Gamma - 1} - E, 
$ 
and then $\rho( {\bf u} ) = D \sqrt{ 1 -  m^2/{(E + p ({\bf u}))}^2  }$.

\end{example}



\begin{example}[Ten-Moment Gaussian Closure System]
In 2D, this system \cite{meena2017positivity,meena2020positivity} reads 
 \begin{align}\label{eq:Ten-Moment}
&	\partial_t {\bf u} +  \partial_x {\bf f}_1 ( {\bf u} ) + \partial_y {\bf f}_2  ( {\bf u} ) = {\bf 0}, 
	\\		\nonumber
&
	{\bf u} = 	
\begin{pmatrix}
	\rho  
	\\
	m_1 
	\\
	m_2 
	\\
	E_{11}
	\\
	E_{12}
	\\
	E_{22}
\end{pmatrix},\qquad 
{\bf f}_j ( {\bf u} )= 	
\begin{pmatrix}
	m_j 
	\\
	m_1 v_j + p_{1j}
	\\
	m_2 v_j + p_{2j}
	\\
	E_{11} v_j +p_{1j}	v_1
	\\
	E_{12} v_j + \frac12 ( p_{1j} v_2 + p_{2j} v_1 )
	\\
	E_{22} v_j + p_{2j} v_2
\end{pmatrix},\quad j=1,2.
\end{align}
Here $\rho$, ${\bm m}=(m_1,m_2)$, ${\bm v} = {\bm m}/\rho$, ${\bf E}=(E_{ij})_{1\le i,j \le 2}$, and ${\bf p}=(p_{ij})_{1\le i,j \le 2}$ are respectively 
the density, momentum vector, velocity, symmetric energy tensor, and 
symmetric anisotropic pressure tensor. 
The system \cref{eq:Ten-Moment} is closed by ${\bf p}= 2 {\bf E} - \rho {\bm v} \otimes {\bm v} $. 
For this system, the density $\rho$ is positive, and the pressure tensor $\bf p$ is  positive-definite, namely, the evolved variables $\bf u$ 
should belong to the following invariant region 
\begin{align}\label{eq:10M-G}
G &= \left\{ {\bf u}  \in \mathbb R^6:~\rho>0,~ {\bf E} - \frac{ {\bm m} \otimes {\bm m} }{2\rho}~\mbox{is positive-definite} \right\}
\\ \label{eq:10M-G11}
& = \left\{ {\bf u}  \in \mathbb R^6:~\rho>0,~ {\bm z}^\top \left( {\bf E} - \frac{ {\bm m} \otimes {\bm m} }{2\rho} \right) {\bm z} >0~~ \forall {\bm z} \in \mathbb R^2 \setminus \{ {\bf 0} \}  \right\}.
\end{align}

\end{example}

\begin{example}[Ideal MHD System]\label{ex:IMHD}
	This system \cite{Wu2017a,WuShu2018} can be written as 
\begin{equation}\label{eq:idealMHD}
	\partial_t	
	\begin{pmatrix}
		\rho  
		\\
		{\bm m}
		\\
		{\bf B}
		\\
		E	
	\end{pmatrix}
	+ 	\nabla \cdot
	\begin{pmatrix}
		{\bm m}
		\\
		{\bm m} \otimes {\bm v} - {\bf B} \otimes {\bf B}
		+ p_{tot} {\bf I}  
		\\
		{\bm v} \otimes {\bf B} - {\bf B} \otimes {\bm v}
		\\
		\left( E + p_{tot} \right) {\bm v} 
		- ( {\bm v} \cdot {\bf B} ) {\bf B}	
	\end{pmatrix} = {\bf 0}
\end{equation}	
with $\rho$ being the density, $\bm m$ the momentum vector, 
${\bm v}={\bm m}/\rho$ the velocity, 
$E=\rho e + \frac12 ( \rho \|{\bm v} \|^2 + \| {\bf B} \|^2 )$ denoting the total energy, 
$p_{tot} = p+\frac12 \| {\bf B} \|^2$ being the total pressure, 
$p$ the thermal pressure, 
and $\bf B$  the magnetic field which satisfies the extra divergence-free condition $\nabla \cdot {\bf B} = 0.$ For this system, the density $\rho$ and the internal energy $\rho e$ 
are positive, namely, $\bf u$ should stay in the invariant region 
\begin{equation}\label{eq:G-iMHD}
	G = \left\{ 
	{\bf u}=(\rho, {\bm m}, {\bf B}, E)^\top \in \mathbb R^8:~ \rho>0,~ 
	g({\bf u}) := E - 
	\frac{ \|{\bm m}\|^2 }{2\rho} -\frac{ \|{\bf B}\|^2 }2 > 0 
	\right\}.
\end{equation}

\end{example}

\begin{example}[Relativistic MHD System]\label{ex:RMHD}
This system \cite{WuShu2020NumMath} takes the form of 
 \begin{equation}\label{eq:RMHD}
			\partial_t 
		\begin{pmatrix}
			D  
			\\
			{\bm m}
			\\
			{\bf B}
			\\
			E	
		\end{pmatrix} + \nabla \cdot	
		\begin{pmatrix}
			D {\bm v} 
			\\
			{\bm m} \otimes {\bm v} - {\bf B} \otimes 
			\left( \gamma^{-2} {\bf B} + ( {\bm v} \cdot {\bf B} ) {\bm v} \right) + p_{tot} {\bf I}
			\\
			{\bm v} \otimes {\bf B} - {\bf B} \otimes {\bm v}
			\\
			{\bm m}	
		\end{pmatrix} = {\bf 0}
\end{equation}
with the mass density $D = \rho \gamma$, the momentum vector ${\bm m} = (\rho h{\gamma^2} + \|{\bf B}\|^2) {\bm v} - ( {\bm v} \cdot {\bf B} ) {\bf B}$, the energy $E=\rho h \gamma^2 - p_{tot} +\| {\bf B} \|^2$, and the magnetic field ${\bf B}$ satisfies $\nabla \cdot {\bf B} = 0$ as the ideal MHD case. The total pressure
$p_{tot}$ consists of the magnetic pressure $p_m:=\frac12 \left(\gamma^{-2} \| {\bf B} \|^2 +( {\bm v} \cdot {\bf B} )^2 \right)$ and the thermal pressure $p$. 
Analogously to \cref{ex:RHD},  
the quantities 
$\rho$,  $\bm v$, $h$, and $\gamma = (1-\| {\bm v} \|^2)^{-\frac12}$ are respectively the rest-mass density, velocity, specific enthalpy, and Lorentz factor. 
The positivity of density and pressure as well as the subluminal constraint $\| {\bm v} \| < c=1$ constitute the invariant region 
\begin{equation}\label{eq:RMHD-G1}
	G = \left \{  {\bf u} =(D, {\bm m}, {\bf B},E)^\top \in \mathbb R^8:  D>0,~p({\bf u})>0,~1 - \|{\bm v}({\bf u})\|>0   \right \}, 
\end{equation}
where $p({\bf u})$ and ${\bm v}({\bf u})$ are highly nonlinear 
and cannot be explicitly formulated, as discussed in \cref{getprimformU}. 
\end{example}

\section{Framework and theory of geometric quasilinearization}\label{sec:GQL}
This section establishes the universal GQL framework, with the geometric insights into understanding the fundamental principle behind the GQL representations.

Let $G \subset \mathbb R^N$ be an invariant region or admissible state set of a physical system. Assume that $G$ can be formulated into the general form \cref{eq:ASS-G}. 
For notational convenience, we represent $G$ as 
\begin{equation}\label{eq:ASS-G1}
	G = \left \{ {\bf u}  \in \mathbb R^N:~g_i({\bf u}) \succ 0,~1\le i \le I \right \}, 
\end{equation}
where the symbol ``$\succ$'' denotes   
 ``$>$'' if $i \in \mathbb I$, or ``$\ge$'' if $i \in \widehat {\mathbb I}$. 
 Let $G_L=\{ {\bf u} \in \mathbb R^N: g_i({\bf u}) \succ 0~ \forall i \in {\mathbb I}_L \}$ be the region  formed by all the linear constraints in $G$, i.e., the function $g_i$ is linear for $i \in {\mathbb I}_L$. If ${\mathbb I}_L = \emptyset$, then we define $G_L = \mathbb R^N$.

We consider the nontrivial case that at least one of the functions $\{ g_i ({\bf u})\}$ is nonlinear, namely, $G \subset G_L$ and $G \neq G_L$. 
The goal of our GQL methodology is to use some extra free  auxiliary variables in exchange for linearity, and more precisely, is to  equivalently represent $G$ by using only {\it linear} constraints  
with the help of free auxiliary variables.

\begin{definition}\label{def:GQL}
We say a set $G_*$ is an equivalent linear representation (termed as GQL representation) of the region $G$, if $G_* = G$ and $G_*$ takes the form 
\begin{equation}\label{eq:Gs}
	G_* = \left\{ 
	{\bf u}\in \mathbb R^N: \varphi_i( {\bf u}; {\bm \theta}_{i*} ) \succ 0~~ 
	\forall {\bm \theta}_{i*} \in \Theta_i,~ 
	1\le i \le I
	\right\}, 
\end{equation}
where the functions $\{\varphi_i\}$ are all {\bf{\em linear}} (affine) with respect to $\bf u$; the  parameters ${\bm \theta}_{i*}$ are independent of ${\bf u}$ and stand for the (possible) extra free auxiliary variables with $\Theta_i$ denoting their ranges. 
\end{definition}

Based on \cref{def:GQL}, we immediately have:  

\begin{theorem}\label{thm:minphi}
Assume that a set $G_*$ is of the form \cref{eq:Gs} with 
 $\varphi_i$ being linear with respect to $\bf u$ and satisfying  
\begin{equation}\label{key656}
\min_{ {\bm \theta}_{i*}  \in \Theta_i }
\varphi_i( {\bf u}; {\bm \theta}_{i*} )  = \lambda_i ({\bf u}) g_i ({\bf u})  
\end{equation}
with $\lambda_i ({\bf u})>0$ for all ${\bf u} \in G_L$. 
Then $G_* = G$, and $G_*$ is the GQL representation of $G$.  
\end{theorem}

\begin{remark} 
For $i \in {\mathbb I}_L$, the function $g_i({\bf u})$ is already linear, thus we can simply take  
$\varphi_i( {\bf u}; {\bm \theta}_{i*} ) = g_i({\bf u})$, without free auxiliary variable ${\bm \theta}_{i*}$ in this case. That is, all the linear constraints remain unchanged in the GQL representation. 
\end{remark}

\cref{thm:minphi} points out a way to seek the GQL representation, namely, by constructing linear functions $\{\varphi_i\}$ such that \cref{key656} holds. We have used this approach in \cite{Wu2017a} to establish the GQL representation of 
the invariant region \cref{eq:G-iMHD} for the ideal MHD equations. However, this constructive approach needs some empirical observations or trial-and-error procedures, as \cref{thm:minphi} does not provide any insight on how to find the qualified  $\{\varphi_i\}$. 
In the following, we explore a simpler yet universal 
approach from the  geometric point of view.

Given that $\{\varphi_i\}$ in \cref{eq:Gs} are all linear with respect to $\bf u$, the set $G_*$ is always convex. This means if the region $G$ has GQL representation \cref{eq:Gs}, then $G$ must also be convex. Hence we should make the following basic (minimal) assumption. 
\begin{assum}\label{assum:convex}
	The invariant region $G$ is convex, and ${\rm int}(G) \neq \emptyset$. 
\end{assum}
This basic assumption is valid for many physical systems including all those introduced in \cref{sec:Examples}. 
Again, we emphasize that the functions $\{g_i ({\bf u})\}$ are {\em not} necessarily concave. 

\subsection{A heuristic example}

Before deriving the general theory, let us look at an example to gain some insight,
which inspires 
 us to achieve the GQL framework.

\begin{figure}[tbh]
	~~\includegraphics[width=0.9\textwidth]{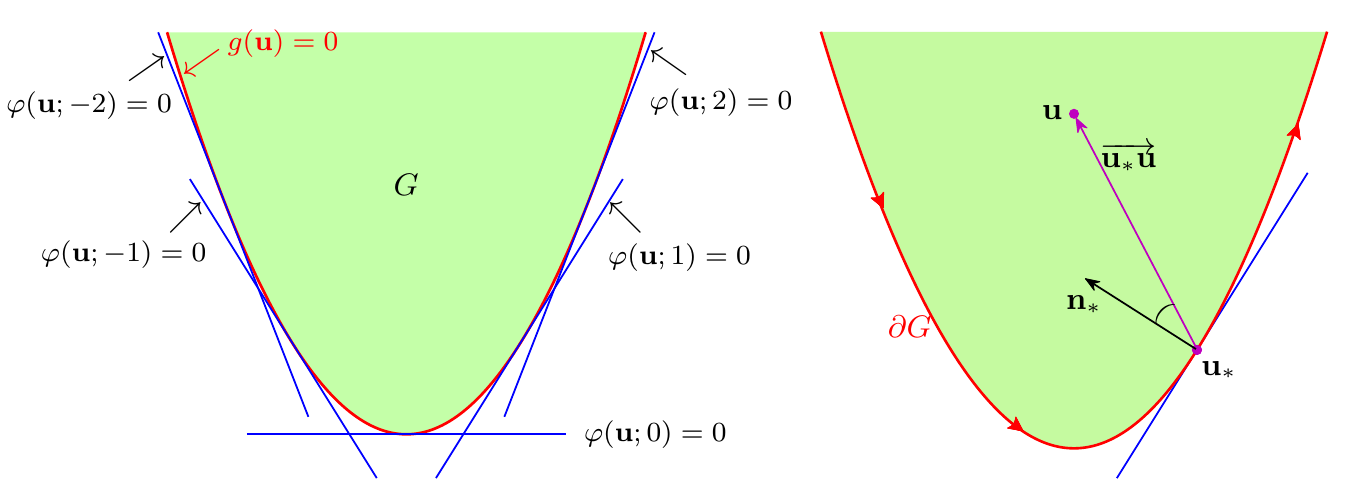} 
	\captionsetup{belowskip=-2pt}
	\caption{\small Illustrations for \cref{ex:heuristic}.}
	\label{fig:example0}
\end{figure}

\begin{example}\label{ex:heuristic}
{\em Consider the simple example mentioned in \cref{eq:Ex1A}--\cref{eq:Ex1B}, i.e., $G=\{ {\bf u}=(u_1,u_2)^\top \in \mathbb R^2:~g({\bf u})=u_2 - u_1^2 >0 \}$. According to 
\cref{thm:minphi}, the GQL representation of $G$ is    
\begin{equation}\label{eq:Hexample}
	G_* = \left \{ {\bf u}=(u_1,u_2)^\top \in \mathbb R^2:~\varphi ( {\bf u}; \theta_* ) = u_2 - 2   u_1 \theta_* + \theta_*^2> 0~~ \forall \theta_* \in \mathbb R \right \}.
\end{equation} 
As such, we gain the linearity by introducing the extra free auxiliary variable $\theta_*$. 
To understand the intrinsic mechanisms, we draw the graph of the region $G$ and its boundary curve 
$\partial G=\{ {\bf u}: g({\bf u}) =0 \}$ on the $u_1$--$u_2$ plane in \cref{fig:example0}. 
We also plot 
the graphs of $\{ {\bf u}: \varphi({\bf u};\theta_*) = 0 \}$ 
for several special values of $\theta_* \in \{ \pm 2, \pm 1, 0 \}$ in the left subfigure of \cref{fig:example0}. 
It is observed that all the lines  $\{ {\bf u}: \varphi({\bf u};\theta_*) = 0 \}$ are tangent to the  
parabolic curve $\partial G$, which exactly forms an envelope of the tangent lines.

Let 
 ${\bf u}_* = ( \theta_*, \theta_*^2 )^\top$ denote an arbitrary point on $\partial G$. One can verify that ${\bf n}_* =( -2 \theta_*, 1 )^\top$ is an inward-pointing normal vector of  $\partial G$ at ${\bf u}_*$, and 
\begin{equation*}
	\varphi ( {\bf u}; \theta_* ) =  {\bf u} \cdot {\bf n}_* - {\bf u}_* \cdot {\bf n_*} 
= \overrightarrow{{\bf u}_* {\bf u}} \cdot {\bf n}_* >0\quad \forall {\bf u} \in G. 
\end{equation*}
Imagine we are walking along the boundary $\partial G$ 
in the direction shown in the right subfigure of \cref{fig:example0}, 
then the region $G$ always lie entirely on the left side of 
the tangent lines, namely, the angle between the two vectors $\overrightarrow{{\bf u}_* {\bf u}}$ and ${\bf n}_*$ is always less than $90^\circ$ for all ${\bf u} \in G$ and all ${\bf u}_* \in \partial G$. This intuitively interprets the GQL representation \cref{eq:Hexample} from the geometric viewpoint.  

}

\end{example}


\subsection{Concepts from geometry and convex sets} 
Let us recall some concepts and results 
from theory of geometry and convex analysis \cite{leonard2015geometry,rockafellar2015convex,gruber2007convex}. 

A hyperplane in $\mathbb R^N$ is a plane of dimension $N-1$. Let ${\bf n}_* \neq {\bf 0}$ denote a normal vector of a hyperplane $H$, and let ${\bf u}_*$ be a point on $H$. 
Then $H$ can be expressed as 
$H = \{ {\bf u}\in \mathbb R^N: ( {\bf u}-{\bf u}_*) \cdot {\bf n}_*  = 0  \}$, 
and it divides $\mathbb R^N$ into two halfspaces:  
$H^+ = \{ {\bf u} \in \mathbb R^N: ( {\bf u}-{\bf u}_*) \cdot {\bf n}_*  \ge 0  \}$ 
and $H^- = \{ {\bf u} \in \mathbb R^N: ( {\bf u}-{\bf u}_*) \cdot {\bf n}_*  \le 0  \}$. 

\begin{definition}[Supporting Hyperplane and Halfspace] 
	The hyperplane $H = \{ {\bf u} \in \mathbb R^N: ( {\bf u}-{\bf u}_*) \cdot {\bf n}_* = 0  \}$ through ${\bf u}_* \in \partial G$ is  called a supporting hyperplane to $G$ 
	at ${\bf u}_*$, if $G$ lies in one of the two closed halfspaces  determined by $H$. 
	Furthermore, if the normal vector ${\bf n}_*$ points towards  
	$G$, then the closed halfspace containing $G$ is  
	$H^+ =\{ {\bf u}\in \mathbb R^N:  ( {\bf u}-{\bf u}_*) \cdot {\bf n}_*   \ge 0  \} $ and is  
	called a closed supporting halfspace to $G$. See \cref{ex:supporting}. 
\end{definition}

\begin{figure}[htbp]
	\centering
	\includegraphics[width=0.388\textwidth]{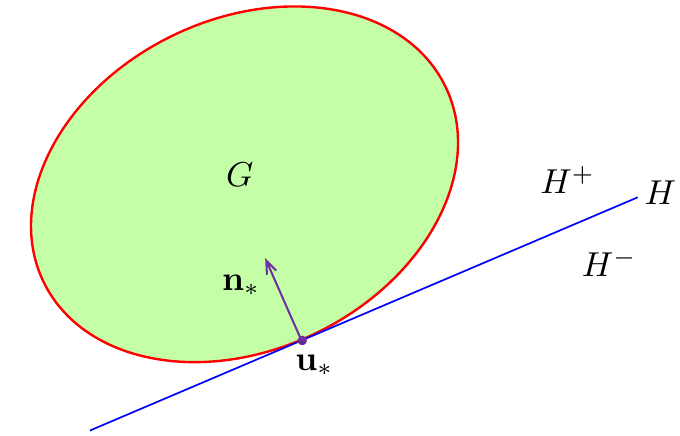}
	\captionsetup{belowskip=-2pt}
	\caption{\small Supporting hyperplane and halfspace.}
	\label{ex:supporting}
\end{figure}


\begin{theorem}[Supporting Hyperplane Theorem \cite{leonard2015geometry}]\label{lem:support}
If $G$ is a convex set and ${\rm int}(G) \neq \emptyset$, then for any ${\bf u}_* \in \partial G$, there exists a supporting hyperplane to $G$ at ${\bf u}_*$. 
\end{theorem}

\begin{remark}
If the boundary 
 $\partial G$ is smooth 
at a point ${\bf u}_*$, then the supporting hyperplane to 
$G$ at ${\bf u}_*$ is unique and coincide with the tangent \cite{rockafellar2015convex,gruber2007convex}. 
\end{remark}

\subsection{GQL framework}
We are now in the position to establish the GQL framework.

\subsubsection{A special case}
Inspired by 
\cref{ex:heuristic}, 
we first consider a special case that $G$ is either open or closed with differentiable boundary. 
The general case will be discussed in \cref{sec:general-case}.  

\begin{theorem}\label{thm:I=1}
	Suppose that \cref{assum:convex} holds, the region $G$ is either open or closed, and $\partial G$ is differentiable. Then $G$  has the following GQL representation: 
\begin{equation}\label{eq:I1}
	 G_* = 
	\Big\{ {\bf u} \in \mathbb R^N:~ ( {\bf u} - {\bf u}_*) \cdot  {\bf n}_{*}  \succ 0~~\forall {\bf u}_* \in \partial  G   \Big\},
\end{equation}
where the symbol ``$\succ$'' is taken as ``$>$'' if $G$ is open, or as ``$\ge$'' if $G$ is closed, and 
${\bf n}_{*}$ is only dependent on ${\bf u}_*$ and denotes an inward-pointing normal vector of $\partial  G$ at ${\bf u}_*$.

\end{theorem}

The proof of \cref{thm:I=1} is presented in \cref{sec:proofSpeicalCase}. 
Following the proof, one can further extend the above result to any closed convex region $G$, whose boundary is typically not everywhere smooth so that the supporting hyperplanes at each nonsmooth boundary point are not unique. Let ${\mathcal N}({\bf u}_*)$ denote the set of the inward-pointing unit normal vectors of 
all the supporting hyperplanes to $G$ at ${\bf u}_* \in \partial G$. 
Then one can prove that 
\begin{equation}\label{eq:generalQL}
	G = \left\{ {\bf u} \in \mathbb R^N:~ \left (  {\bf u} - {\bf u}_* \right) \cdot {\bf n}  \ge 0~~\forall {\bf n} \in  {\mathcal N}({\bf u}_*),~~\forall 
	{\bf u}_* \in \partial G \right \}.	
\end{equation}
This means any closed convex region is the intersection of all its closed supporting halfspaces \cite{leonard2015geometry}. 
However, the representation \cref{eq:generalQL} is {\em not} applicable to a general convex region that is neither closed nor open (e.g.~the invariant regions in \cref{eq:EulerNS-G2} and \cref{eq:RHD-G2}). Moreover, the representation \cref{eq:generalQL} requires the information of {\em all} the supporting hyperplanes at each nonsmooth boundary point, which can be difficult to explicitly formulate or verify, so that \cref{eq:generalQL} is not desirable 
for bound-preserving study. A practical GQL representation for more general regions will be derived in \cref{sec:general-case}.

\subsubsection{General case}\label{sec:general-case}
Consider a general convex region $G$ that may be {\em not necessarily} open or closed and its boundary may be not everywhere smooth. 
Note that the boundary of a convex region can be partitioned into several pieces, 
each of which can be  
{\em locally} 
represented as the graph of a convex function (with respect to a suitable supporting 
hyperplane). 
Recall that any convex function is locally Lipschitz continuous 
and twice differentiable almost everywhere, 
according to 
the classical theorems of Rademacher and Alexandrov (cf.~\cite{niculescu2006convex}).  
Based on these facts and 
for convenience, we make a considerably mild assumption 
on the convex invariant region $G$. 
We assume that the boundary of $G$ 
is piecewise $C^1$, and without loss of generality, for each $i \in \{1,\dots,I\}$, the function  $g_i({\bf u})$ 
in \cref{eq:ASS-G1} is $C^1$ at any points on 
\begin{equation*}
	{\mathcal S}_i := \partial G \cap \partial G_i,  \qquad \mbox{with} \quad  G_i :=  \left \{ {\bf u}  \in \mathbb R^N:~g_i({\bf u}) \succ 0 \right \}, 
\end{equation*}
where 
$\{ {\mathcal S}_i\}$ are $C^1$ 
hypersurfaces in $\mathbb R^N$ 
and 
constitute the smooth pieces of $\partial G$, i.e., $\partial G=\cup_{1\le i \le I} {\mathcal S}_i $. 
Notice that in general, ${\mathcal S}_i$ may not equal $\partial  G_i$, the region $G_i$ may be not convex, and $G$ may be neither open nor closed; see an example in \cref{fig:ex1A}. These make our following discussions nontrivial.

We remark that $G_i=\{ {\bf u}: g_i({\bf u}) \ge 0 \}$ is closed for $i\in \widehat {\mathbb I}$, and 
$G_i=\{ {\bf u}: g_i({\bf u}) > 0 \}$ is open for $i \in {\mathbb I}$. 
Since for each $i \in {\mathbb I}$, the set $G_i$ is not necessarily convex, there is a possibility that 
$G$ may not be entirely contained in 
an {\em open} supporting halfspace at ${\bf u}_* \in {\mathcal S}_i$. 
This issue is avoid if 
the 
open region $\cap_{i \in \mathbb I} G_i$ is convex, which is satisfied by all the examples in \cref{sec:Examples} and implies that 
\begin{equation}\label{key1124}
	G \cap \left( \cup_{ {\bf u}_* \in {\mathcal S}_i } 
	\{ 
	{\bf u}\in \mathbb R^N:~ ( {\bf u} - {\bf u}_*) \cdot {\bf n}_{i*}  = 0\} \right) = \emptyset  
	\quad \forall i \in {\mathbb I}, 
\end{equation}
where ${\bf n}_{i*}$ 
is   
an inward-pointing normal vector of ${\mathcal S}_i$ at ${\bf u}_*$.

\begin{figure}[htbp]
	\centering
	\includegraphics[width=0.5866\textwidth]{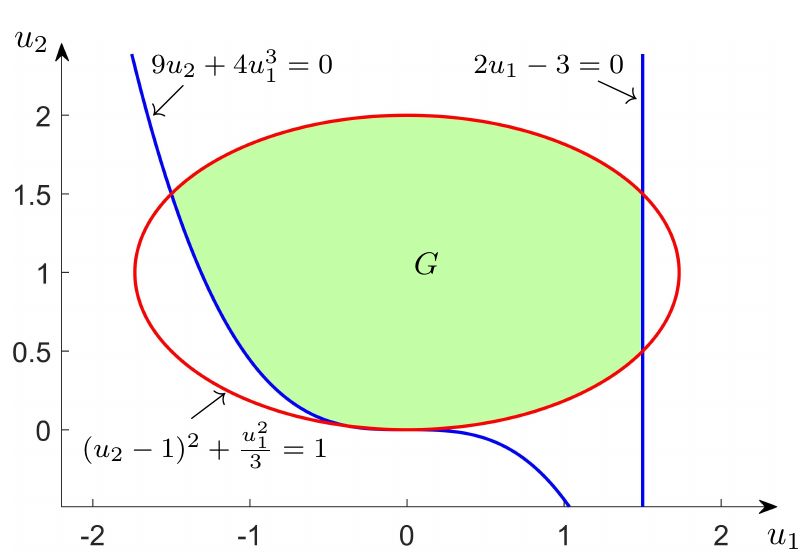} 
	\captionsetup{belowskip=-3pt}
	\caption{\small A convex region $G$ involving nonlinear constraints. $G=\{ {\bf u} \in \mathbb R^2: 
		g_1({\bf u}) \ge 0,~g_2({\bf u}) \ge 0,~g_3({\bf u}) > 0 \}$ with  
		$g_1({\bf u})=3-2u_1$, $g_2({\bf u})=9u_2 + 4 u_1^3$, and $g_3( {\bf u} )=1-{u_1^2}/3 - (u_2 -1)^2$.}
	\label{fig:ex1A}
\end{figure}

\begin{theorem}\label{main:thm}
Suppose that \cref{assum:convex} holds, condition \cref{key1124} is satisfied 
when $\mathbb I \neq \emptyset$, and the boundary of $G$ 
is piecewise $C^1$. Then the region $G$  has the following GQL representation: 
\begin{equation}\label{eq:Gs-general}
		G_* =  \Big\{ 
		{\bf u}\in \mathbb R^N:~ \big( {\bf u} - {\bf u}_*\big) \cdot {\bf n}_{i*}  \succ 0~~ \forall {\bf u}_* \in {\mathcal S}_i,~1\le i \le I
		\Big\},
\end{equation}
where the symbol ``$\succ$'' is taken as 
``$>$'' if $i \in \mathbb I$, or as ``$\ge$'' if $i \in \widehat {\mathbb I}$; the nonzero vector ${\bf n}_{i*}$ 
denotes  
an inward-pointing normal vector of ${\mathcal S}_i$ at ${\bf u}_*$. 
\end{theorem}

\begin{proof}
	The proof is divided into three steps. 
	
{\tt (\romannumeral1) Prove that $G \subseteq G_*$.} Let $\partial G =: \widetilde {\partial G} \cup \widehat  {\partial G}$ 
with $\widetilde {\partial G}$ denoting the set of smooth boundary points and $\widehat {\partial G}$  the set of nonsmooth boundary points.  For any ${\bf u}_* \in \widetilde {\partial G} \cap {\mathcal S}_i$, 
the hyperplane $({\bf u} - {\bf u}_*) \cdot {\bf n}_{i*} =0$ supports the region $G$, implying that 
\begin{equation}\label{key78334} 
	G \subseteq  \left\{ {\bf u} \in \mathbb R^N:~ (  {\bf u} - {\bf u}_*) \cdot {\bf n}_{i*}  \ge 0   \right\} \quad \forall {\bf u}_* \in \widetilde {\partial G} \cap {\mathcal S}_i,~ 1\le i \le I.
\end{equation}  
Next, we consider an arbitrary nonsmooth boundary point ${\bf u}_* \in \widehat {\partial G} \cap {\mathcal S}_i$. There exists a sequence of 
 smooth boundary points $\{{\bf u}_*^{(j)} \}_{j \in \mathbb N} \subset \widetilde {\partial G} \cap {\mathcal S}_i$ such that $
 \mathop {\lim }\limits_{j \to \infty }  {\bf u}_*^{(j)} = {\bf u}_*. 
$
For every ${\bf u}_*^{(j)}$, it follows from \cref{key78334} that 
\begin{equation}\label{key338}
	\big(  {\bf u} - {\bf u}_*^{(j)}  \big) \cdot {\bf n}_{i,{\bf u}_*^{(j)}}  \ge 0  \qquad \forall {\bf u} \in G,
\end{equation}
where ${\bf n}_{i,{\bf u}_*^{(j)}}$ is the inward-pointing normal vector of ${\mathcal S}_i$ 
at ${\bf u}_*^{(j)}$ satisfying $
\mathop {\lim }\limits_{j \to \infty }  {\bf n}_{i,{\bf u}_*^{(j)}} = {\bf n}_{i*}. 
$ 
Taking the limit $j \to +\infty$ in \cref{key338} gives 
\begin{equation*}
	(  {\bf u} - {\bf u}_*  ) \cdot {\bf n}_{i*} \ge 0 \quad \forall {\bf u} \in G \quad \forall {\bf u}_* \in \widehat {\partial G} \cap {\mathcal S}_i,~ 1\le i \le I, 
\end{equation*}
which along with \cref{key78334} yields 
\begin{equation}\label{key889}
	G \subseteq  \left\{ 
	{\bf u}\in \mathbb R^N:~ \big( {\bf u} - {\bf u}_*\big) \cdot {\bf n}_{i*}  \ge 0~~ \forall {\bf u}_* \in {\mathcal S}_i,~1\le i \le I
	\right\}.
\end{equation}
Based on \cref{key1124}, we then conclude that $G \subseteq G_*$.

{\tt (\romannumeral2) Prove that $G_* \subseteq {\rm cl} (G)$ by contradiction.} 
Assume that $G_* \not\subseteq {\rm cl} (G)$, namely, there exists ${\bf u}_0 \in G_*$ but ${\bf u}_0 \notin  {\rm cl} (G)$. According to the theory of convex optimization \cite{boyd2004convex}, 
the minimum of the convex function $\zeta ({\bf u}):= \frac12 \| {\bf u} - {\bf u}_0 \|^2$ over the closed convex region ${\rm cl} (G)$ is attained at certain boundary point ${\bf u}_{*0} \in \partial G$. In other words, 
${\bf u}_{*0}$ is a solution to the following optimization problem 
\begin{equation}\label{key3451}
	\begin{aligned}
	 \mathop{\rm minimize}\limits_{ {\bf u} \in {\rm cl} (G)  }~~ & \zeta ({\bf u})
	\\
	 {\rm subject~to}~~ & -g_i( {\bf u} ) < 0~~\forall i \in \mathbb I; \quad -g_i( {\bf u} ) \le 0~~ \forall i \in \widehat {\mathbb I}. 
\end{aligned}
\end{equation}
Since the function $-g_i( {\bf u} )$ is not necessarily convex, the problem \cref{key3451} is  generally not the standard form of convex optimization. 
Note that the condition ${\rm int}(G) \neq \emptyset$ ensures the 
Slater condition \cite{boyd2004convex,borwein2010convex} is satisfied. The Karush–Kuhn–Tucker (KKT) conditions \cite{boyd2004convex,borwein2010convex} tell us that there exist   
$\{\lambda_0, \lambda_1, \dots, \lambda_I \}$ 
such that  
%
%
\begin{align}\label{KKT-1}
	& 0= \nabla \zeta ({\bf u}_{*0}) - \sum_{i=1}^I \lambda_i \nabla g_i  ({\bf u}_{*0}) ,
	\\ \label{KKT-23}
	& 0 = \lambda_i g_i  ({\bf u}_{*0}), \qquad 1\le i \le I,
	\\ \label{KKT-45}
	& \lambda_i \ge 0, \qquad 0 \le i \le I.
\end{align}
Define $\mathbb I_+ :=\{ 1\le i \le I: \lambda_i >0 \}$. 
Obviously $\mathbb I_+ \neq \emptyset$; otherwise $\lambda_i=0$ for all $1\le i \le I$, so that ${\bf u}_{*0}-{\bf u}_0 = \nabla \zeta ({\bf u}_{*0}) ={\bf 0}$ which leads to the contradiction $\partial G \ni {\bf u}_{*0} = {\bf u}_0 \notin {\rm cl} (G)$. 
This also implies ${\bf u}_{*0} \neq {\bf u}_0$. 
Let ${\bf n}_{i*0}$ be the inward-pointing normal vector of ${\mathcal S}_i$ 
at ${\bf u}_{*0}$. 
Since there exist $\mu_i \ge 0$ 
such that $\nabla g_i  ({\bf u}_{*,0}) = \mu_i {\bf n}_{i*0}$, 
condition \cref{KKT-1} can be rewritten as 
\begin{equation}\label{KKT-1b}
	{\bf u}_{*0}-{\bf u}_0 = \sum_{ i \in \mathbb I_+ } \lambda_i \mu_i {\bf n}_{i*0}.  
\end{equation}
Thanks to \cref{KKT-23}, we obtain $g_i  ({\bf u}_{*0}) =0$ for all $i \in \mathbb I_+$, which along with ${\bf u}_{*0} \in \partial G$ leads to  
\begin{equation*}
	{\bf u}_{*0} \in {\mathcal S}_i =  \partial G_i \cap \partial G \qquad \forall i \in \mathbb I_+.
\end{equation*} 
Because ${\bf u}_0 \in G_*$, we then have $(  {\bf u}_0 - {\bf u}_{*0} ) \cdot {\bf n}_{i*0}  \succ 0$ for all $i \in \mathbb I_+$. 
This, together with \cref{KKT-1b} and ${\bf u}_{*0} \neq {\bf u}_{0}$, leads to a contradiction: 
\begin{align*}
	0 & > -\left \|  {\bf u}_0 - {\bf u}_{*0}  \right\|_2^2 = \left (  {\bf u}_0 - {\bf u}_{*0} \right) \cdot \left( {\bf u}_{*0} - {\bf u}_0  \right )
	\\
	& = \left (  {\bf u}_0 - {\bf u}_{*0} \right) \cdot
	\left( 
	 \sum_{ i \in \mathbb I_+ } \lambda_i  \mu_i {\bf n}_{i*0} \right )
 = \sum_{ i \in \mathbb I_+ } \lambda_i  \mu_i  \left (  {\bf u}_0 - {\bf u}_{*0} \right) \cdot {\bf n}_{i*0} \ge 0.
\end{align*} 
Thus the assumption $G_* \not\subseteq {\rm cl} (G)$ is incorrect. We have $ G_* \subseteq {\rm cl} (G)$.  

{\tt (\romannumeral3) Prove that $ G_* \subseteq G$.} 
If $\mathbb I = \emptyset$, then $G$ is a closed region and $G = {\rm cl}(G)$. 
We immediately obtain $G_* \subseteq G$ from step (\romannumeral2) of this proof. 
In the following, we focus on $\mathbb I \neq \emptyset$ and prove $ G_* \subseteq G$ by contradiction.  
Assume that 
there exists ${\bf u}_0 \in G_*$ but ${\bf u}_0 \notin G$. Because we have already shown $G_* \subseteq {\rm cl} (G)$ in step (\romannumeral2) of this proof, we then get ${\bf u}_0 \in {\rm cl} (G) \setminus G =  \partial G$. 
Note that ${\bf u}_0 \in G_*$ implies 
\begin{equation*}
	\left( {\bf u}_0 - {\bf u}_*\right) \cdot {\bf n}_{i*}  > 0 \qquad \forall {\bf u}_* \in {\mathcal S}_i,~~~\forall i \in \mathbb I,
\end{equation*}
which leads to ${\bf u}_0 \notin {\mathcal S}_i = \partial G_i \cap \partial G$ for all $i \in \mathbb I$. 
It follows that ${\bf u}_0 \notin \partial G_i$ for all $i \in \mathbb I$. 
Note for $i \in \mathbb I$, one has ${\bf u}_0 \in {\rm cl} (G) \subseteq {\rm cl} (G_i)$, 
which gives    
\begin{equation}\label{key33132}
	{\bf u}_0 \in {\rm cl} (G_i) \setminus \partial G_i = G_i \quad \forall i \in \mathbb I.
\end{equation}
On the other hand, ${\bf u}_0 \in {\rm cl} (G) \subseteq \cap_{i \in \mathbb {\widehat I}} G_i$, which along with \cref{key33132} implies    
${\bf u}_0 \in   ( \cap_{ i \in \mathbb I } G_i )\cap (\cap_{ i \in \mathbb {\widehat I} } G_i) = G$. 
This contradicts the assumption that ${\bf u}_0 \notin G$. 
Hence the assumption is incorrect, and we have $ G_* \subseteq G$.

Combining the conclusions proven in steps (\romannumeral1) and (\romannumeral3) gives $G=G_*$. 
\end{proof}

\begin{remark}
	If we replace ${\mathcal S}_i$ with ${\mathcal S}_i \cap \widetilde { \partial G }$ for $i \in \widehat{\mathbb I}$ in \cref{eq:Gs-general}, \cref{main:thm} remains valid, because for $i \in \widehat{\mathbb I}$ we have 	
	$
	\{ 
	{\bf u}: ( {\bf u} - {\bf u}_*) \cdot {\bf n}_{i*}  \ge 0~ \forall {\bf u}_* \in {\mathcal S}_i
	\} = 
	\{ 
	{\bf u}: ( {\bf u} - {\bf u}_*) \cdot {\bf n}_{i*}  \ge 0~ \forall {\bf u}_* \in {\mathcal S}_i \cap \widetilde { \partial G }
	\}.
	$
\end{remark}

\begin{figure}[thbp]
	\centering
	\includegraphics[width=0.5866\textwidth]{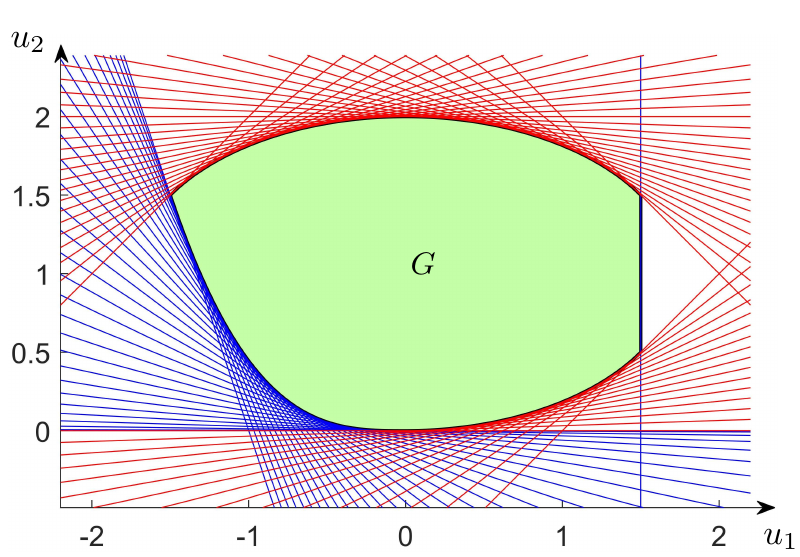} 
	\captionsetup{belowskip=-3pt}
	\caption{\small Illustration of the GQL representation for the convex region $G$ given in \cref{fig:ex1A}. 
	The blue (resp. red) lines correspond to closed (resp. open) supporting halfspaces.}
	\label{fig:ex1B}
\end{figure}

\begin{remark}
An illustration of the GQL representation \cref{eq:Gs-general}  is shown in \cref{fig:ex1B}. 
Different from \cref{eq:generalQL}, the GQL 
 representation \cref{eq:Gs-general} 
involves only at most $N$ rather than {\em all} the supporting halfspaces at 
each nonsmooth ``junction'' point. 
This makes the GQL representation \cref{eq:Gs-general} easier to formulate or construct. 
Besides, \cref{main:thm} does not require 
 $G$ to be closed or open.  
\end{remark}

\begin{remark}[Significance of GQL]\label{rem:S-GQL}
	Compared to the original form \cref{eq:ASS-G1} of the invariant region $G$ with nonlinear constraints, 
	its equivalent GQL representation $G_*$ in \cref{eq:Gs-general} 
	is described with only linear constraints. 
	  Such linearity gives 
	  the GQL representation some 
	   significant advantages over the original form  \cref{eq:ASS-G1} in 
	   analyzing and designing bound-preserving schemes; see \cref{sec:GQL-CP,sec:GQL-MMHD}. 
\end{remark}

\section{Construction of geometric quasilinearization}\label{sec:Construction}
With three methods and several examples, this section discusses how to construct GQL for convex invariant regions. 

\subsection{Methods for constructing GQL representations}\label{sec:method}
Based on \cref{main:thm,thm:minphi}, we introduce three simple effective methods for constructing the GQL representation of $G$.

\subsubsection{Gradient-based method} The first method is based on the following result, which is a direct consequence of \cref{main:thm}. 

\begin{theorem}\label{thm:Method1}
	Assume that the hypotheses of \cref{main:thm} hold and 
	\begin{equation}\label{eq:bb1}
		\nabla g_i({\bf u}_*) \neq {\bf 0}~~ \forall {\bf u}_* \in {\mathcal S}_i,~ 1\le i \le I,
	\end{equation}
then the invariant region $G$ is exactly equivalent to 
\begin{equation}\label{eq:Gs-method1}
		G_* =  \Big\{ 
	{\bf u}\in \mathbb R^N:~ \varphi_i( {\bf u}; {\bf u}_* ) \succ 0~~ \forall {\bf u}_* \in {\mathcal S}_i,~ 1\le i \le I
	\Big\},
\end{equation}
where the function $\varphi_i$ is linear with respect to $\bf u$, 
defined by 
\begin{equation} \label{eq:Method1Phi}
	\varphi_i( {\bf u}; {\bf u}_* ) := \left (   {\bf u} - {\bf u}_* \right) \cdot  \nabla g_i({\bf u}_*).
\end{equation}
\end{theorem}
 
\cref{thm:Method1} says if $\{\nabla g_i\}$ are computable 
and satisfy \cref{eq:bb1}, then we can directly obtain the GQL representation in the form \cref{eq:Gs-method1} with \cref{eq:Method1Phi}. 

In some cases, it is, however, difficult to calculate the gradients of nonlinear functions $\{g_i\}$, e.g., the implicit functions in \cref{eq:RHD-G2} and \cref{eq:RMHD-G1}. 
This motivates us to propose the following cross-product method based on suitable parametrization of the hypersurface 
${\mathcal S}_i$. The use of parametrization can also help to  
reduce or decouple the free auxiliary variables, which is highly desirable for 
bound-preserving applications; see the examples in \cref{sec:GQLexamples} and \cref{rem:reduce}.

\subsubsection{Cross-product method} Assume that for each $i$ the hypersurface ${\mathcal S}_i$ has the following parametric expression 
\begin{equation}\label{eq:para}
	{\mathcal S}_i = \left\{  {\bf u}_* = {\bf U}_i ( {\bm \theta}_{i*} ):~  {\bm \theta}_{i*} \in \Theta_i \subseteq \mathbb R^{N-1}   \right\}, 
\end{equation}
where ${\bf U}_i$ is a $C^1$ vector function defined on the parameter domain $\Theta_i$ with ${\mathcal S}_i$ being the function range. 
Denote $\theta_{i*}^{(k)}$ as the $k$th component of 
${\bm \theta}_{i*}$. 
For each $i$, we define  
\begin{equation*}
{\bm \tau}_{i,k} ( {\bm \theta}_{i*} ) := \frac{ \partial{\bf U}_i } { \partial \theta_{i*}^{(k)} },
\qquad  1\le k \le N-1. 	
\end{equation*}
The vectors $\{ {\bm \tau}_{i,k} ( {\bm \theta}_{i*} ): 1\le k \le N-1 \}$ are ($N-1$)  tangent vectors of the hypersurface ${\mathcal S}_i$ and generate its local tangent
space at ${\bf u}_*$. 
Then, the normal vector of ${\mathcal S}_i$ at ${\bf u}_*$ can be constructed using the $(N-1)$-ary analogue of the cross product (cf.~\cite[Pages 83--85]{Spivak1965Calculus}) 
in $\mathbb R^N$: 
\begin{align*}
& {\bf n}_{i*} = \delta_{i*}  \bigwedge\limits_{k = 1}^{N-1} {\bm \tau}_{i,k} ( {\bm \theta}_{i*} ) := \delta_{i*}  {\bm \tau}_{i,1} ( {\bm \theta}_{i*} ) 
\times {\bm \tau}_{i,2} ( {\bm \theta}_{i*} )  \times \cdots \times {\bm \tau}_{i,N-1} ( {\bm \theta}_{i*} ),
\end{align*}
where $\delta_{i*}$ is a nonzero factor which may be used to simplify the final formula  or/and to adjust the sign such that ${\bf n}_{i*}$ is directed towards the interior of $G$.

As a direct consequence of \cref{main:thm}, the following result holds. 

\begin{theorem}\label{thm:Method2}
	Suppose the hypotheses of \cref{main:thm} hold and 
	\begin{equation*}
		\bigwedge\limits_{k = 1}^{N-1} {\bm \tau}_{i,k} ( {\bm \theta}_{i*} ) \neq {\bf 0}~~ \forall {\bm \theta}_{i*} \in \Theta_i,~ 1\le i \le I,
	\end{equation*}
	then the region $G$ is exactly equivalent to 
	\begin{equation}\label{eq:Gs-method2}
			G_* =  \Big\{ 
			{\bf u}\in \mathbb R^N:~ \varphi_i( {\bf u}; {\bm \theta}_{i*} ) \succ 0~~ \forall {\bm \theta}_{i*} \in \Theta_i,~ 1\le i \le I
			\Big\},
	\end{equation}
	where the function $\varphi_i$ is linear with respect to $\bf u$, defined by 
	\begin{equation} \label{eq:Method2Phi1}
		\varphi_i( {\bf u};  {\bm \theta}_{i*}  ) := \Big ( {\bf u} - {\bf U}_i ( {\bm \theta}_{i*} )  \Big) \cdot \left(  \delta_{i*}  \bigwedge\limits_{k = 1}^{N-1} {\bm \tau}_{i,k} ( {\bm \theta}_{i*} )  \right).
	\end{equation}
\end{theorem}

\begin{remark}In many cases, there exists a natural (usually physics-based)  
	parametrization of the hypersurface 
	${\mathcal S}_i$, typically with the primitive quantities as parametric variables; 
	see the examples in \cref{sec:GQLexamples}. The advantages of using the 
	 parametric form \cref{eq:para} in the GQL representation will become more clear in those examples and the bound-preserving applications in \cref{sec:GQL-CP,sec:GQL-MMHD}. 	
\end{remark}

\subsubsection{Constructive method} 
For completeness, we also summarize 
the constructive approach and its variant as our third method. 
Recall that \cref{thm:minphi} has told us: if we can construct linear functions $\varphi_i( {\bf u}; {\bm \theta}_{i*} )$, $1 \le i \le I$, such that \cref{key656} holds, then 
the GQL representation of $G$ is \cref{eq:Gs}. 
The constructive approach does not require 
the assumptions in \cref{thm:Method1,thm:Method2}, 
but often needs some empirical trial-and-error techniques to find the qualified $\{\varphi_i\}$. 
In practice, one can use the proposed three methods in a hybrid way:  
first formally formulate $\{\varphi_i\}$ via either \cref{eq:Method1Phi} or \cref{eq:Method2Phi1} 
and then verify \cref{key656}.  
Such a hybrid approach is efficient, as it may exempt
the assumptions in \cref{thm:Method1,thm:Method2} and also avoid the trial-and-error procedure.

\subsection{Examples of GQL representations}\label{sec:GQLexamples} 
We give several examples for constructing GQL representations of convex invariant regions.

\subsubsection*{Example 1: Euler and Navier--Stokes systems}
	\begin{theorem}\label{thm:EulerNS-G1}
	For the 1D Euler and Navier--Stokes systems, the	GQL representation of the invariant region $G$ in \cref{eq:EulerNS-G1} is given by 
			\begin{equation}\label{Gs1:EulerNS}
		G_*= \left  \{ {\bf u} = ( \rho, m, E  )^\top:~\rho>0,~~ \varphi( {\bf u}; v_* ) >0~~\forall v_* \in \mathbb R \right  \}
	\end{equation}
	with $\varphi( {\bf u}; v_* ) := E - m v_* + \rho \frac{v_*^2}2$ being linearly dependent on   $\bf u$.  
	\end{theorem}
	
\begin{proof}	
    We respectively use the three methods proposed in \cref{sec:method} to derive the GQL representation for this example.  
	Note the first constraint in \cref{eq:EulerNS-G1} is linear.  
	
{\tt (\romannumeral1) Gradient-based method.} For the second constraint in \cref{eq:EulerNS-G1}, the gradient  
		$\nabla g({\bf u}) = ( \frac{m^2}{2\rho^2}, -\frac{m}{\rho}, 1 )^\top$, and the associated boundary hypersurface ${\mathcal S} = \{ 
		{\bf u}_* = (\rho_*, m_*, E_* )^\top: \rho_*>0,~g({\bf u}_*)=0 \}$ can be parameterized as 
\begin{align} \label{key98332}
			{\mathcal S} = 
		 \left\{ 
		{\bf u}_* = \left(\rho_*, \rho_* v_* , \frac{\rho_* }2 v_*^2 \right)^\top: \rho_*>0,~v_* \in \mathbb R \right\}.
\end{align}
		For ${\bf u}_* \in {\mathcal S}$ and ${\bf u}=(\rho, m, E)^\top$, we have 
		\begin{align}\label{keyddd2}  
			({\bf u}-{\bf u}_*) \cdot \nabla g({\bf u}_*) 
			= (\rho - \rho_*) \frac{ v_*^2 }2 + (m-\rho_* v_*) (-v_*) 
			+ E - \frac{\rho_*}2  v_*^2
			  =  \varphi( {\bf u}; v_* ). 
		\end{align}
		By \cref{thm:Method1}, we obtain the GQL representation \cref{Gs1:EulerNS} of $G$. 
	
{\tt (\romannumeral2) Cross-product method.} Based on the parametrization of $\mathcal S$ in \cref{key98332}, we can compute the normal vector of ${\mathcal S}$ at ${\bf u}_*$ by cross product 
		\begin{equation*}
			\frac{ \partial {\bf u}_* } { \partial \rho_* } 
			\times \frac{ \partial {\bf u}_* } { \partial v_* } 
			= \left( 1, v_*, \frac12 v_*^2 \right)^\top 
			\times \left( 0, \rho_*, \rho_* v_*\right)^\top 
			= \rho_*  \left( \frac12 v_*^2, -v_*, 1 \right)^\top 
			=: \frac{1}{\delta_*} 
			{\bf n}_*, 
		\end{equation*}
	where $\delta_*=1/\rho_*$ is a nonzero factor.  By \cref{thm:Method2} and $({\bf u}-{\bf u}_*) \cdot {\bf n}_*=\varphi( {\bf u}; v_* ) $, we 
	 get the GQL representation \cref{Gs1:EulerNS}. 
	
{\tt (\romannumeral3) Constructive method.}  
		Observe that 
		\begin{equation}\label{eq:constuctive11}
			\varphi( {\bf u}; v_* ) = E - m v_* + \rho \frac{v_*^2}2 
			= \frac{\rho}{2} \left( v_* - \frac{m}{\rho} \right)^2 + g({\bf u}) \ge g({\bf u}),
		\end{equation}
	which implies $\min_{  v_* \in \mathbb R }
	\varphi( {\bf u}; v_* )  = g ({\bf u})$ for $\rho>0$. According to \cref{thm:minphi}, we also achieve the GQL representation \cref{Gs1:EulerNS}. 
\end{proof}

\begin{remark}\label{rem:reduce}
Note 
only one free auxiliary variable $v_*$ explicitly appears in the GQL representation \cref{Gs1:EulerNS}.  This is benefited from the use of parametric form \cref{key98332}. 	
\end{remark}

\begin{remark}[Physical Interpretation of GQL]\label{rem:physical}
	It seems that the linear function $\varphi( {\bf u}; v_* )$ plays an energy-like role from a physical point of view. For the present example, 
	$\varphi( {\bf u}; v_* ) = \frac{1}{2} \rho ( v - v_* )^2 + \rho e$, which represents 
	the total energy in the reference frame moving at a velocity of $v_*$. 
\end{remark}

	We now utilize the cross-product method to construct the GQL representation of the invariant region $\widetilde G$ in \cref{eq:EulerNS-G2}, where the minimum entropy principle $S({\bf u}) :=p \rho^{-\Gamma} \ge S_{min}$ is also included. 

	\begin{theorem}
	For the 1D Euler and Navier--Stokes systems, the GQL representation of the invariant region $\widetilde G$ in \cref{eq:EulerNS-G2} is given by 
	\begin{equation}\label{Gs2:EulerNS}
		\widetilde G_*= \Big  \{ {\bf u} = ( \rho, m, E  )^\top:~\rho>0,~~ \widetilde \varphi( {\bf u}; \rho_*, v_* ) \ge 0~~\forall \rho_* \in \mathbb R^+~~  \forall v_* \in \mathbb R \Big \}
	\end{equation}
	with  $\widetilde \varphi( {\bf u}; \rho_*, v_* ) := {\bf u} \cdot {\bf n}_*  + S_{min} \rho_*^\Gamma$ and ${\bf n}_* := \big( \frac{v_*^2}2 - \frac{S_{min}\Gamma \rho_*^{\Gamma -1} }{\Gamma -1},~-v_*,~ 1   \big)^\top$.  
\end{theorem}

\begin{proof}
	We only need to handle the nonlinear constraint $\widetilde g({\bf u})>0$ in \cref{eq:EulerNS-G2}, with the boundary hypersurface 
	$
	\widetilde {\mathcal S} := \{ {\bf u}_*=(\rho_*,m_*,E_*): \rho_* >0,~\widetilde g({\bf u}_*)=0 \}
	$. Motivated from the equivalence of $\widetilde g({\bf u})=0$ and $p = S_{min}  \rho^{\Gamma}$, we find a natural physics-based parametrization of $\widetilde {\mathcal S}$ as 
	\begin{equation*}
		\widetilde {\mathcal S} = \left\{  
		{\bf u}_*= \left( \rho_*,~\rho_* v_*,~\frac12 \rho_* v_*^2 + 
		\frac{ S_{min}  \rho_*^\Gamma } {\Gamma -1}   \right)^\top:~ \rho_* >0,~v_* \in \mathbb R 
		\right\}. 
	\end{equation*}
	Then we can derive the normal vector ${\bf n}_*$ of ${\mathcal S}$ at ${\bf u}_*$ by cross product: 
	\begin{align*}
		\frac{ \partial {\bf u}_* }{\partial \rho_* } \times  \frac{ \partial {\bf u}_* }{\partial v_* }  = \left( 1,~v_*,~\frac{S_{min}\Gamma \rho_*^{\Gamma -1} }{\Gamma -1} + \frac{ v_*^2 }2  \right)^\top \times  \left( 0,~\rho_*,~\rho_* v_*  \right)^\top 
		= \rho_* {\bf n}_*.
	\end{align*}
	By \cref{thm:Method2} and $({\bf u}-{\bf u}_*) \cdot {\bf n}_* = \widetilde \varphi( {\bf u}; \rho_*, v_* )$, we 
	obtain the GQL representation \cref{Gs2:EulerNS}.
\end{proof}

\subsubsection*{Example 2: M1 model of radiative transfer} 
	
\begin{theorem}\label{thm:M1}
For the gray M1 moment system of radiative transfer, the GQL representation of the invariant region $G$ in \cref{eq:M1-G} is given by 
\begin{equation}\label{Gs:M1}
G_*= \left  \{ {\bf u} = ( E_r, {\mathbfcal F}_r  )^\top \in \mathbb R^4:~E_r - {\mathbfcal F}_r  \cdot {\bm \theta}_* \ge 0~~\forall {\bm \theta}_* \in \mathbb S_1({\bf 0}) \right  \}
\end{equation}
with $\mathbb S_1({\bf 0}) :=\{ {\bm x} \in \mathbb R^3: \| {\bm x} \|=1 \}$ denoting the unit 3D sphere. 
\end{theorem}

\begin{proof}
The constructive method is used for this example. The Cauchy–Schwarz inequality yields 
\begin{equation*}
	\varphi ( {\bf u}; {\bm \theta}_* ) := E_r - {\mathbfcal F}_r  \cdot {\bm \theta}_* \ge 
	g({\bf u}) \qquad \forall {\bm \theta}_* \in \mathbb S_1({\bf 0}),
\end{equation*}
where equality holds for ${\mathbfcal F}_r \neq {\bm 0}$ with ${\bm \theta}_* = {\mathbfcal F}_r / \| {\mathbfcal F}_r \| $  and for ${\mathbfcal F}_r = {\bm 0}$ with any ${\bm \theta}_*$. Thus, $\min_{ {\bm \theta}_* \in \mathbb S_1(0)  } \varphi ( {\bf u}; {\bm \theta}_* ) = g({\bf u})$, and by \cref{thm:minphi} we get the GQL representation \cref{Gs:M1}. 
\end{proof}

\subsubsection*{Example 3: Relativistic hydrodynamic system}

\begin{theorem}\label{thm:RHD1}
	For the 1D RHD system \cref{eq:1DRHD}, the GQL representation of the invariant region $G$ in \cref{eq:RHD-G1} is given by 
	\begin{equation}\label{Gs1:RHD}
		G_*= \left  \{ {\bf u} = ( D, m, E  )^\top: D>0,~\varphi({\bf u}; v_*)>0~~\forall v_*\in (-1,1) \right  \}
	\end{equation}
with $\varphi({\bf u}; v_*) := E-m v_* - D\sqrt{1-v_*^2}$ being a linear function of $\bf u$. 
\end{theorem}

\begin{proof}
	The first constraint in \cref{eq:RHD-G1} is linear. We deal with the second one by the constructive method. 
	The Cauchy–Schwarz inequality implies 
	\begin{equation*}
		\varphi ( {\bf u}; v_* ) \ge E- \sqrt{D^2 + m^2}  \sqrt{ v_*^2 + \left( \sqrt{1-v_*^2} \right)^2  } = E- \sqrt{D^2 + m^2} = g( {\bf u} ), 
	\end{equation*}
	where equality holds if $v_* = m/\sqrt{D^2 + m^2}$. This means $\min_{ v_*\in (-1,1) }
	\varphi ( {\bf u}; v_* ) = g( {\bf u} )$. According to \cref{thm:minphi}, we get the GQL representation \cref{Gs1:RHD}. 
\end{proof}

We now utilize the cross-product method to construct the GQL representation of the invariant region $\widetilde G$ in \cref{eq:RHD-G2}, where the minimum entropy principle is also included as a constraint. 

\begin{theorem}\label{thm:RHD2}
	For the 1D RHD system \cref{eq:1DRHD}, the GQL representation of the invariant region $\widetilde G$ in \cref{eq:RHD-G2} is given by 
	\begin{equation}\label{Gs2:RHD}
		\widetilde G_*= \Big  \{ {\bf u} = ( D, m, E  )^\top:~\rho>0,~~ \widetilde \varphi( {\bf u}; \rho_*, v_* ) \ge 0~~\forall \rho_* \in \mathbb R^+~~\forall v_* \in (-1,1) \Big \}
\end{equation}
with  $\widetilde \varphi( {\bf u}; \rho_*, v_* ) := {\bf u} \cdot {\bf n}_*  + S_{min} \rho_*^\Gamma$ and ${\bf n}_* := \left( -\sqrt{1-v_*^2} \big(1+ \frac{S_{min}\Gamma \rho_*^{\Gamma -1} }{\Gamma -1} \big),-v_*, 1   \right)^\top$.  
\end{theorem}

\begin{proof}
	We only need to tackle the second and third constraints in \cref{eq:RHD-G2}. 
  For the third constraint $\widetilde g({\bf u})\ge 0$, the corresponding boundary hypersurface is 
	$
	\widetilde {\mathcal S} := \{ {\bf u}_*=(\rho_*,m_*,E_*): \rho_* >0,~g({\bf u}_*)>0,~\widetilde g({\bf u}_*)=0 \}
	$. Based on the equivalence of $\widetilde g({\bf u})=0$ and $p = S_{min}  \rho^{\Gamma}$, we obtain a natural physics-based parametrization of $\widetilde {\mathcal S}$, namely, 
	{\small\begin{align*}
		\widetilde {\mathcal S} = \left\{  
		{\bf u}_* = \left( \frac{\rho_*}{\sqrt{1-v_*^2}},~ \frac{ \Big(\rho_*+\frac{S_{min}\Gamma \rho_*^\Gamma }{\Gamma -1} \Big) v_*}{1-v_*^2},~
	  \frac{ \rho_*+\frac{S_{min}\Gamma \rho_*^\Gamma }{\Gamma -1}  }{1-v_*^2} - S_{min} \rho_*^\Gamma   \right)^\top:~ \rho_* >0,~v_* \in (-1,1) 
		\right\}. 
	\end{align*}}
	We can then derive the normal vector ${\bf n}_*$ of ${\mathcal S}$ at ${\bf u}_*$ by cross product: 
	\begin{align*}
		\frac{ \partial {\bf u}_* }{\partial \rho_* } \times  \frac{ \partial {\bf u}_* }{\partial v_* } &  = \frac{1}{\delta_*} {\bf n}_*, \quad \mbox{with}~~ \delta_* := (1-v_*)^{5/2} \left(   \rho_* + \frac{ S_{min} \Gamma }{\Gamma -1} \rho_* ^\Gamma 
		( 1+v_*^2 - \Gamma v_*^2 )   \right)^{-1}. 
	\end{align*} 
	By \cref{thm:Method2} and $({\bf u}-{\bf u}_*) \cdot {\bf n}_* = \widetilde \varphi( {\bf u}; \rho_*, v_* ) $, 
	the  GQL representation for 
	$\widetilde g({\bf u})\ge 0$ is 
	\begin{equation}\label{key64234}
		\widetilde \varphi( {\bf u}; \rho_*, v_* ) \ge 0\qquad \forall \rho_* \in \mathbb R^+,\quad \forall v_* \in \mathbb R.
	\end{equation}
	Note that $S_{min}>0$ and 	
\begin{equation*}
		g({\bf u}) > g({\bf u}) -  \frac{ S_{min}}{ \Gamma -1 } 
	\left(  \frac{D^2}{ \sqrt{ D^2 + m^2 } }   \right)^\Gamma =  \widetilde \varphi \left( {\bf u};  \frac{ D^2 }{ \sqrt{D^2+m^2} }, \frac{m}{ \sqrt{D^2+m^2} } \right),
\end{equation*}
	which means that \cref{key64234} also implies $g({\bf u})>0$ in \cref{eq:RHD-G2}. That is, the second and third constraints in \cref{eq:RHD-G2} can be equivalently
	represented by \cref{key64234}. Therefore, we obtain the GQL representation \cref{Gs2:RHD}.
\end{proof}

\subsubsection*{Example 4: Ten-moment Gaussian closure system}
\begin{theorem}\label{thm:10M}
	For the 2D ten-moment Gaussian closure system \cref{eq:Ten-Moment}, the GQL representation of the invariant region $G$ in \cref{eq:10M-G11} is given by 
	\begin{equation}\label{Gs:10M}
		G_*= \Big  \{ {\bf u} \in \mathbb R^6 : \rho>0,~\varphi( {\bf u}; {\bm z}, {\bm v}_* )>0~~\forall  {\bm v}_* \in \mathbb R^2~~ \forall {\bm z} \in \mathbb R^2 \setminus \{ {\bf 0} \} \Big \},
	\end{equation}
	where ${\bf u} := ( \rho, {\bm m}, E_{11}, E_{12}, E_{22} )^\top$, and the function 
	$\varphi( {\bf u}; {\bm z}, {\bm v}_* )$ is linear with respect to $\bf u$: 	
\begin{equation}\label{eq:Phi-10M}
		\varphi( {\bf u}; {\bm z}, {\bm v}_* ) := {\bm z}^\top \left( 
	{\bf E} - {\bm m} \otimes {\bm v}_* +  \rho  \frac{{\bm v}_* \otimes {\bm v}_*}2 
	\right) {\bm z}.
\end{equation}
\end{theorem}

\begin{proof}
	We only need to deal with the nonlinear constraint in \cref{eq:10M-G11}.  Note that 
	\begin{align*}
		\varphi( {\bf u}; {\bm z}, {\bm v}_* ) = {\bm z}^\top \left( {\bf E} - \frac{ {\bm m} \otimes {\bm m} }{2\rho} \right) {\bm z} +  \frac{\rho}{2} \left| {\bm z} \cdot \left( 
		{\bm v}_* - \frac{\bm m}{\rho} \right) \right|^2,  
	\end{align*}
	which implies $\min_{ {\bm v}_* \in \mathbb R^2 }  \varphi( {\bf u}; {\bm z}, {\bm v}_* ) = {\bm z}^\top \big( {\bf E} - \frac{ {\bm m} \otimes {\bm m} }{2\rho} \big) {\bm z}$. By \cref{thm:minphi}, we immediately obtain the GQL representation \cref{Gs:10M}. 
\end{proof}

\subsubsection*{Example 5: Ideal MHD system}

\begin{theorem}\label{thm:IMHD-GQL}
	For the ideal MHD system \cref{eq:idealMHD}, the GQL representation of the invariant region $G$ in \cref{eq:G-iMHD} is given by 
\begin{equation}\label{eq:Gs-iMHD}
	G_* = \Big\{ 
	{\bf u}=(\rho, {\bm m}, {\bf B}, E)^\top \in \mathbb R^8:~ \rho>0,~\varphi({\bf u}; {\bm v}_*, {\bf B}_*) 
	 > 0~~
	\forall {\bm v}_*, {\bf B}_* \in \mathbb R^3
	\Big\}
\end{equation}
	with $\varphi({\bf u}; {\bm v}_*, {\bf B}_*) := {\bf u} \cdot {\bf n}_*  + \frac{ \|{\bf B}_*\|^2 } 2$ and $
	{\bf n}_* := \big( 
	\frac{\|{\bm v}_*\|^2}{2}, -{\bm v}_*, -{\bf B}_*, 1  
	\big)^\top.
	$
\end{theorem}

\begin{proof}
	We use to the gradient-based method. For the nonlinear constraint in \cref{eq:G-iMHD}, the gradient of $g({\bf u})$ is 
	$
		\nabla g ({\bf u}) = \big( \frac{ \| {\bm m} \|^2 }{ 2 \rho^2 }, -\frac{\bm m}{\rho}, -{\bf B}, 1    \big)^\top, 
	$
	and the corresponding boundary hypersurface is 
	$
	{\mathcal S} := \{ {\bf u}_*=(\rho_*,{\bm m}_*,{\bf B}_*,E_*)^\top: \rho_* >0, g({\bf u}_*)=0 \}
	$. 
	Based on the equivalence of $g({\bf u})=0$ and $p=0$, we obtain a natural physics-based  parametrization of $ {\mathcal S}$, namely, 
	\begin{equation*}
		{\mathcal S}= \left \{ {\bf u}_*= \left( \rho_*,\rho_* {\bm v}_*, {\bf B}_*, \frac{ 1  }{2} \big( \rho_* \|{\bm v}_*\|^2 + \| {\bf B}_*  \|^2 \big)  \right)^\top: \rho_* > 0,~{\bm v}_* \in \mathbb R^3,~{\bf B}_* \in \mathbb R^3
		\right\}.
	\end{equation*}
	For ${\bf u}_* \in {\mathcal S}$ and ${\bf u}=(\rho, {\bm m}, {\bf B}, E)^\top$, we have 
	$
		({\bf u}-{\bf u}_*) \cdot \nabla g({\bf u}_*) 
		= \varphi({\bf u}; {\bm v}_*, {\bf B}_*). 
	$
	By \cref{thm:Method1}, we obtain the GQL representation \cref{eq:Gs-iMHD}. 
\end{proof}

\subsubsection*{Example 6: Relativistic MHD system}

\begin{theorem}\label{thm:RMHD-GQL}
	For the relativistic MHD system \cref{eq:RMHD}, the GQL representation of the invariant region $G$ in \cref{eq:RMHD-G1} is given by 
	\begin{equation}\label{eq:Gs-RMHD}
		G_* = \Big\{ 
		{\bf u} \in \mathbb R^8:~ D >0,~\varphi({\bf u}; {\bm v}_*, {\bf B}_*) 
		> 0~~
		\forall {\bf B}_* \in \mathbb R^3~~ \forall {\bm v}_* \in \mathbb B_1({\bf 0})
		\Big\},
	\end{equation}
	where ${\bf u}=(D, {\bm m}, {\bf B}, E)^\top$, $\mathbb B_1({\bf 0}) :=\{ {\bm x} \in \mathbb R^3: \| {\bm x} \|\le 1 \}$ is a unit 3D ball, and the linear function $\varphi({\bf u}; {\bm v}_*, {\bf B}_*) := {\bf u} \cdot {\bf n}_*  + p_m^*$ with 
	\begin{align}
			& p_{m}^*  := \frac12 \left( 
			{ (1-{\| {\bm v}_* \|}^2) \|{\bf B}_*\|^2 +( {\bm v}_* \cdot {\bf B}_*)^2 } \right), \label{eq:RMHD:vecns2}
		\\
		\label{eq:RMHD:vecns}
	&{\bf n}_* := \left( - \sqrt {1 - {\| {\bm v}_* \|}^2},~
		- {\bm v}_*,~ - (1 - {\| {\bm v}_* \|}^2) {\bf B}_* - ( {\bm v}_* \cdot {\bf B}_*) {\bm v}_*,~1 \right)^{\top}.
\end{align}	 
Note that $p_{m}^*$ and ${\bf n}_*$ only depend on the free auxiliary variables $({\bm v}_*,{\bf B}_*)$. 
\end{theorem}

\begin{proof}
As shown in \cite{WuTangM3AS}, 
the region $G$ in \cref{eq:RMHD-G1} can be equivalently represented as 
	\begin{equation}\label{eq:G2-RMHD}
	G = \left\{ 
	{\bf u}\in \mathbb R^8:~ D >0,~g_2({\bf u})>0,~p({\bf u})>0
	\right\}
\end{equation}
with $g_2({\bf u}) :=E-\sqrt{ D^2 + \| {\bm m} \|^2 }$. Although the implicit function $p({\bf u})$ defined in \cref{getprimformU} can not be explicitly formulated, 
the corresponding 
boundary hypersurface ${\mathcal S} := 
\{ {\bf u}_*=(D_*, {\bm m}_*, {\bf B}_*, E_*)^\top: D_*>0, g_2({\bf u}_*)>0, p({\bf u}_*)=0 \}
$ has an explicit physics-based parameterization:
\begin{align*}
	{\mathcal S} = \Big\{ {\bf u}_* & = \Big(
	\rho_* \gamma_*,~ \rho_* \gamma_*^2 {\bm v}_* + \| {\bf B}_* \|^2 {\bm v}_* - ( {\bm v}_* \cdot {\bf B}_* ) {\bf B}_*,~{\bf B}_*,
	\\ 
& \qquad
	  \rho_* \gamma_*^2 + \| {\bf B}_* \|^2 - p_m^* 
	   \Big)^\top: \rho_*>0,~{\bf B}_* \in \mathbb R^3,~{\bm v}_* \in \mathbb B_1({\bf 0})
	\Big \}
\end{align*}
with $p_m^*$ defined in \cref{eq:RMHD:vecns2} and $\gamma_* := ( 1-\| {\bm v_*} \|^2 )^{\frac12}$. 
This parameterization 
is helpful for dealing with the highly nonlinear constraint $p({\bf u})>0$ by 
 the cross-product method. 
 For $1\le i \le 3$, denote ${\bf e}_i :=(\delta_{1i},\delta_{2i},\delta_{3i})$ with $\delta_{ij}$ being the Kronecker delta. 
 Taking the partial derivatives of 
 ${\bf u}_*$ with respect to the parametric variables 
 $\{ \rho_*, {\bm v}_*, {\bf B}_* \}$ gives 
{\small \begin{align*}
	\frac{ \partial {\bf u}_* }{ \partial \rho_* } 
	&= \left( \gamma_*,~\gamma_*^2 {\bm v}_*,~0,~0,0,~\gamma_*^2 \right)^\top,
	\\
	\frac{ \partial {\bf u}_* }{ \partial  v_{i*} } 
	&= \Big( \rho_* \gamma_*^3 v_{i*} , (\rho_*\gamma_*^2 + \| {\bf B}_* \|^2) {\bf e}_i + 2 \rho_* \gamma_*^4 v_{i*} {\bm v}_* - B_{i*} {\bf B}_*,
	0,0,0, 2 \rho_* \gamma_*^4 v_{i*} 
	- B_{i*} ( {\bm v}_* \cdot {\bf B}_* ) + \| {\bf B}_* \|^2 v_{i*}
	 \Big)^\top,
	 \\
	 \frac{ \partial {\bf u}_* }{ \partial B_{i*} } 
	 &= \Big( 0,~-({\bm v}_*\cdot {\bf B}_*) {\bf e}_i + 2 B_{i*} {\bm v}_* - v_i {\bf B}_*,~{\bf e}_i,
	 B_{i*}(1+ \| {\bm v}_* \|^2 ) - v_{i*} ({\bm v}_*\cdot {\bf B}_*) \Big),~~1\le i \le 3, 	
\end{align*}}
which are all perpendicular to the nonzero vector ${\bf n}_*$ defined in \cref{eq:RMHD:vecns}. 
This means ${\bf n}_*$ is parallel to the cross product 
$\frac{ \partial {\bf u}_* }{ \partial \rho_* } \times \big(
\bigwedge\limits_{i = 1}^{3} \frac{ \partial {\bf u}_* }{ \partial  v_{i*} } \big) \times \big( \bigwedge\limits_{i = 1}^{3} \frac{ \partial {\bf u}_* }{ \partial  B_{i*} } \big) $, implying that 
${\bf n}_*$ is a normal vector of ${\mathcal S}$ at ${\bf u}_*$. 
It can be verified that ${\bf n}_*$ is always directed towards the concave side of 
$\mathcal S$. 
By \cref{thm:Method2} and $({\bf u}-{\bf u}_*) \cdot {\bf n}_* = \varphi  ({\bf u}; {\bm v}_*, {\bf B}_*) $, we know that 
the  GQL representation for 
$p({\bf u})>0$ is 
\begin{equation}\label{key6467}
	\varphi  ({\bf u}; {\bm v}_*, {\bf B}_*) > 0\quad 
	 \forall {\bf B}_* \in \mathbb R^3 \quad  \forall {\bm v}_* \in \mathbb B_1({\bf 0}).
\end{equation}
If taking ${\bm v}_* = {\bm m}/\sqrt{D^2+\|{\bm m} \|^2}$ and ${\bf B}_* = {\bf 0}$, we obtain $\varphi  ({\bf u}; {\bm v}_*, {\bf B}_*)=g_2({\bf u})$, 
which means that \cref{key6467} also implies $g_2({\bf u})>0$ in \cref{eq:G2-RMHD}.  In other words, the second and third constraints in \cref{eq:G2-RMHD} can be equivalently
represented by \cref{key6467}. Therefore, we obtain the GQL representation \cref{eq:Gs-RMHD}.
\end{proof}

\section{Geometric quasilinearization for bound-preserving analysis}\label{sec:GQL-CP}
This section applies the GQL approach to analyze the bound-preserving property of numerical schemes and shows its remarkable advantages over direct and traditional approaches  by  diverse examples covering different schemes of three PDE systems in one and two dimensions.  
We only focus on first-order schemes for illustrative purposes, while 
the GQL approach is readily extensible to high-order schemes. 
The application of GQL to design high-order bound-preserving scheme will also be 
explored in \cref{sec:GQL-MMHD} for the multicomponent MHD system, to further demonstrate its capability in addressing challenging bound-preserving
problems that could not be coped with by direct approaches.

\subsection{Example 1: Euler system}\label{sec:Euler} Consider a finite volume scheme 
\begin{equation}\label{1DEuler}
\bar{\bf u}_j^{n+1} = \bar{\bf u}_j^{n} - \sigma   
\left( \hat {\bf f}_{j+\frac12} - \hat {\bf f}_{j-\frac12} \right),  
\end{equation}
for solving the 1D Euler system \cref{eq:1DEuler} on a uniform spatial mesh $\{ [x_{j-1/2}, x_{j+1/2} ] \}$ 
with $\sigma :={\Delta t}/{\Delta x}$ denoting the ratio of the temporal step-size $\Delta t$ to the spatial step-size $\Delta x$. Here $\bar{\bf u}_j^{n}$ is an approximation to the average of ${\bf u}(x,t_n)$ on cell $[x_{j-1/2}, x_{j+1/2} ]$, and 
$\hat {\bf f}_{j+1/2}$ is a numerical flux at $x_{j+1/2}$. 
For system \cref{eq:1DEuler}, it holds that ${\bf f}({\bf u}) = v {\bf u} + p(0,1,v)^\top$, which will be used in the following analysis.

We apply the GQL approach to 
analyze the bound-preserving property of the scheme \cref{1DEuler} 
with the invariant region $G$ defined in \cref{eq:EulerNS-G1}. Thanks to the 
GQL representation in \cref{thm:EulerNS-G1}, we have 
	\begin{equation}\label{eq:EulerGs}
	G =	G_*= \left  \{ {\bf u}:~{\bf u} \cdot {\bf e}_1 >0,~ {\bf u} \cdot {\bf n}_* >0~~\forall v_* \in \mathbb R \right  \}
	\end{equation}
	with ${\bf e}_1:= (1,0,0)^\top$ and ${\bf n}_* := \big( \frac{v_*^2}2, -v_*,1 \big)^\top$. 
GQL transfers the bound-preserving problem into preserving the positivity of 
 ${\bf u} \cdot {\bf e}_1$ and ${\bf u} \cdot {\bf n}_*$, which are all linear with respect to $\bf u$ and helpful for bound-preserving study.

\subsubsection*{Example 1.1: Lax--Friedrichs scheme} To clearly illustrate the basic idea, we begin with 
the simple Lax--Friedrichs scheme with the numerical flux 
$\hat {\bf f}_{j+1/2}$ taken as 
\begin{equation}\label{eq:LF}
	 \hat {\bf f}^{\mathrm {LF}} 
	( \bar{\bf u}_j^{n}, \bar{\bf u}_{j+1}^{n} ) := \frac12 
	\Big(  {\bf f}( \bar{\bf u}_j^{n}  )  
	+ {\bf f}( \bar{\bf u}_{j+1}^{n}  ) - \alpha_{n}
	( \bar{\bf u}_{j+1}^{n}  - \bar{\bf u}_j^{n} )   \Big),
\end{equation}
where $\alpha_{n} := \max_j \alpha ( \bar{\bf u}_{j}^n )$ with $\alpha ({\bf u}) :=  |v| + \sqrt{\Gamma p/\rho}$ being the spectral radius of the Jacobian matrix $\partial {\bf f}/\partial {\bf u}$. 
Given that $\bar{\bf u}_j^{n} \in G$ for all $j$, we wish $\bar{\bf u}_j^{n+1} \in G$. 
 
For respectively ${\bf n}={\bf e}_1$ and ${\bf n}={\bf n}_*$, thanks to the linearity of ${\bf u}\cdot {\bf n}$ we obtain 
\begin{equation*}
	\bar{\bf u}_j^{n+1} \cdot {\bf n} = \left( 1- \sigma \alpha_n  \right) \bar{\bf u}_j^{n} \cdot {\bf n}
	+ \frac{\sigma}2 \Big( 
	\alpha_n \bar{\bf u}_{j+1}^{n} \cdot {\bf n} - {\bf f} (\bar{\bf u}_{j+1}^{n}) \cdot {\bf n}  
	+  \alpha_n \bar{\bf u}_{j-1}^{n} \cdot {\bf n} + {\bf f} (\bar{\bf u}_{j-1}^{n}) \cdot {\bf n}   \Big).
\end{equation*}
The problem boils down to control the effect of ${\bf f}(\bar{\bf u}_{\pm}^{n})  \cdot {\bf n}$ by using the 
positivity of $\bar{\bf u}_{j\pm 1}^{n} \cdot {\bf n}$. 
For any ${\bf u}\in G$, we have ${\bf u}\cdot {\bf n} >0$ and  
\begin{align*}
\pm {\bf f} ( {\bf u} ) \cdot {\bf e}_1  & = \pm v ( {\bf u} \cdot {\bf e}_1 ) < \alpha ({\bf u}) {\bf u} \cdot {\bf e}_1,
\\
\pm {\bf f} ( {\bf u} ) \cdot {\bf n}_*  & = \pm  
v  ( {\bf u}  \cdot {\bf n}_* ) \pm   p  ( v-v_* )  
\\
& \le |v| ( {\bf u}  \cdot {\bf n}_* ) + \left( \frac12 \rho ( v-v_* )^2 + \rho e  \right)  \frac{p}{\rho \sqrt{2e}}  
\\
&= \left( |v| +  \frac{p}{\rho \sqrt{2e}} \right)  {\bf u}  \cdot {\bf n}_*  
< \alpha ({\bf u}) {\bf u} \cdot {\bf n}_*,
\end{align*}
which yield $\alpha_n \bar{\bf u}_{j\pm 1}^{n} \cdot {\bf n}  \mp {\bf f} (\bar{\bf u}_{j\pm1}^{n}) \cdot {\bf n}>0$.  
Thus 
we obtain 
$\bar{\bf u}_j^{n+1} \cdot {\bf n} > ( 1- \sigma \alpha_n  ) \bar{\bf u}_j^{n} \cdot {\bf n} \ge 0$ provided that $\sigma \alpha_n  \le 1$. This proves that  
the scheme \cref{1DEuler} with the Lax--Friedrichs flux \cref{eq:LF} is bound-preserving under the standard CFL condition $\sigma \alpha_n  \le 1$.

\begin{remark}{\em As we have seen, unlike the traditional approaches that require 
substituting the target scheme into the original nonlinear constraint of $G$ in \cref{eq:EulerNS-G1}, the GQL approach skillfully transfers all the constraints into linear ones which can be investigated in a unified way.} 
\end{remark}

\subsubsection*{Example 1.2: Gas-kinetic scheme} In order to demonstrate the advantages 
of the GQL approach in bound-preserving analysis, we consider a challenging example---the gas-kinetic scheme with the numerical flux  $\hat {\bf f}_{j+1/2}$ taken as 
\begin{align}\label{eq:GK}
&	\hat {\bf f}^{\mathrm {GK}} 
	( \bar{\bf u}_j^{n}, \bar{\bf u}_{j+1}^{n} ) := {\bf f}^+(\bar {\bf u}_j^{n}) + {\bf f}^-(\bar {\bf u}_{j+1}^{n}).
\\  \label{eq:fpm}
 &	
{\bf f}^\pm ({\bf u})  := \int_{ \mathbb R^\pm } \int_{ \mathbb R^{ M } } 
\begin{pmatrix}
	w
	\\
	w^2
	\\
	\frac{w}2 (w^2 + {\bm \xi }^2 )
\end{pmatrix} F(w, {\bm \xi}; {\bf u}) {\rm d} {\bm \xi } {\rm d} w, 
\end{align}
where $w$ is the particle velocity, ${\bm \xi} \in \mathbb R^{M}$ denotes the internal variables whose degrees of freedom $M = (3-\Gamma)/(\Gamma -1)$, 
the equilibrium distribution function $F$ is 
\begin{equation}\label{eq:Maxwellian}
	F(w, {\bm \xi}; {\bf u})  := \rho \left(  \frac{ \lambda  }{\pi} \right)^{ \frac{M+1}2  } {\rm e}^{ 
	-\lambda \left( (w-v)^2 + \| {\bm \xi} \|^2 \right) }
\end{equation}
with $\rho$ being the fluid velocity, $v$ being the fluid velocity, and $\lambda = \rho/(2 p)$.  

In a traditional approach \cite{tao1999gas}, the bound-preserving property of this scheme was studied by: (\romannumeral1) 
first,  evaluating the integration \cref{eq:fpm} as 
\begin{equation}\label{eq:fpm2}
	{\bf f}^\pm ({\bf u}) = \rho 
	\begin{pmatrix}
		\frac{v}2 {\rm erfc} (\mp \sqrt{\lambda} v ) \pm \frac12 \frac{ {\rm e}^{-\lambda v^2} }{ \sqrt{\pi \lambda} }
		\\
		\left(\frac{v^2}2 + \frac1{4\lambda} \right) {\rm erfc} (\mp \sqrt{\lambda} v ) \pm \frac{v}2 \frac{ {\rm e}^{-\lambda v^2} }{ \sqrt{\pi \lambda} }
		\\
		\left(\frac{v^3}4 + \frac{M+3}{8\lambda} v \right) {\rm erfc} (\mp \sqrt{\lambda} v ) \pm \left(\frac{v^2}4 + \frac{M+2}{8\lambda} \right)  \frac{ {\rm e}^{-\lambda v^2} }{ \sqrt{\pi \lambda} }
	\end{pmatrix}
\end{equation} 
with ${\rm erfc}(x) := \frac{2}{\sqrt{\pi}} \int_{x}^{+\infty} {\rm e}^{-w^2} dw$;   
(\romannumeral2) then, plugging the numerical flux \cref{eq:GK} with \cref{eq:fpm2} into \cref{1DEuler} and splitting the scheme \cref{1DEuler} into two steps; (\romannumeral3) and finally checking the bound-preserving properties of the  split schemes 
by verifying the original constraints of $G$ in \cref{eq:EulerNS-G1}. 
For this scheme, verifying 
the  nonlinear constraint in \cref{eq:EulerNS-G1} are difficult and complicated. 

Benefited from its linear feature, the GQL approach is highly effective for this challenging case. 
For ${\bf n}={\bf e}_1$ or ${\bf n}={\bf n}_*$, thanks to the linearity of ${\bf u}\cdot {\bf n}$ we obtain 
\begin{equation*}
	\bar{\bf u}_j^{n+1} \cdot {\bf n} =  \bar{\bf u}_j^{n} \cdot {\bf n} 
	- \sigma \left( {\bf f}^+( \bar {\bf u}_j^n ) - {\bf f}^-( \bar {\bf u}_j^n ) \right) \cdot {\bf n} 
	- \sigma {\bf f}^-( \bar {\bf u}_{j+1}^n ) \cdot {\bf n}  
	+ \sigma {\bf f}^+( \bar {\bf u}_{j-1}^n ) \cdot {\bf n}.
\end{equation*}
Note that for any ${\bf u} \in G$, we have $F(w, {\bm \xi}; {\bf u}) >0$ and 
\begin{align*}
	\pm {\bf f}^\pm ( {\bf u} ) \cdot {\bf e}_1 & =  \int_{ \mathbb R^\pm } \int_{ \mathbb R^{ M } } |w| F(w, {\bm \xi}; {\bf u}) {\rm d} {\bm \xi } {\rm d} w > 0,
	\\
	\pm {\bf f}^\pm ( {\bf u} ) \cdot {\bf n}_* & = \int_{ \mathbb R^\pm } \int_{ \mathbb R^{ M } } 
	\frac{ |w| }2 \Big( ( w - v_* )^2 + \| {\bm \xi} \|^2   \Big) F(w, {\bm \xi}; {\bf u}) {\rm d} {\bm \xi } {\rm d} w > 0.
\end{align*}
It follows, for ${\bf n} = {\bf e}_1$ and ${\bf n} = {\bf n}_*$ respectively, that 
\begin{equation}\label{key4532}
	\bar{\bf u}_j^{n+1} \cdot {\bf n} >  \bar{\bf u}_j^{n} \cdot {\bf n} 
- \sigma \left( {\bf f}^+( \bar {\bf u}_j^n ) - {\bf f}^-( \bar {\bf u}_j^n ) \right) \cdot {\bf n}. 
\end{equation}
Next, we use the positivity of ${\bf u} \cdot {\bf n}$ to bound the effect of $\left( {\bf f}^+(  {\bf u} ) - {\bf f}^-( {\bf u} ) \right) \cdot {\bf n}$ as follows: 
\begin{align*}
\left( {\bf f}^+(  {\bf u} ) - {\bf f}^-( {\bf u} ) \right) \cdot {\bf e}_1 
& = ( {\bf u} \cdot {\bf e}_1 ) \left(  \frac{ \lambda  }{\pi} \right)^{ \frac12  }   \left( \int_{ \mathbb R }  |w|  {\rm e}^{-\lambda (w-v)^2}  {\rm d} w \right)
\\
& \le   ( {\bf u} \cdot {\bf e}_1 )  \left(  \frac{ \lambda  }{\pi} \right)^{ \frac12  }   \left( \int_{ \mathbb R } ( |v| + |w-v|)  {\rm e}^{-\lambda (w-v)^2}  {\rm d} w \right)
\\
& =  ( {\bf u} \cdot {\bf e}_1 ) \left(  |v| + 1/\sqrt{\pi \lambda } \right) < a({\bf u}) {\bf u} \cdot {\bf e}_1,
\\
\left( {\bf f}^+(  {\bf u} ) - {\bf f}^-( {\bf u} ) \right) \cdot {\bf n}_* 
& = \int_{ \mathbb R } \int_{ \mathbb R^{ M } } 
\frac{ |w| }2 \Big( ( w - v_* )^2 + \| {\bm \xi} \|^2   \Big) F(w, {\bm \xi}; {\bf u}) {\rm d} {\bm \xi } {\rm d} w
\\
& \le \int_{ \mathbb R } \int_{ \mathbb R^{ M } } 
\frac{ |v| + |w-v| }2 \Big( ( w - v_* )^2 + \| {\bm \xi} \|^2   \Big) F(w, {\bm \xi}; {\bf u}) {\rm d} {\bm \xi } {\rm d} w
\\
& = |v| ( {\bf u} \cdot {\bf n}_* ) +  \frac{ \rho }{ 2 \sqrt{\pi \lambda} } \left( (v-v_*)^2 + 
\frac{M+2}{2 \lambda}  \right) 
\\
& \le |v| ( {\bf u} \cdot {\bf n}_* ) +  \frac{ \rho }{ 2 \sqrt{\pi \lambda} } \left( (v-v_*)^2 + 
\frac{M+1}{2 \lambda}  \right)  \frac{M+2}{M+1} 
\\
& = \left( |v| + \frac{M+2}{M+1 } (\pi \lambda)^{-\frac12} \right) ( {\bf u} \cdot {\bf n}_* )  < 
a({\bf u}) ( {\bf u} \cdot {\bf n}_* ).
\end{align*}
This implies $\left( {\bf f}^+( \bar {\bf u}_j^n ) - {\bf f}^-( \bar {\bf u}_j^n ) \right) \cdot {\bf n} < \alpha (\bar{\bf u}_j^n) \bar{\bf u}_j^n \cdot {\bf n} \le\alpha_n \bar{\bf u}_j^n \cdot {\bf n} $. It then follows from \cref{key4532} that 
$\bar{\bf u}_j^{n+1} \cdot {\bf n} > (1-\sigma \alpha_n)  \bar{\bf u}_j^{n} \cdot {\bf n} \ge 0 $ provided that $\sigma \alpha_n \le 1$. This proves that  
the scheme \cref{1DEuler} with the gas-kinetic flux \cref{eq:GK} is bound-preserving under the standard CFL condition $\sigma \alpha_n \le 1$.

\begin{remark}{\em The linearity of GQL brought by introducing the free auxiliary variable $v_*$ gives remarkable advantages in our above analysis. 
		Because $v_*$ is independent of all the system variables $\bf u$, it can freely move cross the integrals. 
		We no longer need to substitute a complicated scheme into the nonlinear function $g({\bf u})$ in \cref{eq:EulerNS-G1} to verify $g({\bf u})>0$. Instead, we work on the simpler but equivalent linear constraint ${\bf u} \cdot {\bf n}_*>0$. 
	The interested readers may compare the above analysis based on GQL and the traditional analysis in \cite{tao1999gas}.}
\end{remark}

\subsection{Example 2: Navier--Stokes system} Consider the scheme 
\begin{equation}\label{1DNS-scheme}
	\bar{\bf u}_j^{n+1} = \bar{\bf u}_j^{n} - \sigma   
	\big( \hat {\bf f}_{j+\frac12} - \hat {\bf f}_{j-\frac12} \big)+ \frac{\Delta t}{\Delta x^2}
	\frac{\eta}{\tt Re} {\bf H}_j 
\end{equation}
with ${\bf H}_j:= {\bf r} ( \bar {\bf u}_{j+1}^n ) - 
2 {\bf r} ( \bar {\bf u}_{j}^n ) + {\bf r} ( \bar {\bf u}_{j-1}^n )$, 
for solving the 1D dimensionless compressible Navier--Stokes equations \cref{eq:1DNS}. Here $\hat {\bf f}_{j+1/2}$ is taken as a bound-preserving numerical flux for the 1D Euler system \cref{eq:1DEuler}, for example, the Lax--Friedrichs flux \cref{eq:LF} or the gas-kinetic flux \cref{eq:GK}, which satisfy: if $\bar{\bf u}_j^{n} \in G$ for all $j$, then 
\begin{equation*} 
\big( \hat {\bf f}_{j+\frac12} - \hat {\bf f}_{j-\frac12} \big) \cdot {\bf n}  < \alpha_n \bar{\bf u}_j^{n} \cdot {\bf n} \qquad \forall j 
\end{equation*}
holds for respectively ${\bf n} = {\bf e}_1$ and ${\bf n} = {\bf n}_*$, according to the analysis in \cref{sec:Euler}. 
Thus we have 
\begin{equation}\label{1DNS-scheme2}
	\bar{\bf u}_j^{n+1} \cdot {\bf n} > (1-\sigma\alpha_n  ) \bar{\bf u}_j^{n} \cdot {\bf n} + \frac{\Delta t}{\Delta x^2}
	\frac{\eta}{\tt Re} {\bf H}_j \cdot {\bf n}.
\end{equation} 
Thanks to GQL, 
we clearly see that the bound-preserving essence is to 
control the potentially negative term ${\bf H}_j \cdot {\bf n}$  by the positive term $\bar {\bf u}_j^n \cdot {\bf n}$. 
Note ${\bf H}_j \cdot {\bf e}_1=0$, thereby $\bar{\bf u}_j^{n+1} \cdot {\bf e}_1 > (1-\sigma\alpha_n  ) \bar{\bf u}_j^{n} \cdot {\bf e}_1 \ge 0$ if $\sigma \alpha_n \le 1$. 
For any ${\bf u} \in G$ and $v_* \in \mathbb R$, we have 
\begin{align*}
 - \frac{v_*^2}2 < {\bf r}({\bf u})\cdot {\bf n}_* 
	 = \frac12 ( v - v_* )^2 + \frac{\Gamma}{{\tt Pr}~\eta} e - \frac{v_*^2}2 \le \max \left \{ 1, \frac{\Gamma}{ {\tt Pr}~\eta }  \right \} 
	\frac{1}{\rho} ( {\bf u} \cdot {\bf n}_* ) - \frac{ v_*^2 }2. 
\end{align*}
This gives 
\begin{align*}
	 {\bf H}_j \cdot {\bf n}_* & = \Big( {\bf r} ( \bar {\bf u}_{j+1}^n ) \cdot {\bf n}_* + {\bf r} ( \bar {\bf u}_{j-1}^n ) \cdot {\bf n}_* \Big)  - 
	2 {\bf r} ( \bar {\bf u}_{j}^n ) \cdot {\bf n}_*
	\\ 
	& \ge  \left( - \frac{v_*^2}2  - \frac{v_*^2}2 \right) 
	- 2 \left( \max \left \{ 1, \frac{\Gamma}{ {\tt Pr}~\eta }  \right \} 
	\frac{1}{\bar \rho_j^n  } ( \bar {\bf u}_j^n \cdot {\bf n}_* ) - \frac{ v_*^2 }2 \right)
	\\
	&  = - \frac{2}{\bar \rho_j^n  }  \max \left \{ 1, \frac{\Gamma}{ {\tt Pr}~\eta }  \right \} 
	 ( \bar {\bf u}_j^n \cdot {\bf n}_* ).  
\end{align*}
It then follows from \cref{1DNS-scheme2} that  
\begin{equation*}
	\bar{\bf u}_j^{n+1} \cdot {\bf n}_* > (1-\sigma  \alpha_n  ) \bar{\bf u}_j^{n} \cdot {\bf n}_* - \frac{\Delta t}{\Delta x^2}
	\frac{\eta}{\tt Re}
	 \frac{2}{\bar \rho_j^n  }  \max \left \{ 1, \frac{\Gamma}{ {\tt Pr}~\eta }  \right \} 
	( \bar {\bf u}_j^n \cdot {\bf n}_* ).
\end{equation*}
We then immediately have $\bar{\bf u}_j^{n+1} \cdot {\bf n}_*>0$, provided that
\begin{equation}\label{eq:CFLNS}
	\alpha_n \frac{\Delta t}{\Delta x} + \frac{\Delta t}{\Delta x^2} \frac{2}{\bar \rho_j^n {\tt Re} }  \max \left \{ \eta, \frac{\Gamma}{ {\tt Pr} }  \right \} \le 1. 
\end{equation} 
In conclusion, the scheme \cref{1DNS-scheme} is bound-preserving under condition \cref{eq:CFLNS}. 

\begin{remark} 

A standard approach for handling bound-preserving problems with multiple terms (e.g., convection term and diffusion term \cite{ZHANG2017301}, or convection term and source term \cite{zhang2011}) is based on decomposing the schemes into a convex combination of some subterms, and then enforcing all the subterms in $G$. This may lead to 
stricter conditions on the time step-size $\Delta t$. 
Since the linear feature of GQL has already naturally incorporated the convexity of $G$ into the GQL representation, technical convex decomposition is not necessary in the GQL approach. 
\end{remark}

\subsection{Example 3: Ten-moment Gaussian closure system}\label{sec:10M} 
Consider the scheme   
\begin{equation}\label{10M-scheme}
	\bar{\bf u}_{ij}^{n+1} = \bar{\bf u}_{ij}^{n} - \sigma_{1}  
	\big( \hat {\bf f}_{1,i+\frac12,j} - \hat {\bf f}_{1,i-\frac12,j} \big) 
	- \sigma_{2}   
	\big( \hat {\bf f}_{2,i,j+\frac12} - \hat {\bf f}_{2,i,j-\frac12} \big),
\end{equation}
for solving the 2D Gaussian closure equations \cref{eq:Ten-Moment}     
on a uniform Cartesian mesh $\{ [x_{i-1/2}, x_{i+1/2} ] \times [y_{j-1/2}, y_{j+1/2} ] \}$, 
with $\sigma_{1} =\frac{\Delta t}{\Delta x}$, $\sigma_{2} =\frac{\Delta t}{\Delta y}$. Here $\bar{\bf u}_{ij}^{n}$ denotes an approximation to the average of ${\bf u}(x,y,t_n)$ on each cell, and the Lax-Friedrichs numerical fluxes are considered, i.e.
\begin{align}\label{eq:3376}
	& \hat {\bf f}_{1,i+1/2,j} = \hat {\bf f}^{\rm LF}_1 ( \bar {\bf u}_{ij}^n, \bar {\bf u}_{i+1,j}^n ), \quad 
	\hat {\bf f}_{2,i,j+1/2} = \hat {\bf f}^{\rm LF}_2 ( \bar {\bf u}_{ij}^n, \bar {\bf u}_{i,j+1}^n ),
	\\ \label{eq:33736}
	& \hat {\bf f}^{\rm LF}_\ell (  {\bf u}^L,  {\bf u}^R ) 
	:= \frac12 \Big(  {\bf f}_\ell ( {\bf u}^L ) + {\bf f}_\ell ( {\bf u}^R ) - 
	\alpha_{\ell,n} ( {\bf u}^R - {\bf u}^L )  \Big), \quad \ell=1,2, 
\end{align}
where $\alpha_{\ell,n} = \max_{ij}  \alpha_{\ell} (\bar {\bf u}_{ij}^n)$, and $\alpha_{\ell} ({\bf u}) := |v_\ell| + \sqrt{ { p_{\ell \ell} }/{\rho } }$.

In the original form \cref{eq:10M-G} of $G$, the second constraint is 
the positive definiteness of a matrix $
{\bf E} - \frac{ {\bm m} \otimes {\bm m} }{2\rho}$ which nonlinearly 
depends on $\bf u$.   This leads to the challenges in   
the bound-preserving study. 
Thanks to \cref{thm:10M}, 
 the invariant region $G$ is equivalently represented as  
 	\begin{equation}\label{Gs2:10M}
  G_*= \Big  \{ {\bf u} \in \mathbb R^6 : {\bf u} \cdot {\bf e}_1 >0,~\varphi( {\bf u}; {\bm z}, {\bm v}_* )>0~~\forall  {\bm v}_* \in \mathbb R^2~~ \forall {\bm z} \in \mathbb R^2 \setminus \{ {\bf 0} \} \Big \}, 
 \end{equation}
where ${\bf e}_1 :=(1,0,\cdots,0)^\top$ and the linear function $\varphi( {\bf u}; {\bm z}, {\bm v}_* )$ is defined by \cref{eq:Phi-10M}.

We apply the GQL approach to investigate the bound-preserving property of the scheme \cref{10M-scheme} with \cref{eq:3376}. 
Similar to the Euler system, for any ${\bf u} \in G$ we have  ${\bf f}_\ell ( {\bf u} ) \cdot {\bf e}_1 = v_\ell ( {\bf u} \cdot {\bf e}_1 )$ and 
$$
\pm {\bf f}_\ell ( {\bf u} ) \cdot {\bf e}_1 \le |v_\ell| ( {\bf u} \cdot {\bf e}_1 )  <  \alpha_{\ell} ({\bf u})  ( {\bf u} \cdot {\bf e}_1 ), 
$$
which gives $\bar{\bf u}_{ij}^{n+1} \cdot {\bf e}_1 > 0$ under the CFL condition $\sigma_{1}  \alpha_{1,n} + \sigma_{2}  \alpha_{2,n} < 1$. In the following, we focus on 
the second constraint in \cref{Gs2:10M}. 
Thanks to the linearity of $\varphi( \cdot; {\bm z}, {\bm v}_* )$, we obtain 
\begin{equation}\label{key23554}
	\varphi( {\bf f}_1 ({\bf u}) ; {\bm z}, {\bm v}_* ) 
	= v_1  \varphi( {\bf u}; {\bm z}, {\bm v}_* ) + \left[ {\bm z} \cdot ( {\bm v} - {\bm v}_* )  \right] ( {\bm p}_1 \cdot {\bm z} )
\end{equation}
with the vector ${\bm p}_1 := (p_{11},p_{12})^\top$. 
For any ${\bf u} \in G$, using the AM–GM inequality gives   
\begin{align*}
	\Big| \left[ {\bm z} \cdot ( {\bm v} - {\bm v}_* )  \right] ( {\bm p}_1 \cdot {\bm z} ) 
	\Big| 
	& \le \frac12 \sqrt{\rho p_{11}} \left|  {\bm z} \cdot ( {\bm v} - {\bm v}_* ) \right|^2  + \frac1{2 \sqrt{  \rho p_{11} } }  \left| {\bm p}_1 \cdot {\bm z} \right|^2 
	\\
	&= \sqrt{ \frac{p_{11}} {\rho} } \varphi( {\bf u}; {\bm z}, {\bm v}_* ) 
	- \frac{ z_2^2 \det ( {\bf p} )  }{ 2\sqrt{\rho p_{11}} } 
	\le \sqrt{ \frac{p_{11}} {\rho} } \varphi( {\bf u}; {\bm z}, {\bm v}_* ),
\end{align*}
which together with the identity \cref{key23554} yields  
\begin{equation}\label{key235543}
	\pm \varphi( {\bf f}_1 ({\bf u}) ; {\bm z}, {\bm v}_* ) 
	\le \left( |v_1| + \sqrt{ {p_{11}} / {\rho} } \right)  \varphi( {\bf u}; {\bm z}, {\bm v}_* )  = \alpha_1 ( {\bf u} ) \varphi( {\bf u}; {\bm z}, {\bm v}_* ).  
\end{equation}
Using  
the linearity of $\varphi( \cdot; {\bm z}, {\bm v}_* )$ again  and \cref{key235543}, we obtain 
\begin{align*}
	 \varphi \left(  \hat {\bf f}_{1,i+\frac12,j} - \hat {\bf f}_{1,i-\frac12,j} ; {\bm z}, {\bm v}_* \right) & = 
	\frac12 \left[  \varphi( {\bf f}_1 (\bar {\bf u}_{i+1,j}^n) ; {\bm z}, {\bm v}_* ) - \alpha_{1,n} \varphi( \bar {\bf u}_{i+1,j}^n; {\bm z}, {\bm v}_* )   \right]
	\\
	 & \quad  + \frac12 \left[ - \varphi( {\bf f}_1 (\bar {\bf u}_{i- 1,j}^n) ; {\bm z}, {\bm v}_* ) - \alpha_{1,n} \varphi( \bar {\bf u}_{i- 1,j}^n; {\bm z}, {\bm v}_* )   \right]
\\
& \quad 
+ \alpha_{1,n} \varphi( \bar {\bf u}_{ij}^n; {\bm z}, {\bm v}_* )  
  \le \alpha_{1,n} \varphi( \bar {\bf u}_{ij}^n; {\bm z}, {\bm v}_* ). 
\end{align*}
Similarly, we have 
$
	\varphi \big(  \hat {\bf f}_{2,i,j+\frac12} - \hat {\bf f}_{2,i,j-\frac12} ; {\bm z}, {\bm v}_* \big) 
	\le \alpha_{2,n} \varphi( \bar {\bf u}_{ij}^n; {\bm z}, {\bm v}_* ). 
$ It then follows that 
\begin{equation}\label{10M-scheme3}
 \varphi \left( 	\bar{\bf u}_{ij}^{n+1}; {\bm z}, {\bm v}_* \right)  \ge 
 \left( 1 - \sigma_{1}  \alpha_{1,n} - \sigma_{2} \alpha_{2,n} \right)  \varphi \left( 	\bar{\bf u}_{ij}^{n}; {\bm z}, {\bm v}_* \right) >0,
\end{equation}
under the CFL condition $\sigma_{1}  \alpha_{1,n} + \sigma_{2}  \alpha_{2,n} < 1$. This, along with $\bar{\bf u}_{ij}^{n+1} \cdot {\bf e}_1 > 0$, implies $\bar{\bf u}_{ij}^{n+1} \in G_*=G$ and the bound-preserving property of the scheme \cref{10M-scheme} with \cref{eq:3376}.

\section{Application of GQL to design bound-preserving schemes for multicomponent MHD}\label{sec:GQL-MMHD} 
This section applies the GQL approach to develop bound-preserving high-order finite volume and discontinuous Galerkin schemes 
for the multicomponent MHD system. 
We mainly focus on the 2D case, while our discussions are  
extensible to the 3D case.  
The 2D multicomponent compressible MHD system for a ideal fluid mixture with $N_c$ components 
can be written as 
\begin{subequations}\label{eq:MMHD}
\begin{align}\label{eq:MMHDEQ}
	&\partial_t {\bf u}	+ \partial_x {\bf f}_1 ( {\bf u} ) + \partial_y {\bf f}_2 ( {\bf u} ) = {\bf 0},
	\\ \label{eq:MMHDuF}
	& {\bf u} = 
	\begin{pmatrix}
		\rho {\bf Y}
		\\
		\rho
		\\
		{\bm m}
		\\
		{\bf B}
		\\
		E
	\end{pmatrix},\quad
{\bf f}_\ell ( {\bf u} )= 	
\begin{pmatrix}
	\rho {\bf Y} v_\ell
	\\
	\rho v_\ell
	\\
	{\bm m} v_\ell - {\bf B} B_\ell + p_{tot} {\bf e}_\ell
	\\
	{\bf B} v_\ell  - {\bm v} B_\ell
	\\
	v_\ell ( E+p_{tot} ) - B_\ell ( {\bm v} \cdot {\bf B} )
\end{pmatrix},\quad \ell=1,2, 
\end{align}
\end{subequations}
along with the extra divergence-free condition on the magnetic field ${\bf B}$: 
\begin{equation}\label{eq:divB}
	\nabla \cdot {\bf B} := \partial_x B_1 + \partial_y B_2 = 0. 
\end{equation}
In \cref{eq:MMHDuF}, 
$\rho$ denotes the total density, ${\bm m}=\rho{\bm v}$ is the momentum with ${\bm v}$ being the fluid velocity, ${\bf Y}=( Y_1, \dots, Y_{n_c-1} )^\top$ denotes the mass fractions of 
the first $(n_c-1)$ components, the mass fraction of the $n_c$th component is $Y_{n_c} := 1-\sum_{k=1}^{n_c-1} Y_k$, and $p_{tot}=p+\frac{\| {\bf B} \|^2}2$ is the total pressure with the thermal pressure $p$ calculated by   
 \begin{equation}\label{MMHDdef:p}
 	p = ( \Gamma({\bf u}) - 1 ) \left( E - 
 	\frac{ \|{\bm m}\|^2 }{2\rho} -\frac{ \|{\bf B}\|^2 }2   \right), \qquad \Gamma({\bf u}):= 
 	\frac{ \sum_{k=1}^{n_c} \Gamma_k C_{v_k} Y_k  }{ \sum_{k=1}^{n_c} C_{v_k} Y_k  }, 
 \end{equation}
where $C_{v_k}>0$ and 
$\Gamma_k>1$ respectively denote the heat capacity at constant volume and the ratio of specific heats for species $k$. 

\subsection{GQL representation of invariant region} 
For the system \eqref{eq:MMHD}, the total density $\rho$ and the thermal pressure $p$ are all positive, and the mass fractions $\{Y_k\}_{k=1}^{n_c}$ are between $0$ and $1$. These constraints constitute the following invariant region 
\begin{equation}\label{eq:G-MMHD}
		G = \left\{ 
	{\bf u} \in \mathbb R^{n_c+7}:~ 0\le Y_k\le 1,~ 1\le k \le n_c,~ \rho > 0,~
	p({\bf u})  > 0 
	\right\}
\end{equation}
with $p({\bf u})$ is a highly nonlinear function defined by \eqref{MMHDdef:p}. 
Due to the strong nonlinearity and the underlying connections between the bound-preserving and divergence-free properties, 
the design and analysis of bound-preserving schemes for system \cref{eq:MMHD} are highly challenging. 

Following the GQL framework, 
the convex region $G$ in \eqref{eq:G-MMHD} can be equivalently represented as 
\begin{equation}\label{eq:Gs-MMHD}
	G_* = \left\{ 
	{\bf u} \in \mathbb R^{n_c+7}:~ {\bf u} \cdot {\bf e}_k\ge 0, 0\le k < n_c,~~{\bf u} \cdot {\bf e}_{n_c}> 0,~~
	\varphi({\bf u}; {\bm v}_*, {\bf B}_*) 
	> 0~
	\forall {\bm v}_*, {\bf B}_* \in \mathbb R^3
	\right\},
\end{equation} 
where ${\bf e}_0 := {\bf e}_{n_c} - \sum_{k=1}^{n_c-1} {\bf e}_k$, the vector  
${\bf e}_k$ for $k\ge 1$ has a $1$ in the $k$th component and zeros elsewhere, 
and 
 $\varphi({\bf u}; {\bm v}_*, {\bf B}_*) := {\bf u} \cdot {\bf n}_*  + \frac{ \|{\bf B}_*\|^2 } 2$ with $
{\bf n}_* = ( {\bf 0}_{n_c-1}, 
\frac{\|{\bm v}_*\|^2}{2}, -{\bm v}_*, -{\bf B}_*, 1  
)^\top.
$
In the following, we will derive bound-preserving schemes for \cref{eq:MMHD} based on the GQL representation \cref{eq:Gs-MMHD}. 
{\em The GQL approach will not only help 
	overcome the difficulties arising from the nonlinearity, but also play a crucial role in establishing the key relations between the bound-preserving property and a discrete divergence-free  (DDF) condition on the numerical magnetic field.}

\subsection{GQL bridges bound-preserving property and DDF condition}\label{sec:relation}
We focus on the Euler forward method for time discretization, while all our discussions are directly extensible to 
high-order strong-stability-preserving 
time discretizations \cite{GottliebShuTadmor2001} which are formally convex combinations of Euler forward.  
Consider the finite volume methods and the scheme of the cell averages of the discontinuous Galerkin method, which can be written into a unified form as 
\begin{equation}\label{2MHDscheme}
	\bar{\bf u}_{ij}^{n+1} = \bar{\bf u}_{ij}^{n} - \sigma_{1}  
	\big( \hat {\bf f}_{1,i+\frac12,j} - \hat {\bf f}_{1,i-\frac12,j} \big) 
	- \sigma_{2}   
	\big( \hat {\bf f}_{2,i,j+\frac12} - \hat {\bf f}_{2,i,j-\frac12} \big),
\end{equation}
for solving \cref{eq:MMHD}    
on a uniform Cartesian mesh $\{ {\mathcal I}_{ij} := [x_{i-1/2}, x_{i+1/2} ] \times [y_{j-1/2}, y_{j+1/2} ] \}$, 
with $\sigma_{1} =\frac{\Delta t}{\Delta x}$ and $\sigma_{2} =\frac{\Delta t}{\Delta y}$. Here $\bar{\bf u}_{ij}^{n}$ denotes the approximate cell average of ${\bf u}(x,y,t_n)$ on ${\mathcal I}_{ij}$. For a $(K+1)$th-order accurate scheme, 
in each cell ${\mathcal I}_{ij}$ 
a polynomial vector of degree $K$, denoted by 
 ${\bf U}_{ij}^n(x,y)$, is also constructed as the approximate solution, which is  either the reconstructed polynomial solution in a finite volume scheme or the discontinuous Galerkin polynomial solution. 
Denote $\{ \omega_q, x_i^{(q)} \}_{q=1}^Q$ and $\{ \omega_q, y_j^{(q)} \}_{q=1}^Q$ as the Gauss quadrature weights and nodes in $[x_{i-1/2}, x_{i+1/2} ]$ and $[y_{j-1/2}, y_{j+1/2} ]$, respectively.  
Let ${\bf u}^{\pm,q}_{i\mp\frac{1}{2},j} = {\bf U}_{ij}^n (x_{i\mp\frac12},y_j^{(q)})$, 
${\bf u}^{q,\pm}_{i,j\mp\frac{1}{2}} = {\bf U}_{ij}^n (x_i^{(q)},y_{j\mp\frac12})$. 
 The numerical fluxes in \eqref{2MHDscheme} are then given by  
\begin{equation}\label{eq:MMHD_numflux}
\hat {\bf f}_{1,i+\frac12,j} = \sum\limits_{q =1}^{Q}  \omega_q 
		\hat {\bf f}_1^{\rm LF}( {\bf u}^{-,q}_{i+\frac{1}{2},j}, {\bf u}^{+,q}_{i+\frac{1}{2},j} ),~~ \hat {\bf f}_{2,i,j+\frac12} = \sum\limits_{q =1}^{Q} \omega_q
		\hat {\bf f}_2^{\rm LF}( {\bf u}^{q,-}_{i,j+\frac{1}{2}} , {\bf u}^{q,+}_{i,j+\frac{1}{2}}   ),
\end{equation}
where 
$\hat {\bf f}_\ell^{\rm LF}(\cdot, \cdot)$ is taken as the 
Lax-Friedrichs flux \cref{eq:33736} with the numerical viscosity parameters 
	\begin{equation}\label{eq:2DhighLFpara}
	\alpha_{1,n} \ge 
	\max_{ i,j,\mu}  \widehat \alpha_1 \big(  { \bf u }_{i+\frac12,j}^{\mp,q} ,  { \bf u }_{i-\frac12,j}^{\pm,q}  \big)
	,\quad \alpha_{2,n} \ge   \max_{ i,j,q}  \widehat \alpha_2 \big(  { \bf u }_{i,j+\frac12}^{q,\mp} ,  { \bf u }_{i,j-\frac12}^{q,\pm}  \big). 
\end{equation}
Here 
$\widehat \alpha_\ell ( {\bf u}, \tilde{\bf u} ) =
\max\big\{ |v_\ell|+ {\mathcal{C}}_\ell, |\tilde v_\ell| +  \tilde {\mathcal{C}}_\ell , \frac{| \sqrt{\rho} v_\ell + \sqrt{\tilde \rho} \tilde v_\ell |} {\sqrt{\rho}+\sqrt{\tilde \rho}} + \max\{ {\mathcal{C}}_\ell , \tilde {\mathcal{C}}_\ell \}  \big\}
+ \frac{ \|{\bf B}-\tilde{\bf B}\| }{ \sqrt{\rho} + \sqrt{\tilde \rho}  }$, $\ell=1,2$,  and ${\mathcal C}_1$ and ${\mathcal C}_2$ are the fast magneto-acoustic speeds in the $x$- and $y$-directions, respectively.

Seeking a condition for the scheme \cref{2MHDscheme} to be bound-preserving is very challenging, due to the complexity of the system \cref{eq:MMHD} and the region \cref{eq:G-MMHD} as well as 
the intrinsic relations between the bound-preserving property and the DDF condition,  
On one hand, it is very difficult to establish such relations, since   
the bound-preserving property is an {\em algebraic} property while the DDF condition is a discrete {\em differential} property. In fact, their relations remained unclear for a long time, until the recent work \cite{Wu2017a} on the single-component MHD case.  
On the other hand, the DDF condition strongly couples 
the states $\{ {\bf u}^{\pm,q}_{i\mp\frac{1}{2},j}, {\bf u}^{q,\pm}_{i,j\mp\frac{1}{2}} \}$, making the traditional or standard analysis approaches (which typically rely on decomposing high-order or/and multidimensional schemes 
into convex combinations of first-order 1D schemes \cite{zhang2010,zhang2010b,zhang2012maximum}) {\em inapplicable} to the present case.

First, let us consider the first-order scheme to gain some insights. In this case, the polynomial degree $K=0$ so that 
${\bf U}_{ij}^n (x,y) \equiv \bar {\bf u}_{ij}^n$ for all $(x,y)\in {\mathcal I}_{ij}$, and we 
can reformulate the scheme \cref{2MHDscheme} as 
\begin{equation}\label{eq:MHD1st}
	\bar {\bf u}_{ij}^{n+1} = (1-  \lambda  ) \bar {\bf u}_{ij}^n + \sigma_1 \alpha_{1,n} {\bf \Pi}_1 + \sigma_2 \alpha_{2,n} {\bf \Pi}_2,
\end{equation}
with $\lambda := \sigma_1 \alpha_{1,n} +  \sigma_2 \alpha_{2,n} $, and  
{\small	\begin{align*}
	& {\bf \Pi}_1 = \frac12 
	\left( \bar {\bf u}_{i+1,j}^n - \frac{ {\bf f}_1( \bar {\bf u}_{i+1,j}^n)}{\alpha_{1,n}} 
	+ \bar {\bf u}_{i-1,j}^n + \frac{ {\bf f}_1( \bar {\bf u}_{i-1,j}^n)}{\alpha_{1,n}} \right),  
	{\bf \Pi}_2 = \frac12  \left(
	 \bar {\bf u}_{i,j+1}^n - \frac{ {\bf f}_2( \bar {\bf u}_{i,j+1}^n)}{\alpha_{2,n}} 
	+ \bar {\bf u}_{i,j-1}^n + \frac{ {\bf f}_2( \bar {\bf u}_{i,j-1}^n)}{\alpha_{2,n}} \right). 
\end{align*}
}

\begin{theorem}\label{thm:1st_MMHD}
If $\bar{\bf u}_{ij}^n \in G$ for all $i$ and $j$, then, under the CFL condition $\lambda \le 1$, the solution $\bar{\bf u}_{ij}^{n+1}$ of \eqref{eq:MHD1st} satisfies 
\begin{align}\label{eq:constrant1_1st}
	& \bar{\bf u}_{ij}^{n+1} \cdot {\bf e}_k \ge 0, \quad 0\le k < n_c,\qquad \bar{\bf u}_{ij}^{n+1} \cdot {\bf e}_{n_c} > 0,
	\\ \label{eq:constrant2_1st}
	& \varphi( \bar{\bf u}_{ij}^{n+1}; {\bm v}_*, {\bf B}_*) 
	> 
	-\Delta t ( {\bm v}_* \cdot {\bf B}_* ) {\rm div}_{ij} \bar {\bf B}
	 \qquad 
	\forall {\bm v}_*, {\bf B}_* \in \mathbb R^3,
\end{align} 
where 
$
{\rm div}_{ij} \bar {\bf B} := \frac{ \bar B_{1,i+1,j}^n - \bar B_{1,i-1,j}^n }{2 \Delta x} 
+ \frac{ \bar B_{2,i,j+1}^n - \bar B_{2,i,j-1}^n }{2 \Delta y} 
$ is a discrete divergence. 
Furthermore, if the states $\{\bar{\bf u}_{ij}^n\}$ satisfy the DDF condition ${\rm div}_{ij} \bar {\bf B} = 0$, then \cref{eq:constrant1_1st}--\cref{eq:constrant2_1st} imply $\bar{\bf u}_{ij}^{n+1} \in G_*=G$.
\end{theorem}

\begin{proof}
	For $0\le k < n_c$ and any ${\bf u}\in \{ \bar{\bf u}_{ij}^n \}$, we have 
	 $\pm {\bf f}_\ell (  {\bf u} ) \cdot {\bf e}_k 
	= \pm v_\ell ( {\bf u} \cdot {\bf e}_k  ) \le \alpha_{\ell,n} ( {\bf u} \cdot {\bf e}_k  ) $, which implies 
	${\bf \Pi}_\ell \cdot {\bf e}_k \ge 0$. Similarly, ${\bf \Pi}_\ell \cdot {\bf e}_{n_c} > 0$. These lead to \cref{eq:constrant1_1st}. 
	Following \cite[Lemma 2.6]{Wu2017a}, we can derive that  
	\begin{align*}
	 \varphi ( {\bf \Pi}_1  ; {\bm v}_*, {\bf B}_* ) >  \frac{ {\bm v}_* \cdot {\bf B}_* }{2\alpha_{1,n}} \left(  \bar B_{1,i-1,j}^n - \bar B_{1,i+1,j}^n \right),~~ 
	\varphi ( {\bf \Pi}_2 ; {\bm v}_*, {\bf B}_* ) >  \frac{ {\bm v}_* \cdot {\bf B}_* }{2\alpha_{2,n}} \left(  \bar B_{2,i,j-1}^n - \bar B_{2,i,j+1}^n \right). 
	\end{align*}
Thanks to the linearity of $\varphi (\cdot ; {\bm v}_*, {\bf B}_*)$, it then follows from \cref{eq:MHD1st} that 
\begin{align*}
\varphi( \bar{\bf u}_{ij}^{n+1}; {\bm v}_*, {\bf B}_*)  &= (1-\lambda) \varphi( \bar{\bf u}_{ij}^{n}; {\bm v}_*, {\bf B}_*) 
+ \sigma_1 \alpha_{1,n} \varphi ({\bf \Pi}_1 ; {\bm v}_*, {\bf B}_*) + \sigma_2 \alpha_{2,n} \varphi ({\bf \Pi}_2 ; {\bm v}_*, {\bf B}_*)
\\
& > 
 (1-\lambda) \varphi( \bar{\bf u}_{ij}^{n}; {\bm v}_*, {\bf B}_*) - \Delta t ( {\bm v}_* \cdot {\bf B}_* ) {\rm div}_{ij} \bar {\bf B}, 
\end{align*}
which yields \cref{eq:constrant2_1st} under the CFL condition $\lambda \le 1$. 
\end{proof}

\cref{thm:1st_MMHD} shows the connection between  
the bound-preserving property and a DDF condition, which is bridged by \cref{eq:constrant2_1st} with the help of the free auxiliary variables $\{ {\bm v}_*, {\bf B}_* \}$ in the GQL representation \cref{eq:Gs-MMHD}. 
This demonstrates the essential importance of the GQL approach in establishing this connection and its significant advantages for bound-preserving analysis and design.

Now, we use the GQL approach to explore bound-preserving high-order schemes with 
$K\ge 1$.  
Denote $\{ \widehat x_i^{(\beta)}   \}_{\beta=1} ^L$ and $\{ \widehat y_j^{(\beta)} \}_{\beta=1} ^{L}$ as the Gauss--Lobatto quadrature points in $[x_{i-1/2}, x_{i+1/2} ]$ and $[y_{j-1/2}, y_{j+1/2} ]$, respectively, and $\{\widehat \omega_\beta\}_{\beta =1}^L$ as the weights, with $L = \left\lceil \frac{K+3}2 \right\rceil$. 
Similar to \cref{thm:1st_MMHD} and \cite[Theorem 4.7]{Wu2017a}, the following result can be derived with the proof omitted here. 
\begin{theorem}\label{thm:high_MMHD}
	If, for all $i$ and $j$, $\bar{\bf u}_{ij}^n \in G$ and the polynomial vector ${\bf U}_{ij}^n(x,y)$ satisfies 
	\begin{equation}\label{PPLimiterCondition}
		{\bf U}_{ij}^n( \widehat x_i^{(\beta)}, y_j^{(q)} ), {\bf U}_{ij}^n(  x_i^{(q)}, \widehat y_j^{(\beta)} ) \in G \qquad \forall \beta,q,
	\end{equation}
then, the solution $\bar{\bf u}_{ij}^{n+1}$ of the scheme \cref{2MHDscheme} satisfies 
\begin{equation}\label{eq:constrant2_kst}
	\varphi( \bar{\bf u}_{ij}^{n+1}; {\bm v}_*, {\bf B}_*) 
	> 2( \widehat \omega_1 - \lambda ) \varphi( {\bf \Pi} ; {\bm v}_*, {\bf B}_*) 
	-\Delta t ( {\bm v}_* \cdot {\bf B}_* ) {\rm div}_{ij}  {\bf B}
\end{equation}
with ${\bf \Pi} := 
\frac{1}{ 2\lambda }
\sum_{q} { \omega _q  }
\big[
\sigma_1 \alpha_{1,n} \big(
{ \bf u }_{i+\frac12,j}^{-,q} 
+ { \bf u }_{i-\frac12,j}^{+,q}  
\big)
+ \sigma_2 \alpha_{2,n} \big(
{ \bf u }_{i,j+\frac12}^{q,-} 
+ { \bf u }_{i,j-\frac12}^{q,+}  
\big)
\big] \in G$. 
Furthermore, under the CFL condition $\lambda \le \widehat \omega_1$, we have 
\begin{align}\label{eq:constrant1_kst}
	 & \bar{\bf u}_{ij}^{n+1} \cdot {\bf e}_k\ge 0, \quad ~ 0\le k < n_c, \qquad \bar{\bf u}_{ij}^{n+1} \cdot {\bf e}_{n_c} > 0,
	\\   \label{eq:constrant3_kst}
	& \varphi( \bar{\bf u}_{ij}^{n+1}; {\bm v}_*, {\bf B}_*) > -\Delta t ( {\bm v}_* \cdot {\bf B}_* ) {\rm div}_{ij}  {\bf B} 
	\qquad 
	\forall {\bm v}_*, {\bf B}_* \in \mathbb R^3,
\end{align} 
where the discrete divergence is defined as ${\rm div} _{ij} {\bf B} := \frac12 \left( {\rm div} _{ij}^{{-}} {\bf B} + {\rm div} _{ij}^{{+}} {\bf B} \right)$ with 
\begin{align*}
	&
	{\rm div} _{ij}^{\mp} {\bf B} := \frac{1}{\Delta x} \sum \limits_{q=1}^{Q} \omega_q \Big(  B_{1,i+\frac{1}{2},j}^{\mp,q}
	- B_{1,i-\frac{1}{2},j} ^{\pm,q}  \Big) + \frac{1}{\Delta y} \sum \limits_{q=1}^{Q} \omega_q \Big(  B_{2,i,j+\frac{1}{2}}^{q,\mp}
	- B_{2,i,j-\frac{1}{2}} ^{q,\pm}  \Big).
\end{align*}
\end{theorem}


\begin{remark}\label{rem:high_MMHD}
The condition \cref{PPLimiterCondition} in \cref{thm:high_MMHD} is a standard condition for  bound-preserving finite volume and discontinuous Galerkin schemes (see \cite{zhang2010,zhang2010b}). This condition can be easily enforced by a simple scaling limiter (see \cref{sec:limiter}). 
Thanks to the GQL representation \cref{eq:Gs-MMHD}, we conclude from \cref{eq:constrant1_kst}--\cref{eq:constrant3_kst} that, in order to ensure $\bar{\bf u}_{ij}^{n+1} \in G_*=G$, a DDF condition 
\begin{equation}\label{DDF1}
	{\rm div} _{ij} {\bf B} := \frac12 \left( {\rm div} _{ij}^{{-}} {\bf B} + {\rm div} _{ij}^{{+}} {\bf B} \right) = 0 
\end{equation}
is also required. Unfortunately, the high-order schemes \cref{2MHDscheme} do not preserve the DDF condition \cref{DDF1}, which depends on the numerical solution information from adjacent cells. 
Although a few globally divergence-free techniques (e.g.~\cite{Li2011,Fu2018,chandrashekar2019global})
were developed and can enforce the condition \cref{DDF1}, the local scaling limiter for \cref{PPLimiterCondition} 
will destroy the globally divergence-free property. 
Notice that 
the locally divergence-free technique (e.g.~\cite{Li2005}) is compatible with the local scaling limiter, but can only guarantee ${\rm div} _{ij}^{{-}} {\bf B}=0$. 
In \cref{sec:highBP}, we will use the GQL approach to explore how to eliminate the effect of the remaining part ${\rm div} _{ij}^{{+}} {\bf B}$ by properly modifying the scheme \cref{2MHDscheme}. 
\end{remark}

\subsection{Seek high-order provably bound-preserving schemes via GQL}\label{sec:highBP} 
We have established the relations between the bound-preserving and divergence-free properties at the numerical level. 
Interestingly, at the continuous level, bound preservation is also closely related to the divergence-free condition \cref{eq:divB}: If condition \cref{eq:divB} is slightly violated, then even the exact solution of system \cref{eq:MMHD} may not stay in $G$; see \cite{WuShu2018} 
for a discussion which is also valid for system \cref{eq:MMHD}. 
To address this issue, we consider a modified formulation of the multicomponent MHD equations 
\begin{equation}\label{eq:MMHD-modify}
	\partial_t {\bf u}	+ \partial_x {\bf f}_1 ( {\bf u} ) + \partial_y {\bf f}_2 ( {\bf u} ) + (\nabla \cdot {\bf B})  {\bf S}({\bf u}) = {\bf 0}
\end{equation}
by adding an extra source term to \cref{eq:MMHDEQ} with ${\bf S}({\bf u})=({\bf 0}_{n_c},{\bf B}, {\bm v}, {\bm v}\cdot {\bf B})^\top$. 
Such a formulation was first proposed by Godunov  \cite{Godunov1972} for the purpose of entropy symmetrization in the single-component MHD case. Notice that, for divergence-free initial conditions, the exact solutions of the modified form  \cref{eq:MMHD-modify} and the standard form \cref{eq:MMHD} are the same. However, if the divergence-free condition \cref{eq:divB} is violated, 
the extra source term in the modified form  \cref{eq:MMHD-modify} becomes beneficial and helps 
 keep the exact solutions always in $G$; see \cite{WuShu2019} 
 for an analysis which also works for system \cref{eq:MMHD-modify}. 
This finding motivates us to explore bound-preserving schemes based on suitable discretization of the 
modified form  \cref{eq:MMHD-modify}. Thus we consider  
\begin{equation}\label{2MHDscheme_BP}
	\bar{\bf u}_{ij}^{n+1} = \bar{\bf u}_{ij}^{n} - \sigma_{1}  
	\big( \hat {\bf f}_{1,i+\frac12,j} - \hat {\bf f}_{1,i-\frac12,j} \big) 
	- \sigma_{2}   
	\big( \hat {\bf f}_{2,i,j+\frac12} - \hat {\bf f}_{2,i,j-\frac12} \big) -  \widehat {\bf S}_{ij} 
\end{equation}
by adding a properly discretized source term $\widehat {\bf S}_{ij}$ into the standard finite volume or discontinuous Galerkin schemes \eqref{2MHDscheme}. 
As discussed in \cref{rem:high_MMHD}, we can adopt a locally divergence-free technique for the magnetic components of ${\bf U}_{ij}^n(x,y)$ such that ${\rm div} _{ij}^{{-}} {\bf B}=0$. 
This gives $2{\rm div} _{ij} {\bf B}={\rm div} _{ij}^{{+}} {\bf B} = {{\rm div} _{ij}^{{+}} {\bf B} - {\rm div} _{ij}^{{-}} {\bf B}}$, thereby leading to 
\begin{equation}\label{eq:LDF}
{\rm div} _{ij} {\bf B} 	
= \frac{1}{2\Delta x} \sum \limits_{q=1}^{Q} \omega_q 
\Big( {\jump{B_1}}_{i+\frac{1}{2},j}^{q}
+ {\jump{B_1}}_{i-\frac{1}{2},j} ^{q}  \Big) + \frac{1}{2\Delta y} \sum \limits_{q=1}^{Q} \omega_q \Big(  {\jump{B_2}}_{i,j+\frac{1}{2}}^{q}
+ {\jump{B_2}}_{i,j-\frac{1}{2}} ^{q} \Big), 
\end{equation}
where ${\jump{B_1}}_{i+\frac{1}{2},j}^{q}=B_{1,i+\frac{1}{2},j}^{+,q}-B_{1,i+\frac{1}{2},j}^{-,q}$ and 
${\jump{B_2}}_{i,j+\frac{1}{2}}^{q}=B_{2,i,j+\frac{1}{2}}^{q,-}-B_{2,i,j+\frac{1}{2}}^{q,-}$ are the jumps of  
the normal magnetic component 
 across the cell interface. 
Using the GQL approach with the linearity of $\varphi( \cdot; {\bm v}_*, {\bf B}_*) $ and the estimate \cref{eq:constrant2_kst} under the hypothesis of \cref{thm:high_MMHD},  
we obtain  
\begin{align}
	\varphi( \bar{\bf u}_{ij}^{n+1}; {\bm v}_*, {\bf B}_*) 
	> 2( \widehat \omega_1 - \lambda ) \varphi( {\bf \Pi} ; {\bm v}_*, {\bf B}_*) 
	-  \left[ \Delta t ( {\bm v}_* \cdot {\bf B}_* ) {\rm div}_{ij}  {\bf B} +  \widehat {\bf S}_{ij} \cdot {\bf n}_* \right]. \label{eq:constrant2_kst22}
\end{align}
Then the key is to carefully design $\widehat {\bf S}_{ij}$ to exactly offset the effect of 
${\rm div} _{ij} {\bf B}$ in \cref{eq:constrant2_kst22}, so that the resulting schemes \eqref{2MHDscheme_BP} become bound-preserving. 
Observing that for any $b \in \mathbb R$ and any ${\bf u} \in G$, 
\begin{equation}\label{eq:bvBSn}
	b( {\bm v}_* \cdot {\bf B}_* + {\bf S}({\bf u})\cdot {\bf n}_* ) 
= b ( {\bm v} - {\bm v}_* ) \cdot ({\bf B}-{\bf B}_*) 
\le |b|\rho^{1/2} \varphi ( {\bf u}; ; {\bm v}_*, {\bf B}_*),
\end{equation}
we devise 
\begin{equation}\label{MMHD-S}
	\begin{aligned}
		\widehat {\bf S}_{ij} & =  	   
		 \frac{\sigma_1}{2} 
		 \sum \limits_{q=1}^{Q} \omega_q 
		 \Big[ {\jump{B_1}}_{i+\frac{1}{2},j}^{q} {\bf S} ( {\bf u}_{i+\frac12,j}^{-,q} )
		 + {\jump{B_1}}_{i-\frac{1}{2},j} ^{q}  {\bf S} ( {\bf u}_{i-\frac12,j}^{+,q} ) \Big]
  \\
 & +  \frac{\sigma_2}{2}  \sum \limits_{q=1}^{Q} \omega_q \Big[ {\jump{B_2}}_{i,j+\frac{1}{2}}^{q} {\bf S} ( {\bf u}_{i,j+\frac12}^{q,-} )
	+ {\jump{B_2}}_{i,j-\frac{1}{2}} ^{q}  {\bf S} ( {\bf u}_{i,j-\frac12}^{q,+} ) \Big],
\end{aligned}
\end{equation}
such that the last term in \cref{eq:constrant2_kst22} satisfies 
\begin{align} \nonumber
	&  \Delta t ( {\bm v}_* \cdot {\bf B}_* ) {\rm div}_{ij}  {\bf B} +  \widehat {\bf S}_{ij} \cdot {\bf n}_* 
	\\ \nonumber
	& = 			 \frac{\sigma_1}{2} 
	\sum \limits_{q=1}^{Q} \omega_q 
	\Big[ {\jump{B_1}}_{i+\frac{1}{2},j}^{q} \Big( {\bm v}_* \cdot {\bf B}_* + {\bf S} ( {\bf u}_{i+\frac12,j}^{-,q} ) \cdot {\bf n}_* \Big)
	+ {\jump{B_1}}_{i-\frac{1}{2},j} ^{q}  \Big( {\bm v}_* \cdot {\bf B}_* + {\bf S} ( {\bf u}_{i-\frac12,j}^{+,q} ) \cdot {\bf n}_* \Big) \Big]
	\\ \nonumber
	& \quad +  \frac{\sigma_2}{2}  \sum \limits_{q=1}^{Q} \omega_q \Big[ {\jump{B_2}}_{i,j+\frac{1}{2}}^{q} \Big( {\bm v}_* \cdot {\bf B}_* + {\bf S} ( {\bf u}_{i,j+\frac12}^{q,-} ) \cdot {\bf n}_* \Big)
	+ {\jump{B_2}}_{i,j-\frac{1}{2}} ^{q}  \Big( {\bm v}_* \cdot {\bf B}_* + {\bf S} ( {\bf u}_{i,j-\frac12}^{q,+}  ) \cdot {\bf n}_* \Big) \Big]
	\\ \label{eq:keyIEQ}
	& \le \varepsilon \lambda \varphi( {\bf \Pi} ; {\bm v}_*, {\bf B}_*), 
\end{align}
where we use \cref{eq:LDF} in the equality and \cref{eq:bvBSn} in 
the inequality, and $\varepsilon = \max\{ \beta_{1}/\alpha_{1,n}, \beta_{2}/\alpha_{2,n} \}$ with $\beta_{1} = \max_{i,j,q} \big\{ \big|{\jump{B_1}}_{i+\frac{1}{2},j}^{q}\big| ({\rho_{i+\frac12,j}^{\pm,q}} )^{1/2} \big\}$ and 
$\beta_{2} = \max_{i,j,q} \big\{ \big|{\jump{B_2}}_{i,j+\frac{1}{2}}^{q} \big| ({\rho_{i,j+\frac12}^{q,\pm}} )^{1/2} \big\}$. 
Combining \cref{eq:constrant2_kst22} with \cref{eq:keyIEQ}, we obtain 
$$
\varphi( \bar{\bf u}_{ij}^{n+1}; {\bm v}_*, {\bf B}_*)  > 2( \widehat \omega_1 - \lambda - \varepsilon \lambda ) \varphi( {\bf \Pi} ; {\bm v}_*, {\bf B}_*) \ge 0, 
$$
under the CFL condition $(1+\varepsilon) \lambda \le \widehat \omega_1$. Notice that the first $n_c$ components of $\widehat {\bf S}_{ij}$ are zeros, which implies \cref{eq:constrant1_kst} in \cref{thm:high_MMHD} also holds for the modified schemes \cref{2MHDscheme_BP}. 
In summary, we obtain:
\begin{theorem}\label{thm:high_MMHD:modify}
	If for all $i$ and $j$, $\bar{\bf u}_{ij}^n \in G$ and the polynomial vector ${\bf U}_{ij}^n(x,y)$ satisfies 
	\cref{PPLimiterCondition} and ${\rm div} _{ij}^{{-}} {\bf B}=0$, 
	then, under the CFL condition $(1+\varepsilon) \lambda \le \widehat \omega_1$, the solution $\bar{\bf u}_{ij}^{n+1}$ of \cref{2MHDscheme_BP} is always preserved in $G_*$.
\end{theorem}

\Cref{thm:high_MMHD:modify} indicates that, if we use the scaling limiter in \cref{sec:limiter} to enforce  \cref{PPLimiterCondition} and a locally divergence-free technique to ensure ${\rm div} _{ij}^{{-}} {\bf B}=0$, then the schemes \cref{2MHDscheme_BP} with \cref{MMHD-S} are bound-preserving.  The bounds are also preserved if a high-order strong-stability-preserving 
time discretization \cite{GottliebShuTadmor2001} is used to replace the Euler forward method.

\section{Experimental results}
\label{sec:experiments}
This section gives two highly demanding numerical examples to further demonstrate 
our theoretical analysis as well as the robustness and effectiveness of 
the bound-preserving schemes designed via GQL in \cref{sec:highBP} for the 2D multicomponent MHD. 
We use the proposed bound-preserving third-order locally divergence-free discontinuous Galerkin method for spatial discretization. 
As the tests involve strong discontinuities, 
the locally divergence-free WENO limiter \cite{ZhaoTang2017} is also employed in some trouble cells adaptively detected by the indicator of \cite{Krivodonova}.
The third-order strong-stability-preserving Runge-Kutta method \cite{GottliebShuTadmor2001}  is adopted   for time discretization, with the CFL number set as $0.15$. 
 

\begin{figure}[htbp]
	\centering
	{\includegraphics[width=0.24\textwidth]{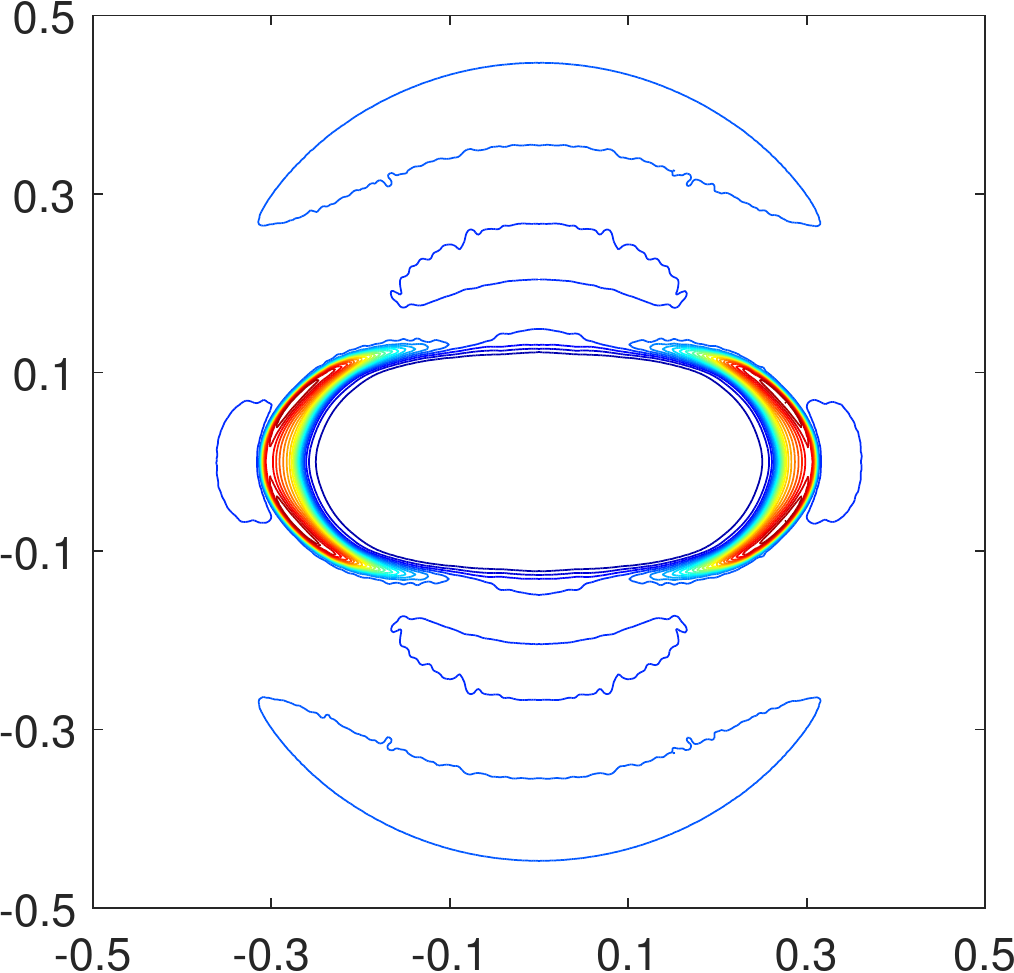}}
	{\includegraphics[width=0.24\textwidth]{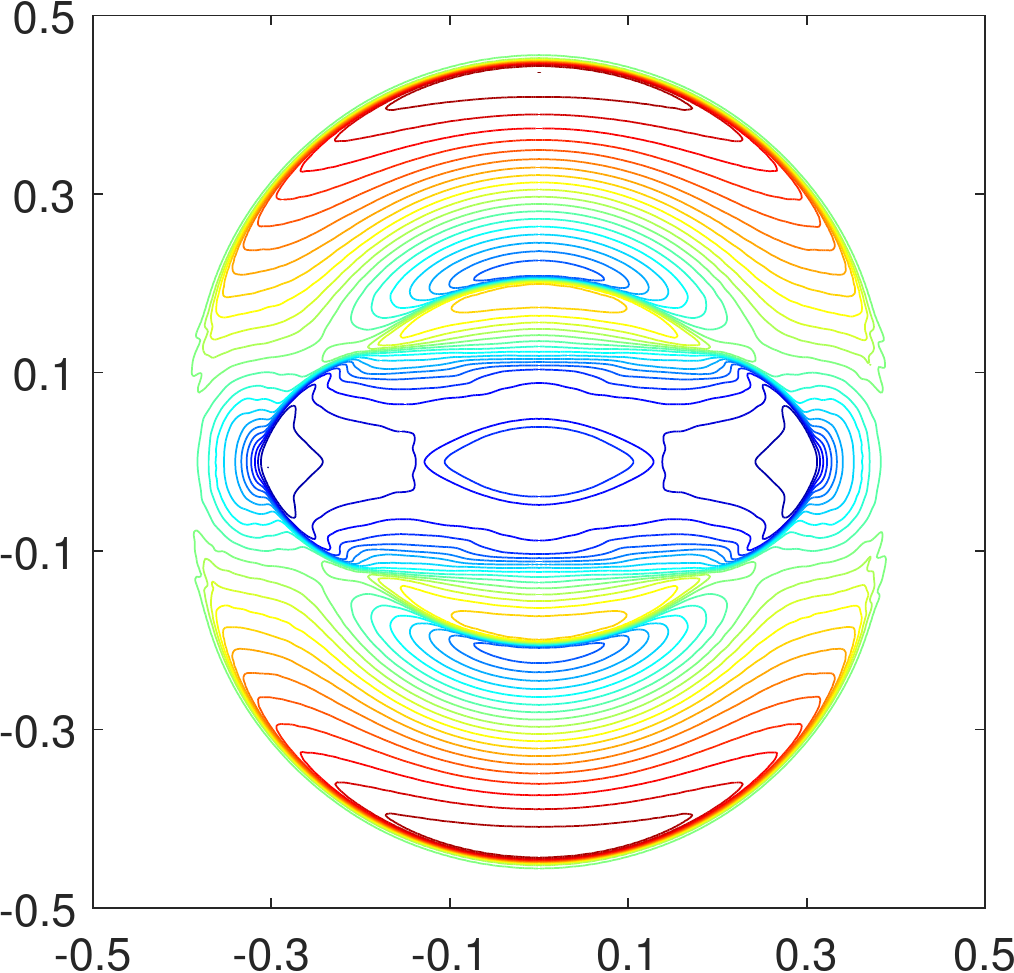}}
	{\includegraphics[width=0.24\textwidth]{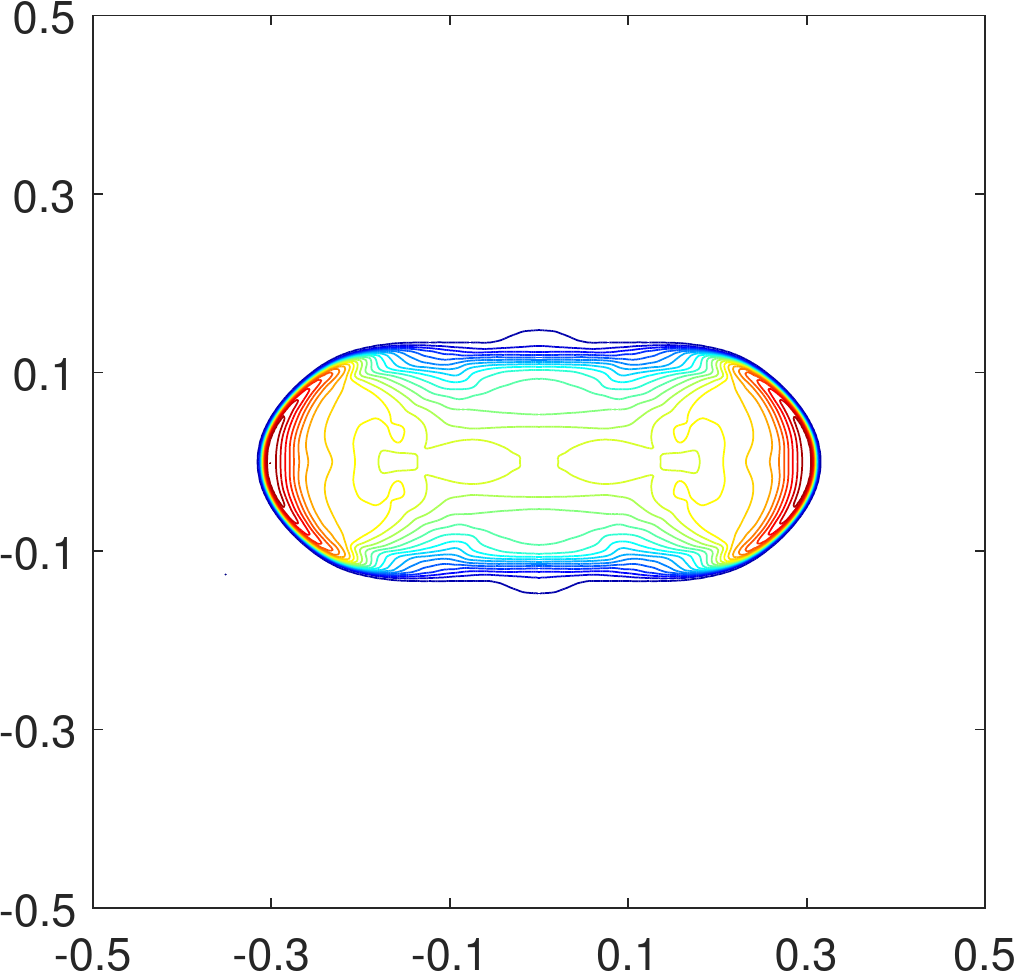}}
	{\includegraphics[width=0.24\textwidth]{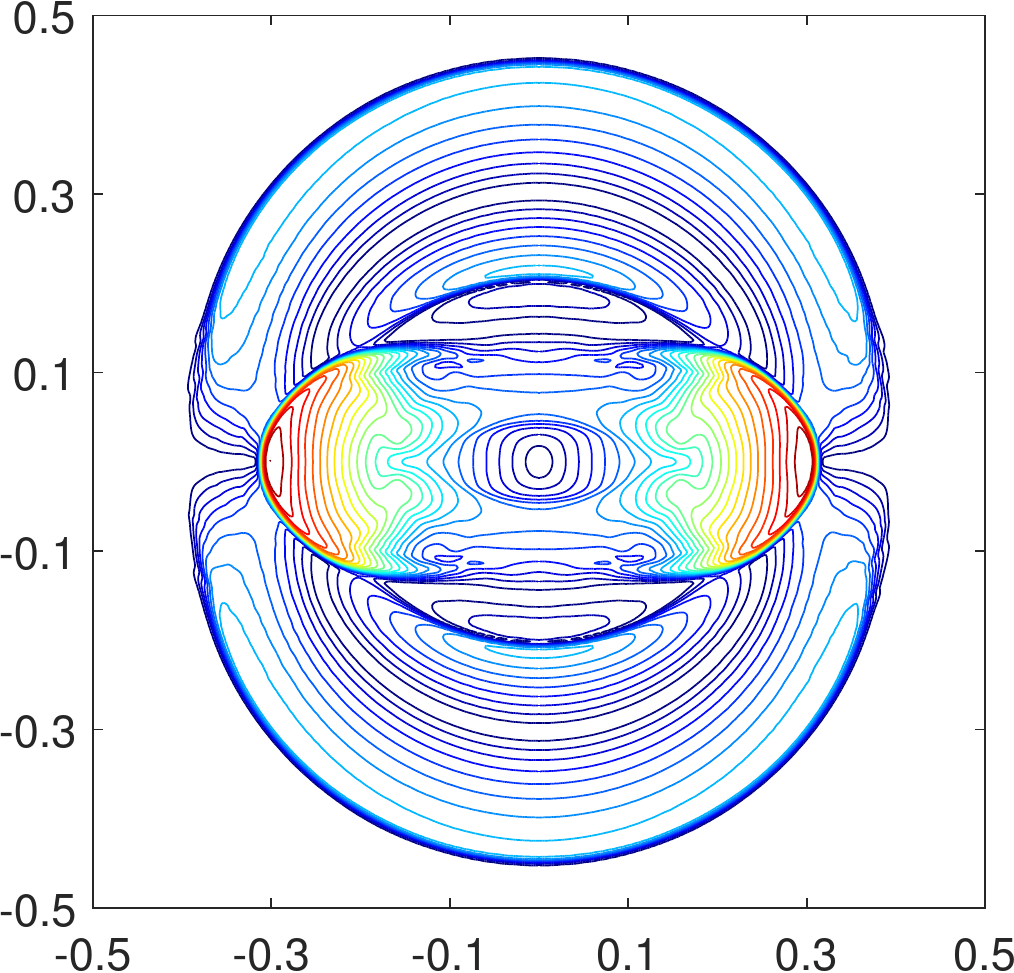}}
	\captionsetup{belowskip=-5pt}
	\caption{The contour plots of $\rho$, $p_m$, $p$, and $\| {\bm v} \|$ (from let to right) for the blast problem at $t=0.01$.}
	\label{fig:BL}
\end{figure}

\begin{example}[Blast problem]\label{ex:BL} 
	This test simulates a benchmark MHD problem in the domain $[-0.5,0.5]^2$ with outflow boundary conditions.  
	The setup is similar to that in \cite{BalsaraSpicer1999} except for a fluid mixture with 
	$n_c=2$, $C_{v_1}=2.42$, $C_{v_2}=0.72$, $\Gamma_1=5/3$, and $\Gamma_2=1.4$.  Initially, 
	the fluid is stationary, with $(\rho,p,Y_1,Y_2)=(1,1000,1,0)$ in the explosion region ($x^2+y^2 \le 0.01$) and  $(1,0.1,0,1)$ in the ambient region   
	($x^2+y^2 > 0.01$). 
	The magnetic field $\bf B$ is initialized as $(100/\sqrt{4\pi},0,0)$. 
	Due to the 
	large jump in $p$ and the strong magnetic field, negative numerical $p$ can be easily produced and often cause failure of the numerical simulations. 
	\cref{fig:BL} presents the contour plots of the density $\rho$, the magnetic pressure $p_m =\frac12 \| {\bf B} \|^2$, the thermal pressure $p$, and the velocity magnitude $\| {\bm v} \|$ computed by the proposed bound-preserving discontinuous Galerkin method with $400 \times 400$ uniform cells. We observe the flow structures are well captured, and our method is highly robust and always preserves the bound principles \cref{eq:G-MMHD} in the whole simulation. 
\end{example}

\begin{figure}[htb]
	\centering
	{\includegraphics[width=0.28\textwidth]{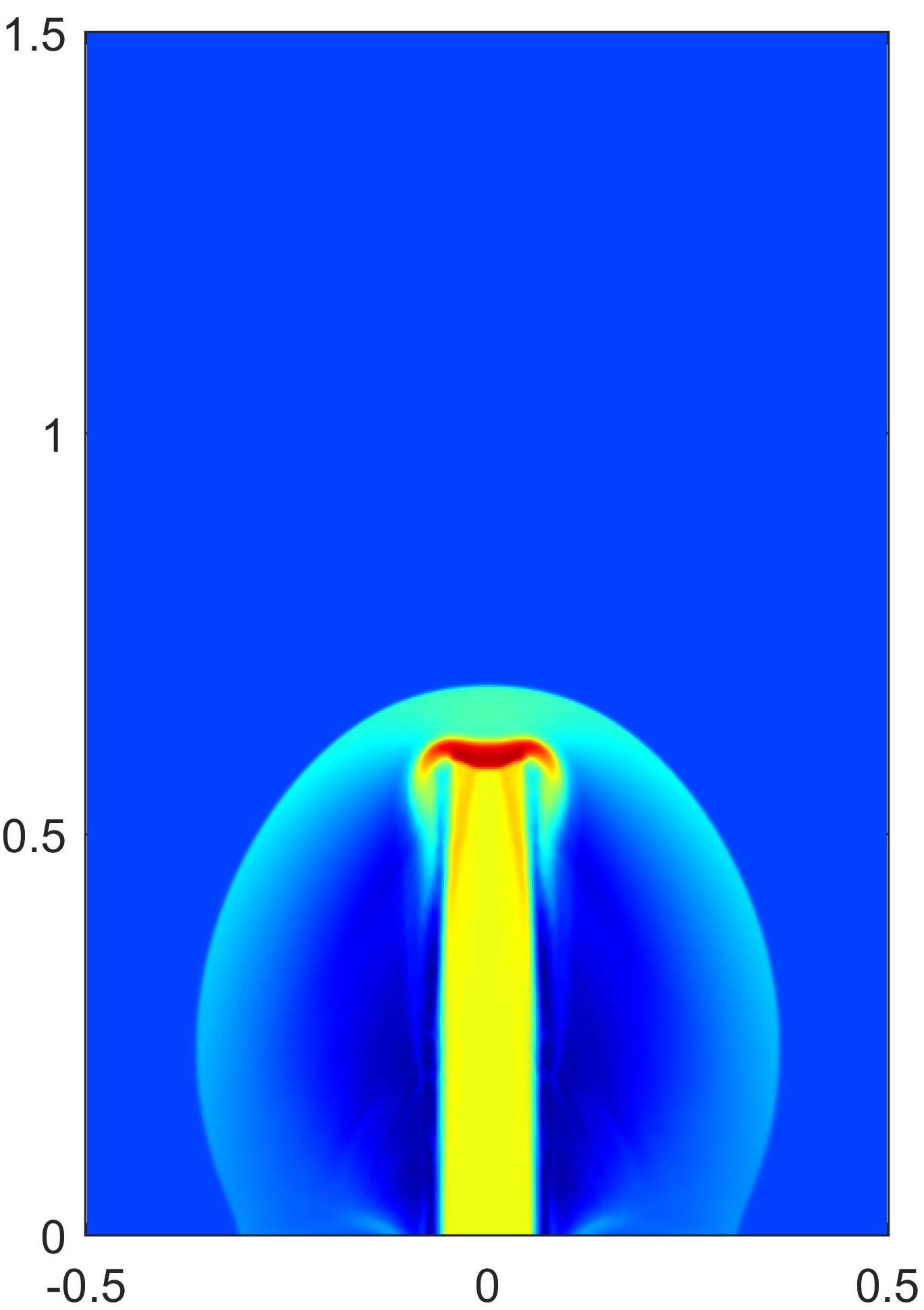}}~~~~
	{\includegraphics[width=0.28\textwidth]{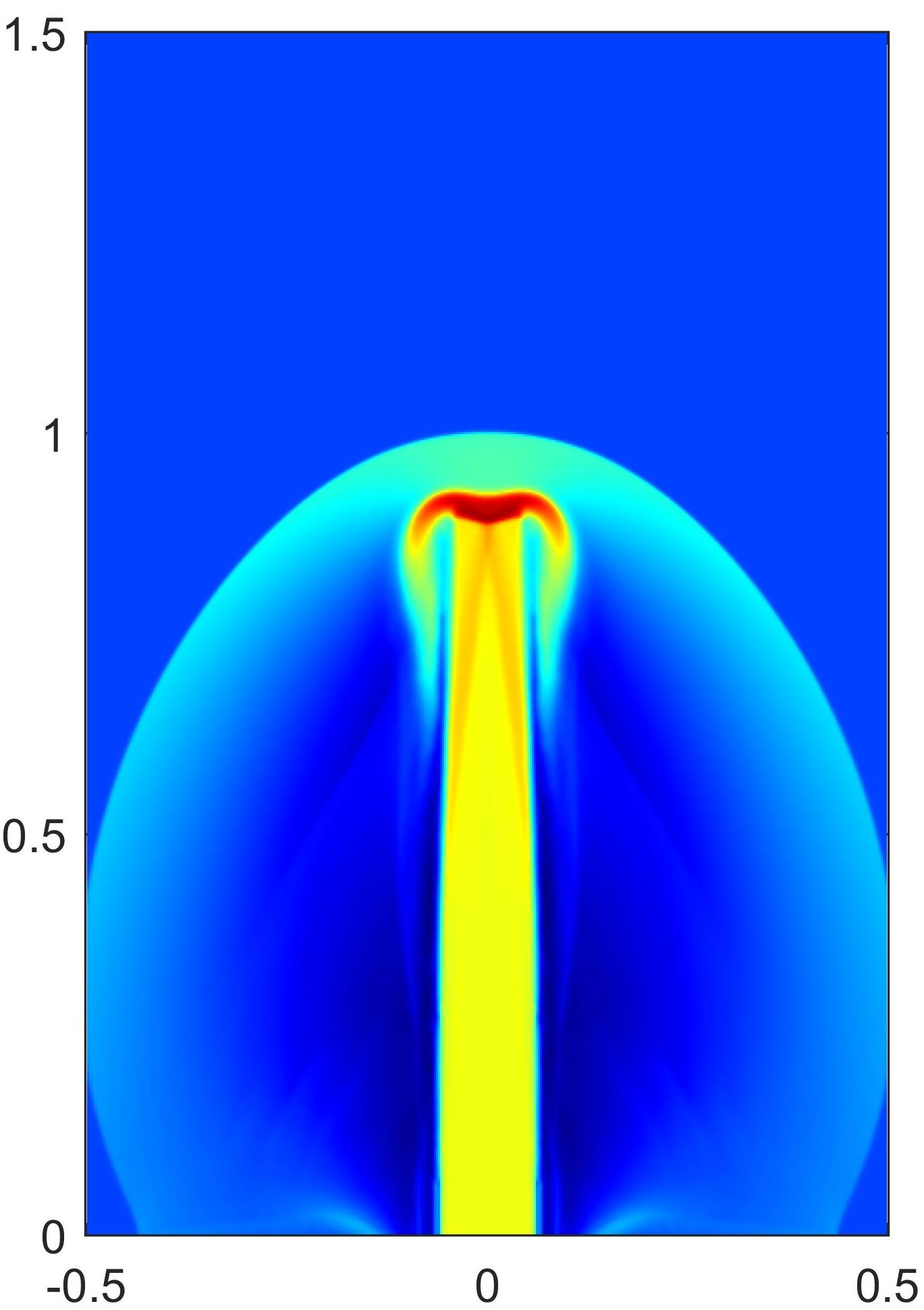}}~~~~
	{\includegraphics[width=0.28\textwidth]{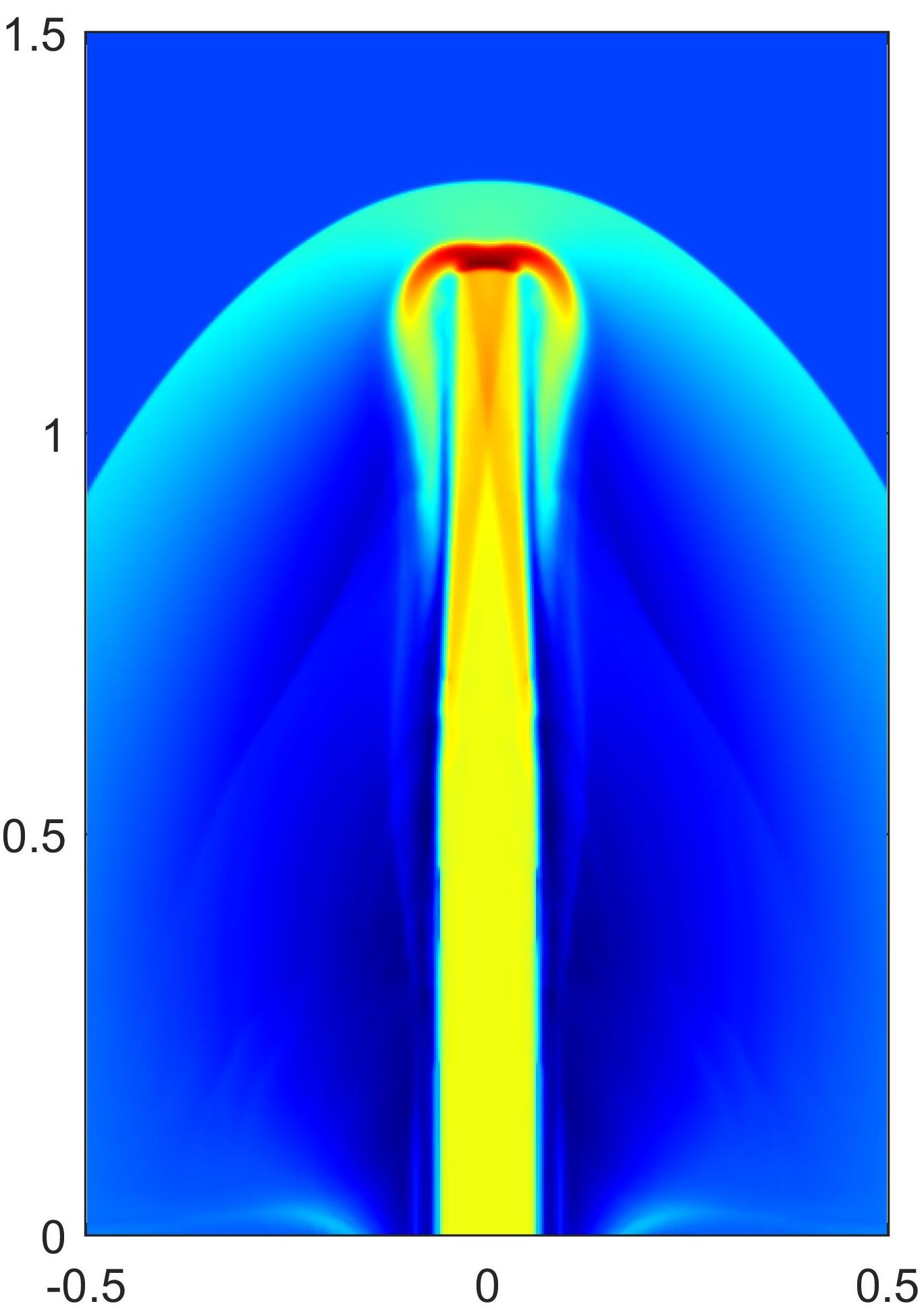}}
	\captionsetup{belowskip=-5pt}
	\caption{\small The plots of $\log(\rho)$ for the jet problem. From left to right: $t=0.001$, $0.0015$ and $0.002$.} 
	\label{fig:jet2}
\end{figure}

\begin{example}[Astrophysical jet] This test simulates a high-speed MHD jet flow  
	in the domain $[-0.5,0.5]\times[0,1.5]$ with $n_c=2$, $C_{v_1}=0.72$, $C_{v_2}=2.42$, $\Gamma_1=1.4$, and $\Gamma_2=5/3$. The domain is initially filled with static fluid 
	with 
	$(\rho,p,Y_1,Y_2)=(0.14,1,0,1)$. The inflow jet condition is fixed on boundary 
	$\{ |x|<0.05, y=0 \}$ with $(\rho,p,Y_1,Y_2)=(1.4,1,1,0)$ and ${\bm v}=(0, 800, 0)$, while the outflow conditions are specified on the other boundaries. 
There is a strong magnetic field $\bf B$ initialized as $(0,\sqrt{4000},0)$, which makes this test more challenging. 
Our simulation is based on 
the proposed bound-preserving method with $200 \times 600$ uniform cells in $[0,0.5]\times[0,1.5]$. 
The numerical results are shown \cref{fig:jet2}. The flow pattern is captured with high resolution and similar to the single-component MHD case reported in \cite{WuShu2018,WuShu2019}. In such an extreme test, our bound-preserving method exhibits good robustness. However, if the proposed scaling limiter is not used to enforce \eqref{PPLimiterCondition}, or if 
the locally divergence-free technique is not employed to ensure  ${\rm div} _{ij}^{{-}} {\bf B}=0$, or if the proposed source term \cref{MMHD-S} is dropped, the resulting 
method even with the WENO limiter is not bound-preserving and would fail quickly due to  nonphysical numerical solutions out of the bounds. This confirms our theoretical analyses and 
the importance of the proposed conditions and techniques. 
\end{example}

\section{Conclusions}
\label{sec:conclusions}
We have systematically proposed a novel and general framework, 
called geometric quasilinearization (GQL), for studying bound-preserving problems with nonlinear constraints. GQL skillfully transfers all nonlinear constraints into linear ones, via properly introducing some free auxiliary variables independent of the system variables. 
We have established the fundamental principle and general theory of GQL, and provided three simple methods for constructing GQL representations. 
The GQL approach equivalently casts the nonlinear bound-preserving problems into preserving the positivity of linear functions, thereby opening up a new effective way for 
bound-preserving study. 
Several examples have been provided to demonstrate the effectiveness and advantages of the GQL approach in addressing nonlinear bound-preserving problems that are highly challenging and could not be easily handled by direct or traditional approaches.  
Besides the examples in this paper, recently the GQL approach also achieved successes in finding (high-order) bound-preserving schemes for several complicated PDE systems in   \cite{Wu2017a,WuShu2018,WuTangM3AS,wu2021minimum,WuShu2019,WuShu2020NumMath}.  

As the proposed GQL framework 
is not restricted to the specific forms of the PDEs, it  
applies to general time-dependent PDE systems 
that possess convex invariant regions with nonlinear constraints. Moreover, 
it can be used in conjunction with 
the well-developed limiters in \cite{zhang2010,zhang2010b,Xu2014,Hu2013} to design high-order bound-preserving schemes. 
It can be expected the GQL approach will be useful for addressing more challenging bound-preserving problems for a variety of PDEs in the future.

\appendix
\section{Proof of Theorem \ref{thm:I=1}}\label{sec:proofSpeicalCase} 	
The proof is divided into two steps.

	{\tt (\romannumeral1) Prove that $G \subseteq G_*$.}  
	For any ${\bf u}_* \in \partial  G$, the hyperplane $({\bf u} - {\bf u}_*) \cdot {\bf n}_{*} =0$ supports the convex region $G$ at ${\bf u}_*$. Thus we have 
	\begin{equation}\label{key632}
		G \subseteq \{{\bf u} \in \mathbb R^N: \left ( {\bf u} - {\bf u}_*,  {\bf n}_{*} \right ) \ge 0~~\forall {\bf u}_* \in \partial  G \}.
	\end{equation}
	If $G$ is closed, then \cref{key632} means $G \subseteq G_*$. Next, we assume $G$ is open and show $G \subseteq G_*$ by contradiction. 
	Assume that 
	\begin{equation}\label{key116}
		\mbox{there exists ${\bf u}_0 \in G$ but ${\bf u}_0 \notin G_*$.}
	\end{equation}
	Then, according to \cref{key632}, there exists ${\bf u}_* \in \partial  G$ such that $( {\bf u}_0 - {\bf u}_*) \cdot {\bf n}_{*}  = 0$. 
	Since $G$ is open, there exists $\delta >0$ such that 
	$\Omega_\delta :=\{ {\bf u} \in \mathbb R^N: \| {\bf u}-{\bf u}_0 \| < \delta  \} \subset G $. We take 
	$ 
		{\bf u}_\delta :={\bf u}_0 -\frac{ \delta  }{2 \| {\bf n}_{*} \|} {\bf n}_{*} \in \Omega_\delta. 	
	$ 
	Then ${\bf u}_\delta \in G$. However, using $( {\bf u}_0 - {\bf u}_*) \cdot {\bf n}_{*}  = 0$ gives 
	\begin{equation*}
		( {\bf u}_\delta - {\bf u}_*) \cdot {\bf n}_{*} 
		= \left ( {\bf u}_0 - {\bf u}_* -\frac{ \delta  }{2 \| {\bf n}_{*} \|} {\bf n}_{*} \right) \cdot 
		{\bf n}_{*} = -\frac{\delta}{2} \| {\bf n}_{*} \| <0,
	\end{equation*}
	which contradicts \cref{key632} and ${\bf u}_\delta \in G$. 
	Thus the assumption \cref{key116} is incorrect, and we have $G \subseteq G_*$.

	{\tt (\romannumeral2) Prove that $G_* \subseteq G$.}  We first show that $G_* \subseteq {\rm cl}(G)$ by contradiction. 
	Assume that 
	\begin{equation}\label{key679}
		\mbox{there exists ${\bf u}_0 \in G_*$ but ${\bf u}_0 \notin {\rm cl}(G)$.} 
	\end{equation}
	According to the theory of convex optimization \cite{boyd2004convex}, 
	the minimum of the convex function $\zeta ({\bf u}):=\| {\bf u} - {\bf u}_0 \|^2$ over the closed convex region ${\rm cl}(G)$ is attained at certain boundary point ${\bf u}_* \in \partial G$. 
	Let $\widehat {\bf u}$ be an arbitrary interior point of $G$. 
	Thanks to the convexity ${\rm cl}(G)$, one has ${\bf u}_\lambda :=\lambda \widehat {\bf u} + (1-\lambda) {\bf u}_* \in {\rm cl}(G)$ for any $\lambda \in [0,1]$. 
	We then know that the quadratic function 
	\begin{equation*}
		\widehat \zeta (\lambda) := \zeta ({\bf u}_\lambda) 
		= \lambda^2 \left \| \widehat {\bf u} - {\bf u}_* \right\|^2  + 2 \lambda \left (  \widehat {\bf u} - {\bf u}_* \right) \cdot \left( {\bf u}_* - {\bf u}_0 \right )  + \| {\bf u}_* - {\bf u}_0 \|^2
	\end{equation*}
	attains its minimum over $[0,1]$ at $\lambda =0$. This implies  $\left (  \widehat {\bf u} - {\bf u}_* \right) \cdot \left( {\bf u}_* - {\bf u}_0 \right ) \ge 0$, for an arbitrary interior point $\widehat {\bf u}$ of $G$. 
	Thus ${\rm int} (G) \subseteq \{ {\bf u}: \left (  {\bf u} - {\bf u}_* \right) \cdot \left( {\bf u}_* - {\bf u}_0 \right ) \ge 0 \} =: H_*^+ $, where $H_*^+$ is a closed halfspace. 
	It follows that 
	$H_*^+$ is a supporting halfspace to $G$, and ${\bf u}_* - {\bf u}_0 $ is an inward-pointing normal vector of $G$ at ${\bf u}_*$.  
	Because $\partial  G$ is smooth, there exists $\mu > 0$ such that ${\bf n}_{*} = \mu ( {\bf u}_* - {\bf u}_0 )$, which implies 
		$(   {\bf u}_0 - {\bf u}_* ) \cdot {\bf n}_{*} 
		= - \mu \left \| {\bf u}_0 - {\bf u}_* \right \|^2 < 0.$ 
	This contradicts the assumption ${\bf u}_0 \in G_*$. Thus the assumption \cref{key679} is incorrect, and we have $G_* \subseteq {\rm cl}(G)$. 
	If $G$ is closed, then we obtain $G_* \subseteq G$. 
	If $G$ is open, then $\partial G \cap G_* = \emptyset$, which along with 
	$G_* \subseteq {\rm cl}(G)$
	yields  $G_* \subseteq G$. 
	
	In summary, we have $G=G_*$, and the proof is completed. 

\section{A simple scaling limiter to enforce \cref{PPLimiterCondition}}\label{sec:limiter} 
The condition \cref{PPLimiterCondition} is not always automatically satisfied by the polynomial vector ${\bf U}_{ij}^n(x,y)$ of the high-order schemes. If this happens, the following limiter is used to modify ${\bf U}_{ij}^n(x,y)$ into $\widetilde {\bf U}_{ij}^n(x,y)$ such that $\widetilde {\bf U}_{ij}^n(x,y)$ satisfies \cref{PPLimiterCondition}. Define $\mathbb Q_{ij} = \{( \widehat x_i^{(\beta)}, y_j^{(q)} ), (  x_i^{(q)}, \widehat y_j^{(\beta)} )~ \forall \beta,q\}$ as the set of all the points involved in \cref{PPLimiterCondition}. Since the limiter is performed separately for each cell, the subscripts $ij$ and superscript $n$ of all quantities are omitted below for convenience. 
First, modify the density as 
$$
\widehat \rho (x,y) = \overline \rho + \theta_1 ( \rho (x,y) - \overline \rho  ), \qquad \theta_1 := ( \overline \rho -\epsilon_1 )/\big( \overline \rho - \min_{ (x,y) \in \mathbb Q_{ij} } \rho (x,y) \big),
$$
where 
$\epsilon_1$ is a small positive number and may be taken as $\min \{10^{-13},\overline \rho\}$. Define $\mathbb S_{k}=\{ (x,y) \in \mathbb Q_{ij}: { \rho Y_k } (x,y) \le 0 \}$.  Then, modify the mass fractions \cite{du2019high} as
$$
\widehat{ \rho Y_k } (x,y) = { \rho Y_k } (x,y) + \theta_2 \left( 
\frac{ \overline{\rho Y_k} }{ \overline \rho } \widehat \rho (x,y) - { \rho Y_k } (x,y) 
\right),\qquad 1\le k \le n_c-1, 
$$
where $\theta_2 = \max_{1\le k \le n_c} \max_{(x,y)\in \mathbb S_{k}} \{ 
\frac{ - { \rho Y_k } (x,y)  }{ \overline{\rho Y_k} \widehat \rho (x,y)/ \overline \rho -  { \rho Y_k } (x,y)  }
\}$ with $\rho Y_{n_c} = \widehat \rho - \sum_{k=1}^{n_c-1} \rho Y_k$. 
Denote $\widehat {\bf U}=( \widehat{ \rho {\bf Y} }, \widehat {\rho}, {\bm m}, {\bf B}, E )^\top$. Finally, modify $\widehat {\bf U}$ to enforce the positivity of 
 $g({\bf U})=E - \frac12 ({ \|{\bm m}\|^2 }/\rho +{ \|{\bf B}\|^2 })$ by  
\begin{equation*}
\widetilde {\bf U}(x,y) = \overline {\bf U} + \theta_3 ( \widehat {\bf U}(x,y) - \overline {\bf U}  ), \qquad \theta_3 :=  ( g( \overline {\bf U} ) - \epsilon_2 )/\big(  g( \overline {\bf U}  ) - \min_{(x,y)\in \mathbb Q_{ij}} g( \widehat {\bf U}(x,y) ) \big),
\end{equation*}
where $\epsilon_2$ is a small positive number and may be taken as $\min \{10^{-13}, g( \overline {\bf U}  )\}$. Note that the pressure function $p({\bf U})$ in \cref{MMHDdef:p} is generally not concave so we use the concave function $g({\bf U})$ instead of $p({\bf U})$. It can be verified that the limited solution $\widetilde {\bf U}(x,y) \in G$ for all 
$(x,y)\in \mathbb Q_{ij}$ and its cell average equals $\overline {\bf U}$. Such type of limiters do not lose the high-order accuracy, as demonstrated in  \cite{zhang2010,zhang2010b,ZHANG2017301}.  
 

\renewcommand\baselinestretch{0.9638}

\bibliographystyle{siamplain}
\bibliography{references_short}

\end{document}